\documentclass[10pt,a4paper]{amsart}

\usepackage{amssymb,amsmath,mathrsfs,mathtools}
\usepackage{bbm,bm,esint,combelow}
\usepackage{times}
\usepackage{enumitem}
\usepackage{tikz}
\usetikzlibrary{arrows.meta,calc,decorations.pathreplacing}
\usepackage{float}
\usepackage[left=3.3cm,right=3.3cm,top=4cm,bottom=3.5cm]{geometry}
\usepackage{dsfont}
\usepackage{color}
\usepackage{refcount}
\usepackage[colorlinks=true,linkcolor=blue!60!black,citecolor=green!40!black,urlcolor=blue!60!black,bookmarksnumbered=true,bookmarksdepth=2]{hyperref}

\definecolor{samp}{RGB}{31,87,160}
\definecolor{hlc}{RGB}{204,102,24}
\colorlet{bandc}{samp!13}
\colorlet{fieldc}{black!72}
\colorlet{fieldthin}{black!38}
\colorlet{gridc}{black!22}
\colorlet{basec}{black!55}
\definecolor{conc}{RGB}{160,32,96}

\hypersetup{pdftitle={Geometric rigidity principles for A-free Young measures},
            pdfauthor={Sergio Scalabrino}}

\theoremstyle{plain}
\newtheorem{theorem}{Theorem}[section]
\newtheorem{lemma}[theorem]{Lemma}
\newtheorem{proposition}[theorem]{Proposition}
\newtheorem{corollary}[theorem]{Corollary}
\newtheorem{definition}[theorem]{Definition}
\theoremstyle{definition}
\newtheorem{remark}[theorem]{Remark}
\newtheorem{remo}[theorem]{Remark}
\newtheorem{rmk}[theorem]{Remark}
\newtheorem{example}[theorem]{Example}
\theoremstyle{plain}
\numberwithin{equation}{section}

\newcommand{\N}{\mathbb{N}}
\newcommand{\R}{\mathbb{R}}

\newcommand{\M}{\mathbb{R}^{N \times n}}
\newcommand{\V}{\mathbb V}
\newcommand{\XX}{\mathbb{X}}
\newcommand{\Omegacl}{\overline{\Omega}}

\newcommand{\WC}{\Lambda_{\mathscr{A}}}
\newcommand{\SG}{\mathbb S_{G}}
\newcommand{\BG}{\mathbb B_{G}}
\newcommand{\B}{\mathbb{B}_G}

\newcommand{\HH}{\mathbb{H}}
\newcommand{\E}{\mathbb E}
\newcommand{\G}{\mathbb G}

\newcommand{\Meas}{\mathscr{M}}
\newcommand{\Lebo}{\mathscr{L}^{1}}
\newcommand{\Leb}{\mathscr{L}^{n}}
\newcommand{\Ha}{\mathcal{H}^{n-1}}
\newcommand{\Hm}{\mathcal{H}^{1}}

\newcommand{\CC}{\operatorname{C}}

\newcommand{\LL}{\operatorname{L}}

\newcommand{\WW}{\operatorname{W}}

\newcommand{\bv}{\operatorname{BV}}

\newcommand{\dual}[2]{\langle #1, #2 \rangle}
\newcommand{\ddual}[2]{\langle\!\langle #1, #2 \rangle\!\rangle}
\newcommand{\norm}[1]{\left\lVert#1\right\rVert}
\newcommand{\dd}{\mathrm{d}}
\newcommand{\dif}{\operatorname{d}\!}
\newcommand{\D}{D}
\newcommand{\di}{\operatorname{div}}

\newcommand{\m}{\operatorname{m}}
\newcommand{\wstar}{\stackrel{*}{\rightharpoonup}}
\newcommand{\toY}{\overset{\mathrm{Y}}{\to}}

\newcommand*{\Lcorner}{%
    \mathchoice
        {\mathrel{\makebox[7pt][c]{\rule{.4pt}{7.5pt}\rule{5pt}{.4pt}}}}%
        {\mathrel{\makebox[7pt][c]{\rule{.4pt}{7.5pt}\rule{5pt}{.4pt}}}}%
        {\mathrel{\makebox[5.5pt][c]{\rule{.4pt}{5.25pt}\rule{3.5pt}{.4pt}}}}%
        {\mathrel{\makebox[4pt][c]{\rule{.4pt}{3.75pt}\rule{2.5pt}{.4pt}}}}%
}
\let\llcorner\Lcorner

\DeclareFontFamily{U}{stix2bb}{}
\DeclareFontShape{U}{stix2bb}{m}{n} {<-> stix2-mathbb}{}

\def\Xint#1{\mathchoice
   {\XXint\displaystyle\textstyle{#1}}%
   {\XXint\textstyle\scriptstyle{#1}}%
   {\XXint\scriptstyle\scriptscriptstyle{#1}}%
   {\XXint\scriptscriptstyle\scriptscriptstyle{#1}}%
   \!\int}
\def\XXint#1#2#3{{\setbox0=\hbox{$#1{#2#3}{\int}$}
     \vcenter{\hbox{$#2#3$}}\kern-.5\wd0}}

\def\dashint{\Xint-}

\newcommand{\e}{\varepsilon}
\newcommand{\eps}{\varepsilon}

\newcommand{\tw}{\tilde{w}}

\let\pphi\phi
\let\phi\varphi

\usepackage{etoolbox}
\makeatletter
\newcommand{\@setaffiliation}{%
  \begingroup\centering\footnotesize
  \def\author##1{}%
  \def\address##1##2{\par##2}%
  \def\email##1##2{\par\texttt{##2}}%
  \addresses\par\endgroup
  \global\let\addresses\@empty}
\apptocmd{\@setauthors}{\@setaffiliation}{}{}
\makeatother

\makeatletter
\def\@secref@aux#1.#2\@nil{#1}
\newcommand{\secref}[1]{%
  \refused{#1}%
  \edef\@secref@tmp{\getrefnumber{#1}}%
  \expandafter\@secref@aux\@secref@tmp.\@nil}
\newcommand{\bibalias}[1]{%
  \if@filesw\immediate\write\@auxout{\string\bibcite{#1}{\the\value{\@listctr}}}\fi
  \@ifundefined{hyper@anchorstart}{}{%
    \Hy@raisedlink{\hyper@anchorstart{cite.#1\@extra@b@citeb}\relax\hyper@anchorend}}%
  \ignorespaces}
\makeatother
\begin{document}
\author{Sergio Scalabrino}
\address{SISSA -- International School for Advanced Studies, Via Bonomea 265,
  34136 Trieste, Italy}
\email{sscalabr@sissa.it}
\title{Geometric rigidity principles for $\mathscr{A}$-free Young measures}
\begin{abstract}
We study generalized Young measures generated by $\mathscr{A}$-free
measures, $\mathscr{A}$ a constant-rank operator, in a compactification of
the value space adapted to an energy. Given an $\mathscr{A}$-quasiconvex
integrand $G$ of linear growth with coercive recession function $G^\infty$,
the unit sphere at infinity is replaced by the indicatrix
$\mathbb{S}_G=\{G^\infty=1\}$, which need be neither convex nor symmetric.
The two compactifications are equivalent, but in the adapted one the
equality cases of the Jensen inequalities characterizing $\mathscr{A}$-free
Young measures become statements of convex geometry about $\mathbb{S}_G$. We
prove the characterization in this setting, with a new streamlined proof of the
inhomogenization step based on a dyadic martingale argument. We then show
that if $\mathscr{A}$-free measures $\mu_j\overset{*}{\rightharpoonup}\mu$
and their $G$-energies converge, then at $|\mu^s|$-almost every point the
concentration-angle measure is supported on the smallest exposed face of
$\overline{\mathrm{co}}(\mathbb{S}_G)$ containing the $G$-polar of $\mu^s$;
in particular, it never charges the set where $G^\infty$ differs from its
convex envelope, which, by the automatic convexity of Kirchheim and
Kristensen, does not meet the wave cone. The concentration on the singular
part of $\mu$ is elementary for every such sequence if and only if each
point of the wave cone on $\mathbb{S}_G$ is an extreme point of
$\overline{\mathrm{co}}(\mathbb{S}_G)$, as for the Euclidean norm underlying
Reshetnyak's continuity theorem; on flat faces of $\mathbb{S}_G$, rigidity
degrades from a point to a face. Finally, we prove a transfer
principle for the convergence of energies of $\mathrm{BV}$ maps under
Dirichlet boundary conditions.
\end{abstract}
\maketitle
\setcounter{tocdepth}{1}
\tableofcontents
\newpage
\section{Introduction}\label{intro section}

\subsection{Oscillation, concentration and generalized Young measures}

A large class of variational problems arising in plasticity, fracture mechanics,
phase transitions and image reconstruction leads to integral functionals
\begin{equation}\label{intro:functional}
\mathcal{F}[v]=\int_{\Omega} F(x,v(x))\,\dd x,
\qquad \mathscr{A}v=0,
\end{equation}
whose integrand grows only \emph{linearly} in the second variable. Two
structural difficulties appear simultaneously. First, $\LL^1$ is not a
dual space: a sequence $(v_j)$ with $\sup_j\norm{v_j}_{\LL^1}<\infty$ is
only weakly$^*$ precompact in the space $\mathscr{M}(\Omega,\V)$ of
bounded Radon measures, so that the natural limit object is a measure
and not a function. Second, and more seriously, the nonlinear quantity
$F(x,v_j)$ need not converge to $F(x,v)$ even along a subsequence,
because of two independent and genuinely different mechanisms:
\emph{oscillation}, in which $v_j$ fluctuates on a vanishing spatial
scale with amplitudes of order one, and \emph{concentration}, in which
$|v_j|$ blows up on sets of vanishing Lebesgue measure. The interplay of
the two is precisely what the theory of generalized Young measures is
designed to record.

Classical Young measures, introduced by L.~C.~Young in his theory of
generalized curves \cite{young} and later brought into the theory of
partial differential equations by Tartar in connection with
compensated compactness \cite{tartar,murat}, keep exact track of
oscillations, but they are blind to concentrations. The remedy proposed by DiPerna and Majda
\cite{dipernamajda} in the context of the incompressible Euler equations
is to \emph{compactify the value space}. One selects a complete separable
subring $\mathscr{R}\subset\CC(\V)$, forms the associated compactification
$\beta_{\mathscr{R}}\V$, and represents the limit of the sequence of
elementary Young measures by a positive Radon measure on
$\Omegacl\times\beta_{\mathscr{R}}\V$. The set of admissible integrands,
and hence the amount of information retained in the limit, is dictated
entirely by the choice of $\mathscr{R}$. In the calculus of variations with linear growth one
choice has become canonical, going back to Alibert and Bouchitt\'e
\cite{alibertbouchitte} and taking its definitive form in the work of
Kristensen and Rindler \cite{KristensenRindler10}: the \emph{sphere}
(or radial) compactification
$$
\overline{B}=B\cup\mathbb{S}_{\V},
\qquad
\iota(z)=\frac{z}{1+|z|},
$$
where $B$ and $\mathbb{S}_{\V}$ are the open unit ball and the unit sphere
of $\V$. It is dual to the space of continuous integrands of linear growth
admitting a regular recession function. In this representation a
generalized Young measure is a triple
$\nu=\big((\nu_x)_{x\in\Omega},\lambda,(\nu^\infty_x)_{x\in\Omegacl}\big)$
consisting of an \emph{oscillation} measure $\nu_x\in\mathscr{M}_1^+(\V)$,
a \emph{concentration} measure $\lambda\in\mathscr{M}^+(\Omegacl)$ and a
family of \emph{concentration-angle} measures
$\nu^\infty_x\in\mathscr{M}_1^+(\mathbb{S}_{\V})$ living on the sphere at
infinity.

The central question of the theory is the \emph{characterization
problem}: which triples arise as limits of elementary Young measures
generated by sequences that respect a given differential constraint? For
gradients in the superlinear regime this is the celebrated theorem of
Kinderlehrer and Pedregal \cite{kinderlehrerpedregal,
kinderlehrerpedregal2}, which places gradient Young measures in duality
with quasiconvex integrands of suitable growth through a Hahn--Banach separation argument.
The linear-growth counterpart, for $\WW^{1,1}$- and $\bv$-derivatives,
is due to Kristensen and Rindler \cite{KristensenRindler10} and was generalized to arbitrary separable compactifications by Seneci \cite{Seneci}; the version
for symmetric gradients and $\mathrm{BD}$ is due to De Philippis and
Rindler \cite{dephilippisrindler-bd}; the $\mathscr{A}$-quasiconvex
setting was developed by Fonseca and M\"uller \cite{FonsecaMuller99} in the case of classical Young measures and,
in the generalized Young measure framework, by Ba\'{\i}a, Matias and
Santos \cite{baiamatiassantos} and Ba\'{\i}a, Chermisi, Matias and Santos
\cite{baiachermisimatiassantos}; finally, for $\mathscr{A}$-free measures
with $\mathscr{A}$ of arbitrary order, complete characterizations were
obtained by Arroyo-Rabasa \cite{ArroyoRabasa21} and by Kristensen and
Rai\cb{t}\u{a} \cite{KristensenRaita22}, whose framework is the one we
follow most closely. Two results have been decisive in making the singular
part of the theory tractable: Alberti's rank-one theorem \cite{Alberti} and
its $\mathscr{A}$-free generalization by De Philippis and Rindler
\cite{DePhilippisRindler16}, which force the polar of the singular part of
an $\mathscr{A}$-free measure to lie in the wave cone $\WC$, and the authomatic convexity result of Kirchheim
and Kristensen \cite{KirchheimKristensen16}. The
existence of potentials $\mathscr{B}$ for constant-rank operators, proved
by Rai\cb{t}\u{a} \cite{raita potentials}, allows a unified exposition.

\subsection{A compactification adapted to the energy}

All of the above, except Seneci \cite{Seneci}, is carried out in the radial compactification, that is,
with respect to the Euclidean norm on $\V$. There is, however, nothing
canonical about this choice, and DiPerna and Majda's original formulation
makes this explicit. The purpose of the present work is to develop the
theory of generalized $\mathscr{A}$-free Young measures in a
compactification \emph{adapted to the energy density under
consideration}, and to show that several statements which look
accidental in the radial picture become transparent --- and strictly
sharper --- in the adapted one.

Concretely, we fix an $\mathscr{A}$-quasiconvex integrand
$G\colon\V\to\R$ of linear growth whose recession function
$$
G^\infty(z)=\lim_{t\to\infty}\frac{G(tz)}{t}
$$
exists locally uniformly in $z$, and we assume that $G^\infty$ is
coercive, $c|z|\le G^\infty(z)\le C|z|$; equivalently, the open sublevel
set
$$
\BG:=\{G^\infty<1\}
$$
is bounded and contains a neighbourhood of the origin (the precise
standing assumptions \textup{(A1)}--\textup{(A3)} are listed in
Section~\ref{notation section}). Since $G^\infty$ is positively
$1$-homogeneous, these assumptions say exactly that $G^\infty$ is a
\emph{gauge} on $\V$: it is positive away from the origin,
$1$-homogeneous, and its indicatrix
$$
\SG:=\{G^\infty=1\}=\partial\BG
$$
is a compact set meeting every ray from $0$ exactly once. What $G^\infty$
need \emph{not} be is convex, or symmetric; by the results of Kirchheim
and Kristensen \cite{KirchheimKristensen16} a positively $1$-homogeneous $\WC$-convex integrand is
convex at wave-cone directions, but it is not convex in general
\cite{Muller92}, and $\SG$
may well have concave portions, flat faces and corners.
Figure~\ref{support sets for nu_x infty} is meant to be read with this in
mind. The non-convexity is, however, confined: the \emph{dent} $N_G$, that
is, the set of points of $\SG$ at which $G^\infty$ differs from its convex
envelope, does not meet the wave cone (Definition~\ref{def:dent} and
Lemma~\ref{lem:dent}). This fact will be decisive below.

In Section~\ref{compactification section} we show that a continuous
integrand $\Phi\colon\V\to\R$ has linear growth and a regular recession
function --- we write $\Phi\in\HH$ --- precisely when the transformed
integrand
$$
T\Phi(\hat z)=\big(1-G^\infty(\hat z)\big)\,
\Phi\!\left(\frac{\hat z}{1-G^\infty(\hat z)}\right),
\qquad \hat z\in\BG,
$$
extends continuously to $\overline{\BG}=\BG\cup\SG$, with boundary values
$\Phi^\infty$; here $\hat z\mapsto\hat z/(1-G^\infty(\hat z))$ is the
inverse of the homeomorphism $\iota_G(z):=z/(1+G^\infty(z))$ of $\V$ onto
$\BG$. Thus $T\colon\HH\to\CC(\overline{\BG})$ is an isometric
isomorphism, the Riesz representation theorem identifies $\HH^*$ with
$\mathscr{M}(\overline{\BG})$, and the whole duality machinery of the
theory goes through verbatim with the Euclidean unit ball $B$ replaced by
$\BG$ and the unit sphere $\mathbb{S}_{\V}$ replaced by $\SG$. For
$G^\infty=|\cdot|$ one recovers the usual radial compactification,
$\iota_G=\iota$. In Section~\secref{rmk:transfer} the same objects are
reached from the opposite direction: the triple representing a Young
measure in the $G$-normalisation is computed from its spherical triple
of \cite{KristensenRindler10} by explicit transfer formulas
(Lemma~\ref{rmk:transfer}), which gives a second proof, by a different
route, of the representation of Young measures in the $G$-compactification
(Lemma~\ref{YM measure theoretic definition}, proved in the Appendix by
disintegration).

It should be said plainly that nothing is gained at the abstract level:
$\overline{\BG}$ and the closed Euclidean ball are equivalent
compactifications of $\V$ (Remark~\ref{rmk:compact}), the two generate the
same admissible class $\HH$, and the two theories are isomorphic. What
changes is the \emph{bookkeeping}, and the bookkeeping is where the
geometry lives. Three consequences are worth isolating.

\begin{enumerate}[label=(\alph*)]
\item The growth condition, the first-moment condition on the
oscillation measure and the normalization
$\int_{\SG}G^\infty\dif\nu^\infty_x=1$ (see \eqref{eq:unitmass}) are all
expressed through the same gauge, so that the energy $G$ enters the
definition of the Young measure rather than being tested against it
afterwards.

\item The concentration-angle measures are probability measures on the
\emph{indicatrix of the energy}, $\SG=\{G^\infty=1\}$. Their barycentres
$\overline{\nu}^\infty_x$ therefore carry intrinsic information: they
always lie in $\overline{\mathrm{co}}(\SG)$, and the rigidity results of
Sections~\ref{convergence of energy section}--\ref{G strict rigidity principle}
turn on \emph{where} in that convex body they happen to sit.

\item Equality cases in the Jensen inequalities that characterize
$\mathscr{A}$-free Young measures become statements of convex geometry
about $\SG$, rather than statements about an auxiliary sphere with no
relation to the functional. This is an important point of the paper, and we expand
on it in \S\ref{intro:energy} below.
\end{enumerate}

The nearest precedent we are aware of is Delladio's notion of
$\Psi$-strict convergence of measures \cite{delladio}, where a convex,
positively $1$-homogeneous $\Psi$ replaces the total variation in
Reshetnyak's theorems, and lower semicontinuity and continuity are
investigated relative to $\Psi$. Our framework differs in two respects:
$G^\infty$ is only assumed $\mathscr{A}$-quasiconvex, hence possibly
non-convex, and the analysis is carried out at the level of Young
measures, which allows us to describe not only when rigidity holds but
also, and with the same geometric criterion, how it fails. A different kind of rigidity, imposed by the differential constraint alone
rather than by the convergence of energies, is the subject of
\cite{ArroyoRabasaAussedat26}, where the rigidity of
\cite{DePhilippisRindler16} is quantified through a spreading inequality
which prevents the concentration angles of $\mathscr{A}$-free measures from
clustering near directions outside $\WC$.

\subsection{Organization and main results}

Section~\ref{notation section} fixes the notation and the standing
assumptions \textup{(A1)}--\textup{(A3)} on $\mathscr{A}$ and $G$.
Section~\ref{compactification section} constructs the $G$-compactification:
the spaces $\HH$ and $\E_G(\Omega,\V)$ of admissible integrands, the
isometries $T$ onto spaces of continuous functions, the identification of
the dual spaces with spaces of measures, the Kantorovich metric on the
positive cones (which makes diagonal extraction arguments legitimate, as
in \cite{KristensenRindler10}), and finally generalized Young measures,
together with their representation by triples
$((\nu_x),\lambda,(\nu^\infty_x))$ satisfying the moment condition
$\int_\Omega\langle G^\infty,\nu_x\rangle\dd x<\infty$
(Lemma~\ref{YM measure theoretic definition}).

Section~\ref{approximation qc function section} collects the properties of
$\mathscr{A}$-quasiconvex integrands of linear growth that are used
throughout, and contains the approximation lemma on which all duality
arguments rest: each $\mathscr{A}$-quasiconvex integrand $f$ of linear
growth is the pointwise decreasing limit of the equi-Lipschitz family
$$
g_\delta=\max\big\{f,\,f^\infty+\delta G^\infty-\tfrac1\delta\big\},
\qquad \delta\searrow0,
$$
which coincides with $f^\infty+\delta G^\infty-\tfrac1\delta$ outside a
sublevel set of $G^\infty$ and satisfies $g_\delta^\infty\searrow f^\infty$.
This is the $G$-analogue of \cite[Lemma~4.3]{KristensenRaita22}, the
$\mathscr{A}$-free version of the corresponding result for gradients
\cite[Lemma~6.3]{KirchheimKristensen16}; see also
\cite[Lemma~5]{KristensenRindler10}, which goes back to
\cite[Lemma~2.3]{alibertbouchitte}; it also shows that the
Jensen inequalities below need only be tested on integrands which coincide
with their recession function outside a ball
(Corollary~\ref{special integrands corollary}).

Section~\ref{hahn banach section} establishes the homogeneous
characterization: a pair
$(\nu^0,\nu^\infty)\in\mathscr{M}_1^+(\V)\times\mathscr{M}^+(\SG)$ belongs
to the set $\mathbf{Y}^{\mathscr{A}}_z$ of homogeneous $\mathscr{A}$-free
Young measures with barycentre $z$ if and only if
$\overline{\nu^0}+\overline{\nu^\infty}=z$ and the Jensen inequality
$$
\int_\V f\dif\nu^0+\int_{\SG}f^\infty\dif\nu^\infty\ \geq\ f(z)
$$
holds for every $\mathscr{A}$-quasiconvex $f$ of linear growth. The proof
follows the Kinderlehrer--Pedregal scheme \cite{kinderlehrerpedregal}, with the necessary modifications for the measure theoretic framework from \cite{KristensenRaita22}:
weak$^*$ closedness and convexity of $\mathbf{Y}^{\mathscr{A}}_z$, followed
by Hahn--Banach separation.

Sections~\secref{sufficiency proposition}--\secref{lemma: CD_f as sup of differentiability point}
contain the characterization of $\mathscr{A}$-free Young measures in the
$G$-compactification. Proposition~\ref{sufficiency proposition} shows that
if $v_j\wstar v$ in $\mathscr{M}(\Omega,\V)$ with
$\mathscr{A}v_j\to\mathscr{A}v$ in $\WW^{-k,p}_{\mathrm{loc}}$ for some
$p\in(1,\frac{n}{n-1})$, and $v_j\toY\nu$, then the Jensen inequalities
hold at $\Leb$-a.e.\ point and the concentration-angle barycentres obey
$$
b_G\,\overline{\nu}^\infty_x
=p_G\in\WC\cap\SG
\quad |v^s|_G\text{-a.e.},
\qquad
\overline{\nu}^\infty_x=0
\quad \lambda^*\text{-a.e.},
$$
where $p_G$ is the $G$-polar of $v^s$ and $\lambda^s=b_G|v^s|_G+\lambda^*$
is the decomposition of $\lambda^s$ with respect to the $G$-variation of
$v^s$; moreover $b_G\ge1$, so that $\overline{\nu}^\infty_x$ lies on the
segment from $0$ to $p_G(x)$ and reaches $\SG$ exactly where $b_G=1$.
The argument combines localization, mollification to area-strict
convergence, Murat's inequality \cite{KristensenRaita22,murat} with Rai\cb{t}\u{a}'s
potentials \cite{raita potentials}, the automatic convexity of Kirchheim-Kristensen \cite{KirchheimKristensen16} and the De Philippis--Rindler theorem
\cite{DePhilippisRindler16}. Section~\ref{inhomogenization section} proves
the converse by a process that, following \cite{KristensenRaita22}, is called inhomogenization, separately for the singular and the
absolutely continuous parts; in place of the Lusin-type approximation of
\cite{KristensenRaita22} we use a dyadic martingale
argument, and for the singular part a Jensen inequality for averages of
concentration-angle measures (Lemma~\ref{integral jensen inequality lemma}),
which rests on the automatic convexity of \cite{KirchheimKristensen16}.
Section~\secref{lemma: CD_f as sup of differentiability point}
characterizes the concentration part alone through the convex deficiency
integrand $\mathrm{CD}_f$, as in \cite{KristensenRaita22}.

Sections~\ref{convergence of energy section}--\ref{G strict rigidity principle}
contain what we regard as the main new results; they are discussed at
length in \S\ref{intro:energy} below. Section~\ref{convergence of energy section}
sets up the convergence of energy principle and describes the three
possible positions of the concentration-angle barycentres;
Section~\secref{lem:dent} revisits the $G$-compactification starting from
the spherical one, introduces the contact set, the dent and the faces of
the indicatrix, and shows that the dent avoids the wave cone;
Section~\ref{G strict rigidity principle} proves the $G$-strict rigidity
principle and discusses its sharpness. Section~\ref{transfer principle section} is devoted to a transfer
principle for the convergence of energies under Dirichlet boundary
conditions. The Appendix
(Section~\ref{appendix section}) contains the proof of
Lemma~\ref{YM measure theoretic definition} and the facts on the
Kantorovich norm used throughout.

\subsection{Convergence of energies: from Reshetnyak's theorem to a
geometric rigidity principle}\label{intro:energy}

The oldest result of the kind we are after is Reshetnyak's continuity
theorem \cite{reshetnyak} (see also \cite[Thm.~2.39]{AFP},
\cite{goffmanserrin,spector}): if $\mu_j\wstar\mu$ in
$\mathscr{M}(\Omega,\V)$ \emph{and} $|\mu_j|(\Omega)\to|\mu|(\Omega)$,
then
$$
\int_{\Omega} f\!\left(\frac{\dif\mu_j}{\dif|\mu_j|}\right)\dif|\mu_j|
\ \longrightarrow\
\int_{\Omega} f\!\left(\frac{\dif\mu}{\dif|\mu|}\right)\dif|\mu|
$$
for every continuous, positively $1$-homogeneous $f$. In Young measure
language the hypothesis of strict convergence is exactly a
\emph{convergence of energies} statement for the particular energy
$G=|\cdot|$, and the conclusion is that, \emph{on positively
$1$-homogeneous integrands}, the generated Young measure acts as the
elementary one: for every continuous, positively $1$-homogeneous $f$,
$$
\int_\Omega\dual{f}{\nu_x}\dif x+\int_\Omega\dual{f}{\nu^\infty_x}\dif\lambda
=\int_\Omega f\Bigl(\frac{\dif\mu}{\dif|\mu|}\Bigr)\dif|\mu| .
$$
This does not make $\nu$ elementary. On the singular part it is, since
$\lambda^s=|\mu^s|$ and
$\nu^\infty_x=\delta_{\frac{\dif\mu^s}{\dif|\mu^s|}(x)}$: no concentration
angles survive. At Lebesgue points, however, oscillation and diffuse
concentration may survive along the half-line through $\mu^a(x)$, which
such integrands cannot detect; $\nu_x=\delta_{\mu^a(x)}$ and $\lambda^a=0$
hold exactly when the convergence is area-strict.
Example~\ref{ex:reshetnyak} recovers this statement as the
special case $G=G^\infty=|\cdot|$ of our framework, and the proof isolates
the mechanism: convergence of energies forces equality in the Jensen
inequalities of Proposition~\ref{sufficiency proposition}, and the
Euclidean sphere is \emph{strictly convex}, so that a probability measure
on $\mathbb{S}_{\V}$ whose barycentre returns to $\mathbb{S}_{\V}$ must be
a Dirac mass.

Once the mechanism is isolated, the natural question is what survives
when $|\cdot|$ is replaced by a general energy. We therefore consider the
following \emph{energy-convergence principle}. Suppose
\begin{equation}\label{intro:hyp}
\mu_j\wstar\mu \ \text{ in } \mathscr{M}(\Omega,\V),
\qquad
\mathscr{A}\mu_j=0,
\qquad
\mathcal{G}(\mu_j)\longrightarrow\mathcal{G}(\mu),
\end{equation}
where $\mathcal{G}(\mu)=\int_\Omega G(\mu)$ is understood in the
Goffman--Serrin sense \cite{goffmanserrin},
$$
\mathcal{G}(\mu)=\int_\Omega G(\mu^a)\dd x
+\int_\Omega G^\infty\Big(\frac{\dif\mu^s}{\dif|\mu^s|}\Big)\dif|\mu^s| ;
$$
in the terminology of Definition~\ref{def:Gvar}, $\mu_j\to\mu$
\emph{$G$-strictly}. (In Section~\ref{convergence of energy section} the
constraint is relaxed to $\mathscr{A}v_j\to\mathscr{A}v$ in
$\WW^{-k,p}_{\mathrm{loc}}(\Omega,\mathbb{W})$.) What does
\eqref{intro:hyp} say about the Young measures $\nu$ generated by
$(\mu_j)$?

The answer is governed by the facial structure of the convex body
$\overline{\mathrm{co}}(\SG)=\{G^{\infty\ast\ast}\le1\}$, where
$G^{\infty\ast\ast}$ is the convex envelope of $G^\infty$. Combining
\eqref{intro:hyp} with the Jensen inequalities turns those inequalities
into equalities, which yields
$$
\langle G,\nu_x\rangle+\lambda^a(x)\langle G^\infty,\nu^\infty_x\rangle
=G\big(\overline{\nu}_x+\lambda^a(x)\overline{\nu}^\infty_x\big)
\quad\Leb\text{-a.e.},
\qquad
\langle G^\infty,\nu^\infty_x\rangle=G^\infty(\overline{\nu}^\infty_x)
\quad |\mu^s|\text{-a.e.},
$$
and shows that the singular part of the concentration measure is exactly
the $G$-variation of $\mu^s$ (Proposition~\ref{prop:rigidity}(i)). Since
$G^\infty\equiv1$ on $\SG$ and $\nu^\infty_x$ is a probability measure
there, the second identity reads
$$
G^\infty(\overline{\nu}^\infty_x)=1,
\qquad\text{that is}\qquad
\overline{\nu}^\infty_x\in\SG :
$$
the barycentre of a probability measure carried by the indicatrix lands
back on the indicatrix; in fact $\overline{\nu}^\infty_x=p_G(x)$ is the
point where the ray through $\frac{\dif\mu^s}{\dif|\mu^s|}(x)$ meets $\SG$.
A priori, three mutually exclusive situations may occur for
$\theta_x:=\overline{\nu}^\infty_x$
(Figure~\ref{support sets for nu_x infty}):

\begin{enumerate}[label=\emph{Case~\arabic*.}, leftmargin=*]
\item $\theta_x$ is an \emph{extreme point} of
$\overline{\mathrm{co}}(\SG)$ --- a point of strict convexity of
$G^\infty$. Then no nontrivial convex combination can reproduce it, and
necessarily $\nu^\infty_x=\delta_{\theta_x}$. This is exactly Reshetnyak
rigidity, localized at $x$.

\item $\theta_x$ lies on a proper \emph{face} of
$\overline{\mathrm{co}}(\SG)$ without being extreme. Then
$\nu^\infty_x$ is supported in the set $F_{\theta_x}$ of
Definition~\ref{def:face}, the trace on $\SG$ of the smallest exposed face
containing $\theta_x$, and no further restriction
holds: rigidity degrades from a point to a face. Case~1 is
characterized by extremality, whereas $F_{\theta}=\{\theta\}$ holds
precisely at the (strictly stronger) \emph{exposed} points
(Remark~\ref{rmk:sharp}); by Straszewicz's theorem the exposed points are
dense in the extreme points of a compact convex set, see
\cite[\S1.4]{schneider}.

\item $\theta_x$ lies in the interior of
$\overline{\mathrm{co}}(\SG)$, which can happen precisely at the points of
the dent $N_G$, where $G^\infty$ is not convex. There convergence of
energies would impose \emph{no} constraint on the support of
$\nu^\infty_x$.
\end{enumerate}

The third case, however, never occurs. By the De Philippis--Rindler
theorem \cite{DePhilippisRindler16} the polar of $\mu^s$, and hence
$\theta_x$, lies in the wave cone $\WC$, while by the automatic convexity
of Kirchheim and Kristensen \cite{KirchheimKristensen16} the dent does not
meet $\WC$ (Lemma~\ref{lem:dent}). This is the content of our main
result, Proposition~\ref{prop:rigidity}: if $\Omega$ is a Lipschitz domain,
no concentration occurs at $\partial\Omega$, and \eqref{intro:hyp} holds,
then
$$
\theta_x=p_G(x)\in\WC\cap\SG
\qquad\text{and}\qquad
\operatorname{supp}\nu^\infty_x\subseteq F_{\theta_x}\subseteq K_G\cap\SG
\qquad\text{for }|\mu^s|\text{-a.e.\ }x,
$$
where $K_G$ is the contact set of Definition~\ref{def:dent}; the
concentration of $\nu$ on the singular part of $\mu$ never charges the
dent. If, moreover, every $\theta\in\WC\cap\SG$ is an extreme point of
$\overline{\mathrm{co}}(\SG)$ (condition \eqref{eq:star}), then
$\nu^\infty_x=\delta_{p_G(x)}$ and no Euclidean mass is lost on the
singular part of $\mu$. Condition \eqref{eq:star} is also necessary
(Remark~\ref{rmk:sharp}), and the remaining two cases genuinely occur: for
the convex gauge $G^\infty(A)=|Ae_1|+|Ae_2|$ on $\R^{N\times2}$, whose
unit ball has flat faces, Example~\ref{ex:staircase} exhibits
$G$-strictly convergent gradients whose concentration-angle measure is a
nontrivial convex combination of two points of a face, and
Example~\ref{ex:transversal} shows that rigidity fails even on a face
transversal to the wave cone. The convex-geometric vocabulary of extreme,
exposed and supporting-hyperplane structure --- for which we refer to
Schneider \cite{schneider} --- is thus not decorative: it is the exact
language in which the rigidity of \eqref{intro:hyp} is expressed.
All of this concerns the singular part of $\mu$, whose polar lies in $\WC$.
At points of diffuse concentration, where $\lambda^a(x)>0$, the barycentre
need not lie in the wave cone, and there automatic convexity has no
counterpart: Cassese \cite[Cor.~1.6]{Cassese26} has recently constructed a
positively $1$-homogeneous rank-one convex integrand on $\R^{4\times4}$ that
is not even quasiconvex at a matrix of rank two. We discuss the consequences
for diffuse concentration in Remark~\ref{rmk:diffuse}. In the absence of
oscillation they are harmless: by the decomposition of concentrations due to
Gennaioli and Rindler \cite{GennaioliRindler26}, diffuse concentration does
not charge the dent either (Proposition~\ref{prop:diffuse}).

This should be compared with a classical circle of results in $\LL^1$.
Visintin \cite{visintin} proved that if $u_j\rightharpoonup u$ weakly in
$\LL^1$ and $u(x)$ is an extreme point of the closed convex hull of the
asymptotic values of $(u_j)$ at $x$, then $u_j\to u$ strongly. Balder
\cite{balder,balder2} and Valadier \cite{valadier} recast this in Young
measure terms, where it becomes the statement that $\nu_x=\delta_{u(x)}$
as soon as $u(x)=\overline{\nu}_x$ is extreme in
$\overline{\mathrm{co}}(\mathrm{supp}\,\nu_x)$. The principle described
above is precisely this extremality mechanism, transplanted from the
oscillation measure $\nu_x$, where it constrains values, to the
concentration-angle measure $\nu^\infty_x$, where it constrains
directions of blow-up --- and with the abstract convex hull of the
support replaced by the concrete convex body generated by the energy.
Proposition~\ref{C closed set and supp(nu_x^infty) subset C_G cap S_G}
makes the analogy complete in the other direction: if the generating
sequence is constrained to take values in a closed set $C\subset\V$, then
$\mathrm{supp}\,\nu_x\subset C$ and
$\mathrm{supp}\,\nu^\infty_x\subset\widetilde{C}_G\cap\SG$, where
$\widetilde{C}_G=\overline{\iota_G(C)}$; this is the generalized-Young-measure
form of the support constraints underlying Tartar's programme on
differential inclusions \cite{tartar,murat}.

In the superlinear regime the analogous principle is well understood and
takes a different shape. Evans and Gariepy
\cite{evans gariepy some remarks} observed that if $F=G+\lambda|\cdot|^q$
with $G$ quasiconvex, $1<q<\infty$, then weak convergence in
$\WW^{1,q}$ together with $I[u_k]\to I[u]$ forces \emph{strong}
convergence, the proof being one line of lower semicontinuity (Morrey
\cite{morrey}, Acerbi--Fusco \cite{acerbifusco}) followed by the uniform
convexity of $\LL^q$. The same idea, in the guise of \emph{strict} or
\emph{uniformly strict} quasiconvexity, goes back to Ball and Murat
\cite{ballmurat} and is central in Evans's partial regularity theory
\cite{evans86}; see also Ball and Zhang \cite{ballzhang} for the
biting-lemma refinements and Bene\v{s}ov\'a and Kru\v{z}\'{\i}k
\cite{benesovakruzik} for a survey. Remark~\ref{evans gariepy connection}
recalls this argument. The linear-growth situation is genuinely harder,
for the simple reason that $\LL^1$ has no uniform convexity to exploit and
that the correct substitute for strong convergence is not norm
convergence but \emph{area-strict} convergence,
$\langle v_j\rangle(\Omega)\to\langle v\rangle(\Omega)$ with
$\langle z\rangle=\sqrt{1+|z|^2}$ --- the notion under which Reshetnyak
continuity does hold for integrands with regular recession, see
\cite[Thm.~2.39]{AFP}, \cite{KristensenRindler10,rindlerbook} and, for the
$\mathscr{A}$-free approximation theory,
\cite{ArroyoRabasa21,ArroyoRabasaDePhilippisRindler20}.
While in the case of $F$ strictly convex, coupled with no concentration
 occurring at $\partial\Omega$, the convergence of energies upgrades
weak$^*$ convergence to area-strict convergence, for non-convex energies
the oscillation part is not controlled by the energy alone; see the remark
\emph{What is not controlled} following Proposition~\ref{prop:rigidity}.

Finally, Section~\ref{transfer principle section} contains a
\emph{transfer} principle, which is the linear-growth analogue of the
Evans--Gariepy splitting. Let $G$ be quasiconvex with regular recession,
let $F$ be continuous of linear growth with generalized recession function
$F^{\#}$ (see the definition in \cite{fonsecamueller-bv, Muller92}), assume that $F-G$ is quasiconvex, and assume the nondegeneracy
condition $F^{\#}>0$ on $\M\setminus\{0\}$. Let $g\in\bv(\R^n,\R^N)$ with
$|D^sg|(\partial\Omega)=0$, for instance $g\in W^{1,1}(\R^n,\R^N)$. Then, for every
sequence $(u_j)\subset\bv_g(\Omega,\R^N)$, that is, $u_j-g$ belongs to the closure of
$\CC^\infty_c(\Omega,\R^N)$ under strict convergence, so that $u_j=g$ on
$\partial\Omega$ in the sense of rough traces, such that $u_j\wstar u$ in
$\bv(\Omega,\R^N)$, where a priori the limit $u$ need not belong to
$\bv_g(\Omega,\R^N)$, we have
$$
\lim_{j\to\infty}\int_\Omega F(Du_j)=\int_\Omega F(Du)
\qquad\Longrightarrow\qquad
\lim_{j\to\infty}\int_\Omega G(Du_j)=\int_\Omega G(Du),
$$
and moreover the extension of $u-g$ by zero does not charge $\partial\Omega$:
no boundary datum is lost in the limit.
The structure of the argument is instructive: after extending $u_j$ and
$u$ by the datum $g$ to a slightly larger Lipschitz domain
$\Omega'\Supset\Omega$, one finds that $D\tilde u_j$ splits exactly,
whereas $D\tilde u$ acquires an extra finite measure $\mu$ concentrated
on $\partial\Omega$ which measures the loss of boundary data in the
weak$^*$ limit. Lower semicontinuity of $F$ on $\Omega'$, in the signed form
with fixed boundary values of Kristensen and Rindler
\cite{kristensenrindler-relax}, then yields $\int_{\partial\Omega}F(\mu)\leq0$, and
the key step is to identify $\mu$, in terms of Young measures, as the barycenter
of the part of the concentration of $(Du_j)$ that escapes to $\partial\Omega$,
and to show that it is singular with respect to $\Leb$, so that $F$ acts on
it through $F^{\#}$ and the sign condition $F^{\#}>0$ forces $\mu=0$; lower
semicontinuity of $F-G$ then gives the convergence of the $G$-energies. If
$F$ has a regular recession function as well, a direct Young measure argument
removes the assumption on $g$. The assumption that $F-G$ is quasiconvex therefore
plays the role that the uniformly convex perturbation $\lambda|\cdot|^q$
plays in \cite{evans gariepy some remarks}: it makes the energy $F$
strictly stronger than $G$ in a way that survives the passage to the
limit. We remark that Example~\ref{example u limit not in bv_g} again shows the
well-known loss of boundary data to be a real phenomenon and not a
technicality, although under the assumptions of the transfer principle it is
ruled out a posteriori.

\subsection{Boundary concentrations}

The measure $\mu$ just described is the trace of a well-known difficulty:
for functionals of linear growth, Dirichlet conditions are not stable
under weak$^*$ convergence, and the relaxed functional on $\bv$ carries
an extra surface term accounting for the discrepancy between the inner trace
of $u$ and the prescribed datum $g$ --- for Lipschitz $\Omega$
$$
\int_{\partial\Omega}F^{\#}\big((g-u)\otimes\nu_\Omega\big)\dif\Ha ,
$$
which is precisely the energy of $\mu=(g-u)\otimes\nu_\Omega\,\Ha\Lcorner\partial\Omega$.
This is classical for the area functional and for convex integrands, see
Giusti \cite{giusti}, Anzellotti and Giaquinta
\cite{anzellottigiaquinta} and Giaquinta, Modica and Sou\v{c}ek
\cite{giaquintamodicasoucek}, and it was extended to the quasiconvex
setting by Ambrosio and Dal Maso \cite{ambrosiodalmaso}, Fonseca and
M\"uller \cite{fonsecamueller-bv} and, in Young measure form, by
Kristensen and Rindler \cite{kristensenrindler-relax,KristensenRindler10}
and by Rindler \cite{rindler-noalberti}.

\section{Notation and standing assumptions}\label{notation section}

\subsection{General notation}

Throughout, $\Omega\subset\R^n$ is a bounded open set, $\Omegacl$ is its
closure, and $\Omega'\Subset\Omega$ means that $\overline{\Omega'}$ is a
compact subset of $\Omega$. We write $\XX:=(-\frac12,\frac12)^n$ for the
open unit cube; a cube is a set $Q=x_Q+r\XX$ (half-open cubes are used in
Section~\ref{inhomogenization section}) with side length $l(Q)=r$, and
$\dashint_Q g\dd x:=|Q|^{-1}\int_Q g\dd x$, where $|Q|=\Leb(Q)$. On $\R^n$
we use the Euclidean norm $|x|$ and the $\ell^\infty$-norm
$\|x\|:=\max_i|x_i|$. The spaces $\V$ and $\mathbb{W}$ are finite
dimensional real inner product spaces, with scalar product $z\cdot w$ and
norm $|\cdot|$; $B$ and $\mathbb{S}_{\V}=\partial B$ are the open unit ball
and the unit sphere of $\V$. When $\V=\M$ is the space of real $N\times n$
matrices we use the Frobenius norm and $a\otimes b:=ab^{\mathsf T}$; in
this case $\mathbb{S}_{\V}=\mathbb{S}^{N\times n-1}$.

\emph{Measures.} For a locally compact separable metric space $X$,
$\mathscr{M}(X,\V)$ is the space of bounded $\V$-valued Radon measures on
$X$, identified with the dual of $\CC_0(X,\V)$; $\mathscr{M}(X)$,
$\mathscr{M}^+(X)$ and $\mathscr{M}^+_1(X)$ are the bounded signed, the
positive and the probability Radon measures on $X$. For
$\mu\in\mathscr{M}(X,\V)$, $|\mu|$ is its total variation, $\mu\Lcorner A$
its restriction to a Borel set $A$, $F_\#\mu$ its push-forward under a
Borel map $F$, and $\frac{\dif\mu}{\dif|\mu|}$ its polar, so that
$\mu=\frac{\dif\mu}{\dif|\mu|}|\mu|$ and
$\big|\frac{\dif\mu}{\dif|\mu|}\big|=1$ $|\mu|$-a.e.\ \cite[Cor.~1.29]{AFP}.
$\Leb$ is the Lebesgue measure on $\R^n$, and $\mathcal{H}^{s}$ is the
$s$-dimensional Hausdorff measure. The Lebesgue--Radon--Nikod\'ym
decomposition of $\mu\in\mathscr{M}(\Omega,\V)$ with respect to $\Leb$ is
written
$$
\mu=\mu^a\Leb\Lcorner\Omega+\mu^s,
\qquad
\mu^a=\frac{\dif\mu}{\dif\Leb}\in\LL^1(\Omega,\V),
\qquad \mu^s\perp\Leb ;
$$
we also write $\frac{\dif\mu}{\dif x}$ for $\mu^a$, $|\mu|^s=|\mu^s|$, and
$\frac{\mu^s}{|\mu^s|}$ or $\frac{\mu^s}{|\mu|^s}$ for the polar
$\frac{\dif\mu^s}{\dif|\mu^s|}$ of the singular part. More generally, for
$\lambda,\rho\in\mathscr{M}^+(\Omegacl)$ we write
$\lambda=\frac{\lambda}{\rho}\,\rho+\lambda^*$, where
$\frac{\lambda}{\rho}=\frac{\dif\lambda}{\dif\rho}$ and
$\lambda^*\perp\rho$. If $\sigma\in\mathscr{M}^+(\V)$ and
$\int|z|\dif\sigma<\infty$, its barycentre is
$\overline{\sigma}:=\int_\V z\dif\sigma(z)$, and for any integrable $f$ we
write $\langle f,\sigma\rangle=\langle\sigma,f\rangle:=\int f\dif\sigma$.

\emph{Convergence of measures.} We write $\mu_j\wstar\mu$ in
$\mathscr{M}(\Omega,\V)$ for weak$^*$ convergence in duality with
$\CC_0(\Omega,\V)$. The convergence is \emph{strict} if moreover
$|\mu_j|(\Omega)\to|\mu|(\Omega)$, and it is \emph{area-strict} (or
$\langle\cdot\rangle$-strict) if moreover
$\langle\mu_j\rangle(\Omega)\to\langle\mu\rangle(\Omega)$, where
$\langle z\rangle:=\sqrt{1+|z|^2}$ and
$\langle\mu\rangle(A):=\int_A\langle\mu^a\rangle\dd x+|\mu^s|(A)$
\cite{AFP,KristensenRindler10}. The $G$-strict convergence is introduced in
Definition~\ref{def:Gvar}.

\emph{Functions.} $\CC(X)$, $\CC_b(X)$, $\CC_c(X)$, $\CC_0(X)$ and
$\CC^\infty_c(X)$ have the usual meaning, and $\LL^p$, $\WW^{k,p}$,
$\WW^{k,p}_0$, $\bv$ denote the Lebesgue, Sobolev and BV spaces.
$\WW^{-k,p}(\R^n,\mathbb{W})$ is the dual of
$\WW^{k,p'}_0(\R^n,\mathbb{W})$, $\frac1p+\frac1{p'}=1$, and we write
$w_j\to w$ in $\WW^{-k,p}_{\mathrm{loc}}(\Omega,\mathbb{W})$ if
$\eta w_j\to\eta w$ in $\WW^{-k,p}(\R^n,\mathbb{W})$ for every
$\eta\in\CC^\infty_c(\Omega)$. For a metric space $(X,\mathrm{d}_X)$ and
$\Phi\colon X\to\R$, $\mathrm{lip}(\Phi)$ is the Lipschitz constant of
$\Phi$, and $\mathrm{LIP}(X)$ is the space of bounded Lipschitz functions
with the norm
$\|\Phi\|_{\mathrm{LIP}}:=\sup_X|\Phi|+\mathrm{lip}(\Phi)$ (see
Section~\ref{kantorovich norm subsection}). Subsets of $\R^n$ are endowed
with the distance induced by the $\ell^\infty$-norm $\|\cdot\|$, and
subsets of $\V$, in particular $\overline{\BG}$, with the Euclidean
distance. The indicator function of a set $A$ is $\boldsymbol{1}_A$.
Finally, $(\rho_\eps)_{\eps>0}$ is a standard mollifier:
$\rho_\eps(x)=\eps^{-n}\rho(x/\eps)$ for a fixed radially symmetric
$\rho\in\CC^\infty_c(\R^n)$ with $\rho\ge0$, $\operatorname{supp}\rho\subset B_1(0)$
and $\int\rho\dd x=1$ (in Section~\ref{inhomogenization section} a mollifier
supported in cubes is used instead).

\emph{Convex sets.} $\mathrm{co}\,K$ and $\overline{\mathrm{co}}\,K$ are
the convex and the closed convex hull of $K\subseteq\V$, and $h^{\ast\ast}$
is the convex envelope of $h\colon\V\to\R$. For extreme points, exposed
points, faces and supporting hyperplanes of convex bodies we follow
\cite{schneider}.

\subsection{The differential operator}

$\mathscr{A}$ is a homogeneous linear differential operator of order $k$
with constant coefficients from $\V$ to $\mathbb{W}$,
$$
\mathscr{A}=\sum_{|\alpha|=k}A_\alpha\partial^\alpha,
\qquad A_\alpha\in\operatorname{Lin}(\V,\mathbb{W}),
$$
acting on $\V$-valued distributions; $v$ is \emph{$\mathscr{A}$-free} if
$\mathscr{A}v=0$ in the sense of distributions. Its principal symbol is
$\mathscr{A}(\xi):=\sum_{|\alpha|=k}\xi^\alpha A_\alpha$, $\xi\in\R^n$, and
its \emph{wave cone} is
$$
\WC:=\bigcup_{\xi\in\R^n\setminus\{0\}}\ker\mathscr{A}(\xi),
$$
a closed and balanced cone, which we also denote by $\Lambda$ when no
confusion can arise. Throughout the paper $\mathscr{A}$ has
\emph{constant rank}, that is, $\operatorname{rank}\mathscr{A}(\xi)$ does
not depend on $\xi\in\R^n\setminus\{0\}$. For $\mathscr{A}=\operatorname{curl}$
on $\V=\M$, $\mathscr{A}$-free fields are (locally) gradients and
$\WC=\{a\otimes\xi:a\in\R^N,\ \xi\in\R^n\}$ is the rank-one cone.

By \cite{raita potentials}, constant rank is equivalent to the existence of
a \emph{potential}: a homogeneous linear differential operator
$\mathscr{B}$ of some order $l$ with constant coefficients, from a finite
dimensional space $\mathbb{U}$ to $\V$, such that
$\operatorname{im}\mathscr{B}(\xi)=\ker\mathscr{A}(\xi)$ for all
$\xi\neq0$; in particular $\mathscr{A}\mathscr{B}=0$. We fix such
$\mathscr{B}$ and $l$ once and for all. To keep the notation light, the
spaces of $\mathbb{U}$-valued maps on which $\mathscr{B}$ acts are written
$\CC^\infty_c(U,\V)$ and $\WW^{l-1,1}(U,\V)$, $U\subseteq\R^n$ open: the
space $\mathbb{U}$ plays no role in what follows.

A locally bounded Borel function $f\colon\V\to\R$ is
\emph{$\mathscr{A}$-quasiconvex} if
$$
f(z)\le\int_\XX f\big(z+\mathscr{B}\pphi(x)\big)\dd x
\qquad\text{for all } z\in\V \text{ and } \pphi\in\CC^\infty_c(\XX,\V) ;
$$
by \cite{raita potentials} this agrees with the definition of Fonseca and
M\"uller \cite{FonsecaMuller99}, which uses $\XX$-periodic
$\mathscr{A}$-free fields with zero mean. For $\mathscr{A}=\operatorname{curl}$
it is quasiconvexity in the sense of Morrey. The
\emph{$\mathscr{A}$-quasiconvex envelope} of a continuous $\Phi\colon\V\to\R$
is
$$
\Phi^{\mathrm{qc}}(z):=\inf\Big\{\int_\XX\Phi\big(z+\mathscr{B}\pphi(x)\big)\dd x
:\ \pphi\in\CC^\infty_c(\XX,\V)\Big\},\qquad z\in\V.
$$
A function $f\colon\V\to\R$ is \emph{$\WC$-convex} if $t\mapsto f(z+tw)$ is
convex for all $z\in\V$ and $w\in\WC$; $\mathscr{A}$-quasiconvex functions
are $\WC$-convex \cite[Prop.~3.4]{FonsecaMuller99}.

\subsection{Recession functions and standing assumptions}

For $f\colon\V\to\R$ of linear growth, $|f(z)|\le c_f(1+|z|)$, the
\emph{upper recession function}
$$
f^\infty(z):=\limsup_{\substack{z'\to z\\ t\to+\infty}}\frac{f(tz')}{t},
\qquad z\in\V,
$$
is real-valued and positively $1$-homogeneous; in the transfer principle of
Section~\ref{transfer principle section}, where the integrand $F$ need not
have a regular recession function, it is denoted by $F^{\#}$. A
continuous $f$ of linear growth has a \emph{regular} recession function if
the limit $\lim_{z'\to z,\,t\to+\infty}t^{-1}f(tz')$ exists for every
$z\in\V$; equivalently, $t^{-1}f(t\,\cdot\,)\to f^\infty$ locally uniformly
on $\V$ as $t\to\infty$. In this case $f^\infty$ is continuous.

The following assumptions are in force throughout the paper; they are
recalled at the beginning of Section~\secref{lem:dent}.

\begin{itemize}
  \item[\textbf{(A1)}] $\WC$ spans $\V$.
  \item[\textbf{(A2)}] $G^{\infty}\colon \V \to [0,\infty)$ is positively
    $1$-homogeneous, $\mathscr{A}$-quasiconvex, and coercive: there are
    $0 < c \le C < \infty$ with $c|z| \le G^{\infty}(z) \le C|z|$ on $\V$.
  \item[\textbf{(A3)}] $G\colon \V\to\mathbb{R}$ is continuous,
    $\mathscr{A}$-quasiconvex, of linear growth, and admits $G^{\infty}$ as
    \emph{regular} recession function:
    $$
      G^{\infty}(z)
      = \lim_{\substack{z'\to z\\ t\to\infty}}\frac{G(t z')}{t}
      \qquad\text{for every } z\in\V .
    $$
\end{itemize}

By \textup{(A2)}, $G^\infty$ is a \emph{gauge}: it vanishes only at the
origin and is positively $1$-homogeneous, but it need be neither convex
nor symmetric. Its open unit sublevel set $\BG:=\{G^\infty<1\}$ and its
indicatrix $\SG:=\{G^\infty=1\}$ are studied in
Section~\ref{compactification section}. By \textup{(A2)} a function $\Phi$
on $\V$ (or on $\Omegacl\times\V$) has linear growth if and only if
$\sup_z|\Phi(z)|/(1+G^\infty(z))<\infty$, and this is the form in which
linear growth will be quantified. By \textup{(A1)}, $\WC$-convex functions
of linear growth are Lipschitz (Lemma~\ref{qc properties lemma}); in
particular $G$ and $G^\infty$ are Lipschitz.

\emph{Functions of measures.} For $F\colon\V\to\R$ continuous of linear
growth and $\mu\in\mathscr{M}(\Omega,\V)$ we use the Goffman--Serrin
convention \cite{goffmanserrin}
$$
\int_A F(\mu):=\int_A F(\mu^a)\dd x
+\int_A F^\infty\Big(\frac{\dif\mu^s}{\dif|\mu^s|}\Big)\dif|\mu^s|,
\qquad A\subseteq\Omega\ \text{Borel},
$$
where the last term is also written $\int_A F^\infty(\mu^s)$. For $F=G$ and
$A=\Omega$ this is the energy $\mathcal{G}(\mu)$ of \eqref{eq:energy}. For
$F=G^\infty$ it is the \emph{$G$-variation}
$$
|\mu|_G:=G^\infty\Big(\frac{\dif\mu}{\dif|\mu|}\Big)|\mu|\in\mathscr{M}^+(\Omega),
\qquad c|\mu|\le|\mu|_G\le C|\mu|,
$$
evaluated on $A$, and the \emph{$G$-polar}
$p_G:=\frac{\dif\mu}{\dif|\mu|_G}=\frac{\dif\mu}{\dif|\mu|}\big/G^\infty\big(\frac{\dif\mu}{\dif|\mu|}\big)$
takes values in $\SG$ $|\mu|_G$-a.e.; see Definition~\ref{def:Gvar} and
Lemma~\ref{lem:wellposed}, where these objects are studied in detail.

\section{The \texorpdfstring{$G$}{G}-compactification and generalized Young measures}\label{compactification section}

\subsection{The compactification of \texorpdfstring{$\V$}{V} induced by \texorpdfstring{$G^\infty$}{G infty}}

Let
$$
\BG:=\{z\in\V: G^\infty(z)<1\},
\qquad
\overline{\BG}=\{G^\infty\le1\},
\qquad
\SG:=\{G^\infty=1\}=\partial\BG .
$$
By \textup{(A2)}, $\BG$ is a bounded open set, star-shaped with respect to
the origin, with $\{|z|<1/C\}\subseteq\BG\subseteq\{|z|<1/c\}$, and $\SG$ is
a compact set which meets every ray $\{tz:t>0\}$, $z\neq0$, exactly once,
namely at $z/G^\infty(z)$. The map
$$
\iota_G(z):=\frac{z}{1+G^\infty(z)},\qquad z\in\V,
$$
is a homeomorphism of $\V$ onto $\BG$, with inverse
$$
\iota_G^{-1}(\hat z)=\frac{\hat z}{1-G^\infty(\hat z)},
\qquad \hat z\in\BG,
$$
and $1-G^\infty(\iota_G(z))=(1+G^\infty(z))^{-1}$. Since
$\iota_G(tz)\to z/G^\infty(z)$ as $t\to\infty$, the compact set
$\overline{\BG}=\BG\cup\SG$ is a compactification of $\V$ whose points at
infinity are the directions of $\V$, represented on the indicatrix $\SG$.
As in \cite{KristensenRindler10} we use the variable $\hat z$ for points of
$\overline{\BG}$. For $G^\infty=|\cdot|$ we recover the ball
compactification $\overline{B}=B\cup\mathbb{S}_{\V}$, with
$\iota(z)=z/(1+|z|)$; by \textup{(A2)} the two compactifications are
equivalent, i.e.\ the identity of $\V$ extends to a homeomorphism
$\overline{B}\to\overline{\BG}$ (Remark~\ref{rmk:compact}).

\subsection{Integrands on \texorpdfstring{$\V$}{V}}\label{homogeneous integrands subsection}

We denote by $\HH$ the space of continuous functions $\Phi\colon\V\to\R$ of
linear growth with a regular recession function $\Phi^\infty$, normed by
$$
\|\Phi\|_{\HH}:=\sup_{z\in\V}\frac{|\Phi(z)|}{1+G^\infty(z)} .
$$
For $\Phi\in\HH$ we define $T\Phi\colon\overline{\BG}\to\R$ by
\begin{equation}\label{definition of T homogeneous setting}
T\Phi(\hat{z})= \begin{dcases}
    \big(1-G^\infty(\hat{z})\big)\, \Phi \left(  \frac{\hat{z}}{1- G^\infty(\hat{z}) }\right) & \text{if } \hat{z} \in \BG ,
    \\
    \Phi^\infty(\hat{z}) & \text{if } \hat{z} \in \SG .
\end{dcases}
\end{equation}

\begin{lemma}\label{T isometry lemma}
$T$ is an isometric isomorphism of $\HH$ onto $\CC(\overline{\BG})$. Its
inverse is $T^{-1}g=(1+G^\infty)\,g\circ\iota_G$, and
$(T^{-1}g)^\infty(z)=G^\infty(z)\,g\big(z/G^\infty(z)\big)$ for $z\neq0$.
\end{lemma}

\begin{proof}
For $\hat z\in\BG$ put $t:=(1-G^\infty(\hat z))^{-1}$, so that
$T\Phi(\hat z)=t^{-1}\Phi(t\hat z)$. If $\hat z\to\theta\in\SG$, then
$t\to\infty$, and the regularity of the recession function gives
$T\Phi(\hat z)\to\Phi^\infty(\theta)$; as $\Phi^\infty$ is continuous,
$T\Phi\in\CC(\overline{\BG})$. Conversely, let $g\in\CC(\overline{\BG})$
and $\Phi:=(1+G^\infty)\,g\circ\iota_G$. Then $\Phi$ is continuous,
$|\Phi|\le\|g\|_\infty(1+G^\infty)$, and, since $\iota_G(tz')\to
z/G^\infty(z)$ as $z'\to z\neq0$ and $t\to\infty$,
$$
\frac{\Phi(tz')}{t}=\Big(\frac1t+G^\infty(z')\Big)\,g\big(\iota_G(tz')\big)
\longrightarrow G^\infty(z)\,g\Big(\frac{z}{G^\infty(z)}\Big),
$$
while the limit is $0$ for $z=0$. Hence $\Phi\in\HH$, and $T\Phi=g$ (on
$\SG$ by the formula for $\Phi^\infty$). Finally
$\|T\Phi\|_\infty=\sup_{\BG}|T\Phi|=\|\Phi\|_{\HH}$, because $\BG$ is dense
in $\overline{\BG}$ and $1-G^\infty(\iota_G(z))=(1+G^\infty(z))^{-1}$.
\end{proof}

For instance, $T(1+G^\infty)\equiv1$, $TG^\infty=G^\infty$ on
$\overline{\BG}$, and $T\ell_e(\hat z)=\hat z\cdot e$ for the linear
integrand $\ell_e(z):=z\cdot e$, $e\in\V$.

By Lemma~\ref{T isometry lemma} and the Riesz representation theorem,
$T^*\colon\mathscr{M}(\overline{\BG})\cong\CC(\overline{\BG})^*\to\HH^*$ is
an isometric isomorphism. Given $\ell\in\HH^*$, let
$\mu:=(T^*)^{-1}\ell$ and
$$
\mu^0:= (\iota_G^{-1})_{\#}\left[ \big(1-G^\infty\big) \mu  \Lcorner  \BG \right]\in\mathscr{M}(\V),
\qquad
\mu^\infty :=\mu  \Lcorner \SG\in\mathscr{M}(\SG).
$$
Then, for every $\Phi\in\HH$,
$$
\ell(\Phi)=\int_{\overline{\BG}} T\Phi \dif \mu
= \int_\V \Phi  \dif \mu^0 + \int_{\SG} \Phi^{\infty}  \dif \mu^\infty ,
$$
and, since $\big(1+G^\infty(\iota_G^{-1}(\hat z))\big)\big(1-G^\infty(\hat z)\big)=1$
for $\hat z\in\BG$,
$$
\int_\V \big(1+G^\infty\big) \dif |\mu^0| = |\mu|(\BG),
\qquad
\norm{\ell}_{\HH^*}= \int_\V \big(1+G^\infty\big) \dif |\mu^0| + |\mu^\infty|(\SG).
$$
In particular $\mu^0$ has finite first moment, and $\ell\ge0$ (that is,
$\ell(\Phi)\ge0$ whenever $\Phi\ge0$) if and only if $\mu^0\ge0$ and
$\mu^\infty\ge0$. In view of this representation we identify $\ell$ with
the pair $(\mu^0,\mu^\infty)$.

\emph{Kantorovich metric.} The compact set $\overline{\BG}$ is a metric space
with the Euclidean distance, so that, by
Section~\ref{kantorovich norm subsection}, the relative weak$^*$ topology of
$\mathscr{M}^+(\overline{\BG})$ is induced by the Kantorovich norm
$\|\cdot\|_{\mathrm{K}}$. Transporting it through $T^*$, for
$\ell\in\HH^*$ we set
$$
\|\ell\|_{\mathrm{K}}:=\|(T^*)^{-1}\ell\|_{\mathrm{K}}
=\sup\Big\{|\ell(\Phi)|:\ \Phi\in\HH,\ \|T\Phi\|_{\mathrm{LIP}}\le1\Big\},
\qquad
d_{\mathrm{K}}(\ell,\ell'):=\|\ell-\ell'\|_{\mathrm{K}} ,
$$
so that, for $\ell=(\mu^0,\mu^\infty)$,
$\|\ell\|_{\mathrm{K}}=\sup_{\|T\Phi\|_{\mathrm{LIP}}\le1}
\big|\int_\V\Phi\dif\mu^0+\int_{\SG}\Phi^\infty\dif\mu^\infty\big|$. As
$T^*$ is a homeomorphism of $\mathscr{M}^+(\overline{\BG})$ onto the
positive cone $(\HH^*)^+$ for the relative weak$^*$ topologies, $d_{\mathrm{K}}$
metrizes the weak$^*$ topology of $(\HH^*)^+$.

An integrand $\Phi\in\HH$ with $\|T\Phi\|_{\mathrm{LIP}}\le1$ is called
\emph{admissible}. Since $\Phi=(1+G^\infty)\,T\Phi\circ\iota_G$, every
admissible $\Phi$ satisfies $|\Phi|\le1+G^\infty$, $|\Phi(0)|\le1$ and
$|\Phi^\infty|\le1$ on $\SG$, and it is Lipschitz, with
\begin{equation}\label{lip constant L}
\mathrm{lip}(\Phi)\le L:=\max\Big\{\mathrm{lip}(G^\infty),\ 1+\frac{\mathrm{lip}(G^\infty)}{c}\Big\}
\end{equation}
by Remark~\ref{lip remark}; for $G^\infty=|\cdot|$ one has $L=2$. The
integrand $1+G^\infty$ is admissible. A linear integrand $\ell_e$ with
$|e|=1$ is admissible only up to a factor: as $|\hat z|\le c^{-1}$ on
$\overline{\BG}$, one has $\|T\ell_e\|_{\mathrm{LIP}}\le1+c^{-1}$, so that
$(1+c^{-1})^{-1}\ell_e$ is admissible, and similarly for
$(1+c^{-1})^{-1}|\cdot|$.

\subsection{Integrands on \texorpdfstring{$\Omegacl\times\V$}{the closure of Omega times V}}\label{inhomogeneous integrands subsection}

We denote by $\E_G(\Omega,\V)$ the space of continuous
$\Phi\colon\Omegacl\times\V\to\R$ such that
$$
\|\Phi\|_{\E_G}:=\sup_{(x,z) \in \Omegacl \times \V} \frac{|\Phi(x,z)|}{1+G^\infty(z)} < \infty
$$
and such that the recession function
$$
\Phi^\infty(x,z) = \lim_{t \to \infty} \frac{\Phi(x,tz)}{t}
$$
exists, uniformly in $(x,z)\in\Omegacl\times\SG$. As in
Lemma~\ref{T isometry lemma}, the map
$T\colon\E_G(\Omega,\V)\to\CC(\overline{\Omega\times\BG})$, where
$\overline{\Omega\times\BG}=\Omegacl\times\overline{\BG}$, defined by
\begin{equation}\label{isomorphism T inhomogeneous setting}
T\Phi(x,\hat{z})= \begin{dcases}
    \big(1-G^\infty(\hat{z})\big)\, \Phi \left( x, \frac{\hat{z}}{1- G^\infty(\hat{z}) }\right) & \text{if } (x,\hat{z})\in \Omegacl \times \BG ,
    \\
    \Phi^\infty(x,\hat{z}) & \text{if } (x,\hat{z})\in \Omegacl \times \SG ,
\end{dcases}
\end{equation}
is an isometric isomorphism, with inverse
$T^{-1}g(x,z)=(1+G^\infty(z))\,g(x,\iota_G(z))$. In other words,
$$
\E_G(\Omega,\V)=\Big\{(x,z)\mapsto\big(1+G^\infty(z)\big)\,g\big(x,\iota_G(z)\big)
:\ g\in\CC(\Omegacl\times\overline{\BG})\Big\},
$$
which is the description of $\E_G(\Omega,\V)$ used in
Section~\secref{rmk:compact}. For $G^\infty=|\cdot|$ one obtains the class
$\E(\Omega,\V)$ of \cite{KristensenRindler10,KristensenRaita22}; by
\textup{(A2)}, $\E_G(\Omega,\V)=\E(\Omega,\V)$ as sets, with equivalent
norms. For $\eta\in\CC(\Omegacl)$ and $\Phi\in\HH$ we write
$(\eta\otimes\Phi)(x,z):=\eta(x)\Phi(z)$. Then
$\eta\otimes\Phi\in\E_G(\Omega,\V)$ and $T(\eta\otimes\Phi)=\eta\otimes T\Phi$,
so that, by the Stone--Weierstrass theorem, the linear span of such
products is dense in $\E_G(\Omega,\V)$.

Consequently $T^*\colon\mathscr{M}(\overline{\Omega\times\BG})\to\E_G(\Omega,\V)^*$
is an isometric isomorphism, and a homeomorphism of the positive cones for
the relative weak$^*$ topologies. Every
$\ell=T^*\mu\in\E_G(\Omega,\V)^*$ acts as
\begin{equation}\label{expression of ell}
\ell(\Phi)= \int_{\Omegacl\times \V} \Phi \, \dif \mu^0 + \int_{\Omegacl\times \SG} \Phi^{\infty} \, \dif \mu^\infty,
\qquad \Phi\in\E_G(\Omega,\V),
\end{equation}
where
$$
\mu^0:= (\mathrm{id}\times\iota_G^{-1})_{\#}\left[ \big(1-G^\infty(\hat{z})\big) \mu  \Lcorner \Omegacl \times \BG \right]\in\mathscr{M}(\Omegacl\times\V),
\qquad
\mu^\infty :=\mu  \Lcorner \Omegacl \times \SG ,
$$
and $\int(1+G^\infty(z))\dif|\mu^0|(x,z)=|\mu|(\Omegacl\times\BG)$.

\subsection{Generalized Young measures}\label{young measures subsection}

\emph{Elementary Young measures.} Every $v\in\mathscr{M}(\Omega,\V)$
defines an element $\boldsymbol{\eps}_v\in\E_G(\Omega,\V)^*$ by
$$
\boldsymbol{\eps}_v(\Phi):=\int_{\Omega}\Phi(x,v^a(x))\dd x
+ \int_{\Omega} \Phi^\infty \left(x, \frac{\dd v^s}{\dd |v^s|}(x)\right) \dd |v^s|(x),
\qquad \Phi \in \E_G(\Omega,\V),
$$
which is well defined by the growth condition on $\Phi$; we call
$\boldsymbol{\eps}_v$ the \emph{elementary Young measure} associated with
$v$. As $\Phi^\infty(x,\cdot)$ is positively $1$-homogeneous, in terms of
the $G$-variation and of the $G$-polar $p_G$ of $v^s$ (see
Section~\ref{notation section})
$$
\boldsymbol{\eps}_v(\Phi)=\int_{\Omega}\Phi(x,v^a(x))\dd x
+\int_\Omega\Phi^\infty\big(x,p_G(x)\big)\dif|v^s|_G(x),
$$
that is, $\boldsymbol{\eps}_v$ is represented by the triple
$(\delta_{v^a(x)},|v^s|_G,\delta_{p_G(x)})$ in the sense of
Lemma~\ref{YM measure theoretic definition} below. Since
$\boldsymbol{\eps}_v\ge0$ and $T(1+G^\infty)\equiv1$,
$$
\|\boldsymbol{\eps}_v\|_{\E_G(\Omega,\V)^*}=\boldsymbol{\eps}_v(1+G^\infty)
=\Leb(\Omega)+|v|_G(\Omega)\le\Leb(\Omega)+C|v|(\Omega).
$$

\emph{Generation.} If $(v_j)\subset\mathscr{M}(\Omega,\V)$ and
$\boldsymbol{\eps}_{v_j}\wstar\nu$ in $\E_G(\Omega,\V)^*$, we write
$v_j\toY\nu$ and say that $(v_j)$ \emph{generates} $\nu$; equivalently,
$(T^*)^{-1}\boldsymbol{\eps}_{v_j}\wstar(T^*)^{-1}\nu$ in
$\mathscr{M}(\overline{\Omega\times\BG})$. As $\E_G(\Omega,\V)$ is
separable, every sequence with $\sup_j|v_j|(\Omega)<\infty$ has a
subsequence which generates some $\nu$. The limit $\nu$ is in general not
elementary, but $(T^*)^{-1}\nu$ belongs to the set
\begin{equation}\label{T*-1 of YM}
\bigcap_{\phi \in \CC(\overline{\Omega})}  \left\{\mu \in \mathscr{M}^+(\overline{\Omega \times \BG}): \int_{\overline{\Omega \times \BG}} \phi(x)(1-G^\infty(\hat z)) \dif \mu(x,\hat z)=\int_{\Omega} \phi \dif x  \right\}.
\end{equation}
Indeed, $T^{-1}\big[\phi(x)(1-G^\infty(\hat z))\big]=\phi\otimes1$ and
$(\phi\otimes 1)^\infty=0$, so that for $\mu=(T^*)^{-1}\boldsymbol{\eps}_v$
the integral in \eqref{T*-1 of YM} equals
$\boldsymbol{\eps}_v(\phi\otimes1)=\int_\Omega\phi\dd x$; moreover the set
\eqref{T*-1 of YM} is convex and weak$^*$ closed, being an intersection of
preimages of points under the weak$^*$ continuous maps
$\mu\mapsto\int\phi(x)(1-G^\infty(\hat z))\dif\mu(x,\hat z)$.

\begin{definition}\label{YM definition}
A \emph{(generalized) Young measure} on $\Omega$ with values in $\V$ is an
element $\nu\in\E_G(\Omega,\V)^*$ such that $(T^*)^{-1}\nu$ belongs to the
set \eqref{T*-1 of YM}. We denote the set of Young measures by
$\mathbf{Y}(\Omega,\V)$.
\end{definition}

\begin{lemma}\label{YM measure theoretic definition}
A functional $\nu \in \E_G(\Omega,\V)^*$ is a Young measure if and only if
it can be identified with a triple
$((\nu_x)_{x \in \Omega}, \lambda, (\nu_x^\infty)_{x \in \Omegacl})$, uniquely
determined by $\nu$, where
\begin{enumerate}
    \item $\nu_x \in \mathscr{M}_1^+(\V)$ for $\Leb$-almost every $x \in \Omega$ and $x \mapsto \nu_x$ is weakly$^*$ $\Leb$-measurable;
    \item $\lambda \in \mathscr{M}^+(\Omegacl)$;
    \item $\nu_x^\infty \in \mathscr{M}_1^+(\SG)$ for $\lambda$-almost every $x \in \Omegacl$ and $x \mapsto \nu^\infty_x$ is weakly$^*$ $\lambda$-measurable;
    \item $\int_\Omega \int_\V G^\infty(z) \dif \nu_x(z)\dif x < \infty$;
\end{enumerate}
and where
$$
\nu(\Phi)= \int_{ \Omega}\int_\V \Phi(x,z)\dif \nu_x(z) \dif x + \int_{\Omegacl}\int_{\SG} \Phi^\infty(x,z)\dif \nu^\infty_x(z) \dif \lambda (x), \qquad \text{for } \Phi \in \E_G(\Omega,\V).
$$
\end{lemma}

We call $\nu_x$ the \emph{oscillation measures}, $\lambda$ the
\emph{concentration measure} and $\nu^\infty_x$ the
\emph{concentration-angle measures} of $\nu$, and we write
$\nu=(\nu_x,\lambda,\nu^\infty_x)$. Condition (4) is the \emph{moment
condition}. The proof of Lemma~\ref{YM measure theoretic definition}, by
disintegration of $(T^*)^{-1}\nu$, is given in the Appendix
(Section~\ref{appendix section}). A different derivation is given in
Section~\secref{rmk:transfer}: by Lemma~\ref{rmk:transfer} the triple of
$\nu$ is obtained from its triple in the spherical normalisation of
\cite{KristensenRindler10}, whose concentration measure is denoted
$\lambda^{\mathbb{S}}$, by the explicit formulas \eqref{eq:transfer}.

\begin{rmk}[The $G$-normalisation]\label{rmk:G normalisation}
Since $G^\infty\equiv1$ on $\SG$, we have $\langle G^\infty,\nu^\infty_x\rangle=1$
for $\lambda$-a.e.\ $x$ (cf.\ \eqref{eq:unitmass}), and testing $\nu$ with
$\eta\otimes(1+G^\infty)$, $\eta\in\CC(\Omegacl)$, gives
$$
\nu\big(\eta\otimes(1+G^\infty)\big)=\int_\Omega\eta\,\langle1+G^\infty,\nu_x\rangle\dd x
+\int_{\Omegacl}\eta\dif\lambda .
$$
Thus $\lambda$ is precisely the part of the $G$-energy carried by
concentrations. In particular, if $v_j\toY\nu$, then
$(1+G^\infty(v_j^a))\Leb\Lcorner\Omega+|v_j^s|_G\wstar
\langle1+G^\infty,\nu_x\rangle\Leb\Lcorner\Omega+\lambda$ in
$\mathscr{M}(\Omegacl)$.
\end{rmk}

\emph{Pairings and barycentre.} We also write $\nu(\Phi)=\ddual{\Phi}{\nu}$,
and $\nu(\eta\otimes\Phi)=\ddual{\eta\otimes\Phi}{\nu}=\langle\nu,\eta\otimes\Phi\rangle$;
by density, $\nu$ is determined by its values on such tensor products,
$\eta\in\CC(\Omegacl)$, $\Phi\in\HH$. The \emph{barycentre} of
$\nu\in\mathbf{Y}(\Omega,\V)$ is the measure
$$
\overline{\nu}:=\overline{\nu}_x\Leb\Lcorner\Omega+\overline{\nu}^\infty_x\lambda\in\mathscr{M}(\Omegacl,\V),
\qquad
\overline{\nu}_x:=\int_\V z\dif\nu_x(z),
\quad
\overline{\nu}^\infty_x:=\int_{\SG}z\dif\nu^\infty_x(z);
$$
here $\overline{\nu}_{(\cdot)}\in\LL^1(\Omega,\V)$ by (4) and \textup{(A2)},
and $\overline{\nu}^\infty_x\in\overline{\mathrm{co}}(\SG)$, so that
$|\overline{\nu}^\infty_x|\le c^{-1}$. If
$\lambda=\lambda^a\Leb\Lcorner\Omega+\lambda^s$, then
$\overline{\nu}\Lcorner\Omega=(\overline{\nu}_x+\lambda^a(x)\overline{\nu}^\infty_x)\Leb\Lcorner\Omega
+\overline{\nu}^\infty_x\lambda^s\Lcorner\Omega$. Testing the generation
with $\eta\otimes\ell_e$, $\eta\in\CC_0(\Omega)$, $e\in\V$, shows that
$v_j\toY\nu$ implies $v_j\wstar\overline{\nu}\Lcorner\Omega$ in
$\mathscr{M}(\Omega,\V)$.

\emph{Homogeneous Young measures.} An element of $\HH^*$ of the form
$$
\Phi\longmapsto\int_\V\Phi\dif\nu^0+\int_{\SG}\Phi^\infty\dif\nu^\infty,
\qquad \nu^0\in\mathscr{M}^+_1(\V),\ \ \int_\V G^\infty\dif\nu^0<\infty,
\ \ \nu^\infty\in\mathscr{M}^+(\SG),
$$
is called a \emph{homogeneous Young measure} and is identified with the
pair $(\nu^0,\nu^\infty)$; its barycentre is
$\overline{\nu^0}+\overline{\nu^\infty}$. Writing $\nu^\infty=m\sigma$ with
$m=\nu^\infty(\SG)\ge0$ and $\sigma\in\mathscr{M}^+_1(\SG)$, we also use
the triple notation $(\nu^0,m,\sigma)$; on a cube $Q$,
$(\nu^0,m\Leb\Lcorner Q,\sigma)$ denotes the Young measure on $Q$ with
constant oscillation and concentration-angle measures and with
concentration measure $m\Leb\Lcorner Q$, whose action on
$x$-independent integrands $\Phi\in\HH$ is $|Q|$ times that of
$(\nu^0,m\sigma)$. For instance, if $\nu=(\nu_x,\lambda,\nu^\infty_x)$ is a
Young measure, then for $\Leb$-a.e.\ $x\in\Omega$ the pair
$(\nu_x,\lambda^a(x)\nu^\infty_x)$, written $(\nu_x,\lambda^a(x),\nu^\infty_x)$
in Section~\ref{inhomogenization section}, is a homogeneous Young measure
with barycentre $\overline{\nu}_x+\lambda^a(x)\overline{\nu}^\infty_x$. For
a bounded open set $U\subset\R^n$ and $w\in\LL^1(U,\V)$, the
\emph{homogeneous elementary Young measure} $\boldsymbol{\eps}_w\in\HH^*$ is
$$
\boldsymbol{\eps}_w(\Phi):=\dashint_U\Phi(w(x))\dd x,\qquad\Phi\in\HH,
$$
that is, $\boldsymbol{\eps}_w=\big(\Leb(U)^{-1}w_\#(\Leb\Lcorner U),0\big)$;
the set $U$ (usually $\XX$ or a cube $Q$) will always be clear from the
context.

\emph{Metrizability.} Bounded subsets of $\E_G(\Omega,\V)^*$ are weak$^*$
metrizable, $\E_G(\Omega,\V)$ being separable. A convenient metric is the
bounded-Lipschitz distance
$$
d(\nu,\nu'):=\sup\Big\{|\nu(\eta\otimes\Phi)-\nu'(\eta\otimes\Phi)|:\
\|\eta\|_{\mathrm{LIP}}\le1,\ \Phi\in\HH,\ \|T\Phi\|_{\mathrm{LIP}}\le1\Big\},
$$
the supremum being over $\eta\in\mathrm{LIP}(\Omegacl)$, with
$\mathrm{lip}(\eta)$ computed with respect to the $\ell^\infty$-distance
(Section~\ref{notation section}); this is the distance used in
Section~\ref{inhomogenization section}. On bounded sets, convergence in $d$
is equivalent to weak$^*$ convergence: the functions $\eta\otimes T\Phi$
appearing in the supremum form a relatively compact subset of
$\CC(\Omegacl\times\overline{\BG})$ by the Arzel\`a--Ascoli theorem, and
bounded weak$^*$ convergent sequences converge uniformly on such sets;
conversely, their linear span is dense by the Stone--Weierstrass theorem.
This is what makes the diagonal extraction arguments of
Sections~\secref{sufficiency proposition}
and~\ref{inhomogenization section} legitimate (cf.\ \cite[Prop.~6]{laszlo},
\cite[Lemma~3]{KristensenRindler10}, \cite{KristensenRaita22}).

\section{\texorpdfstring{$\mathscr{A}$}{A}-quasiconvex integrands of linear growth}\label{approximation qc function section}

We denote by $\mathbf{Q}$ the class of $\mathscr{A}$-quasiconvex functions
$f\colon\V\to\R$ of linear growth, that is, such that for some constant
$c_f>0$
$$
|f(z)| \leq c_f\big(1+G^\infty(z)\big) \qquad \text{for all }z \in \V.
$$
For $f \in \mathbf{Q}$ the upper recession function
\begin{equation}\label{upper recession function}
f^\infty(z)= \limsup_{\substack{z' \to z\\ t \to +\infty }} \frac{f(tz')}{t}
\end{equation}
is positively $1$-homogeneous, $f^\infty(tz)=tf^\infty(z)$ for all $t>0$
and $z\in\V$, and $|f^\infty|\le c_fG^\infty$. If $f$ has a regular
recession function, the upper limit in \eqref{upper recession function} is
a limit; this is the case for every $f\in\mathbf{Q}\cap\HH$.

\begin{lemma}\label{qc properties lemma}
Let $f\in\mathbf{Q}$. Then:
\begin{enumerate}[label=\textup{(\alph*)}]
\item $f$ is $\WC$-convex and Lipschitz, with $\mathrm{lip}(f)$ bounded in
terms of $c_f$, $C$ and $\WC$ only;
\item $f^\infty$ is $\mathscr{A}$-quasiconvex, $\WC$-convex and Lipschitz;
\item $f(z+w)\le f(z)+f^\infty(w)$ for all $z\in\V$ and $w\in\WC$.
\end{enumerate}
\end{lemma}

We shall also use the following automatic convexity result, in the form in
which it is applied in the proof of Lemma~\ref{lem:dent}.

\begin{theorem}[Kirchheim--Kristensen {\cite[Thm.~1.1]{KirchheimKristensen16}}]\label{KK theorem}
Assume \textup{(A1)}, and let $h\colon\V\to\R$ be positively
$1$-homogeneous and $\WC$-convex. Then for every $\theta\in\WC$ there is a
linear function $\ell\colon\V\to\R$ with $\ell\le h$ on $\V$ and
$\ell(\theta)=h(\theta)$. Consequently $h(\overline{\sigma})\le\langle h,\sigma\rangle$
for every $\sigma\in\mathscr{M}^+_1(\V)$ with finite first moment and
$\overline{\sigma}\in\WC$.
\end{theorem}

By Lemma~\ref{qc properties lemma}(b), Theorem~\ref{KK theorem} applies to
$h=f^\infty$ for every $f\in\mathbf{Q}$, and in particular to $h=G^\infty$.
The next lemma is the $G$-analogue of \cite[Lemma~4.3]{KristensenRaita22};
it underlies all the duality arguments below.

\begin{lemma}\label{approximation lemma}
Let $f\in \mathbf{Q}$ and $\delta\in(0,1)$. Then
$$
g_\delta:= \max\Big\{f,\ f^\infty + \delta G^\infty -\frac{1}{\delta} \Big\}
$$
is $\mathscr{A}$-quasiconvex, and there is $s=s(\delta)>0$ such that
$g_\delta(z)= f^\infty(z) + \delta G^\infty(z) -\frac{1}{\delta}$ for all
$z$ with $G^\infty(z)\geq s$; in particular $g_\delta\in\HH$ and
$g_\delta^\infty=f^\infty+\delta G^\infty$. Furthermore
$(g_\delta)_{\delta \in (0,1)}$ is equi-Lipschitz, and
$$
g_\delta(z) \searrow f(z) \quad \text{and}\quad g^\infty_\delta(z) \searrow f^\infty(z)
\qquad\text{for every } z\in\V, \text{ as } \delta \searrow 0 .
$$
\end{lemma}

\begin{proof}
By Lemma~\ref{qc properties lemma}(b) and \textup{(A2)},
$f^\infty+\delta G^\infty-\frac1\delta$ is $\mathscr{A}$-quasiconvex, the
defining inequality being linear in the integrand; hence so is $g_\delta$,
as a maximum of two $\mathscr{A}$-quasiconvex functions.

We find $s=s(\delta)$. By the definition of $f^\infty$, for every
$\theta\in\SG$ there is $s_\theta>0$ such that
$$
\frac{f(t\theta')}{t}<f^\infty(\theta)+\frac{\delta}{4}
\qquad\text{for all } t\ge s_\theta \text{ and } \theta'\in\SG \text{ with } |\theta'-\theta|<1/s_\theta ,
$$
and, $f^\infty$ being continuous, we may also assume that
$f^\infty(\theta)\le f^\infty(\theta')+\frac{\delta}{4}$ for such $\theta'$.
Cover the compact set $\SG$ by finitely many balls
$B_{1/s_{\theta_i}}(\theta_i)$, $i=1,\dots,N$, and let
$s:=\max\{s_{\theta_1},\dots,s_{\theta_N},2/\delta^2\}$. Then for all
$\theta'\in\SG$ and $t\ge s$
$$
\frac{f(t\theta')}{t}<f^\infty(\theta')+\frac{\delta}{2}
\le f^\infty(\theta')+\delta-\frac{1}{t\delta},
$$
because $\frac{1}{t\delta}\le\frac{1}{s\delta}\le\frac{\delta}{2}$.
Multiplying by $t$ and using $G^\infty(t\theta')=t$ and the homogeneity of
$f^\infty$, we get $f(z)<f^\infty(z)+\delta G^\infty(z)-\frac1\delta$ for
every $z=t\theta'$ with $G^\infty(z)=t\ge s$. Hence
$g_\delta=f^\infty+\delta G^\infty-\frac1\delta$ on $\{G^\infty\ge s\}$, and
since $f^\infty$ and $G^\infty$ are continuous and positively
$1$-homogeneous, $g_\delta\in\HH$ with $g^\infty_\delta=f^\infty+\delta G^\infty$.

As $\delta\searrow0$, both $\delta G^\infty$ and $-\frac1\delta$ decrease;
hence $g_\delta$ decreases to $f$, as $f^\infty+\delta G^\infty-\frac1\delta\to-\infty$,
and $g^\infty_\delta=f^\infty+\delta G^\infty$ decreases to $f^\infty$.
Finally, for $\delta\in(0,1)$,
\begin{equation}\label{uniform growth of g delta}
-c_f(1+G^\infty)\le f\le g_\delta\le|f|+|f^\infty|+G^\infty\le(2c_f+1)(1+G^\infty),
\qquad |g^\infty_\delta|\le(c_f+1)G^\infty ,
\end{equation}
so that the functions $g_\delta$ have uniformly linear growth, and they are
equi-Lipschitz by Lemma~\ref{qc properties lemma}(a).
\end{proof}

We call $f\in\mathbf{Q}$ \emph{special} if there is $r=r(f)>0$ such that
$f(z)=f^\infty(z)$ for $|z|>r$; special integrands belong to $\HH$. In
Sections~\ref{inhomogenization section}
and~\secref{lemma: CD_f as sup of differentiability point} the Jensen
inequalities are tested on special integrands only; this is no loss of
generality, by the following consequence of
Lemma~\ref{approximation lemma}.

\begin{corollary}\label{special integrands corollary}
Let $\nu^0\in\mathscr{M}^+_1(\V)$ with $\int_\V G^\infty\dif\nu^0<\infty$,
$\nu^\infty\in\mathscr{M}^+(\SG)$ and $z\in\V$. If
\begin{equation}\label{jensen special}
\int_\V f\dif\nu^0+\int_{\SG}f^\infty\dif\nu^\infty\ \ge\ f(z)
\end{equation}
holds for every special $f\in\mathbf{Q}$, then it holds for every
$f\in\mathbf{Q}$.
\end{corollary}

\begin{proof}
Let $f\in\mathbf{Q}$, and let $g_\delta$ and $s=s(\delta)$ be as in
Lemma~\ref{approximation lemma}. The function $g_\delta+\frac1\delta$ is
$\mathscr{A}$-quasiconvex, of linear growth, with recession function
$g^\infty_\delta$, and it coincides with $g^\infty_\delta$ on
$\{G^\infty\ge s\}\supseteq\{|z|\ge s/c\}$; thus it is special. Since
$\nu^0$ is a probability measure, \eqref{jensen special} for
$g_\delta+\frac1\delta$ is equivalent to \eqref{jensen special} for
$g_\delta$. Letting $\delta\searrow0$, \eqref{jensen special} for $f$ follows
by dominated convergence, thanks to \eqref{uniform growth of g delta}.
\end{proof}

The same argument applies, at every point $x$ at which it holds, to the
pointwise Jensen inequality of Proposition~\ref{sufficiency proposition},
which is \eqref{jensen special} with $\nu^0=\nu_x$,
$\nu^\infty=\lambda^a(x)\nu^\infty_x$ and
$z=\overline{\nu}_x+\lambda^a(x)\overline{\nu}^\infty_x$.

\section{Homogeneous \texorpdfstring{$\mathscr{A}$}{A}-free Young measures}\label{hahn banach section}

Fix $z \in \V$. We denote by $\mathbf{Y}^{\mathscr{A}}_z$ the set of
homogeneous Young measures
$(\nu^0,\nu^\infty)\in\mathscr{M}^+_1(\V)\times\mathscr{M}^+(\SG)$ for which
there is a sequence $(\pphi_j)\subset\CC^\infty_c(\XX,\V)$ such that
$\|\pphi_j\|_{\WW^{l-1,1}(\XX,\V)}\to0$ and
\begin{equation}\label{membership Y condition}
\boldsymbol{\eps}_{z+ \mathscr{B} \pphi_j}(\Phi)=\int_\XX   \Phi\big(z+ \mathscr{B} \pphi_j(x)\big)\dd x
\longrightarrow \int_\V \Phi \, \dd \nu^0 + \int_{\SG} \Phi^\infty \dd \nu^\infty
\qquad \text{for all } \Phi \in \HH,
\end{equation}
that is, $\boldsymbol{\eps}_{z+\mathscr{B}\pphi_j}\wstar(\nu^0,\nu^\infty)$
in $\HH^*$. These are the \emph{homogeneous $\mathscr{A}$-free Young
measures with barycentre $z$}. Since all the functionals involved are
positive, \eqref{membership Y condition} can equivalently be stated in
terms of the Kantorovich metric:
$$
\big\|\boldsymbol{\eps}_{z+ \mathscr{B} \pphi_j} - (\nu^0,\nu^\infty)\big\|_{\mathrm{K}} \longrightarrow 0
\qquad (j\to\infty).
$$

\begin{lemma}\label{elementary in Y lemma}
For every $\pphi\in\CC^\infty_c(\XX,\V)$ we have
$\boldsymbol{\eps}_{z+\mathscr{B}\pphi}\in\mathbf{Y}^{\mathscr{A}}_z$. In
particular the family
$\{\boldsymbol{\eps}_{z+\mathscr{B}\pphi}:\pphi\in\CC^\infty_c(\XX,\V)\}$ is
a weakly$^*$ dense subset of $\mathbf{Y}^{\mathscr{A}}_z$.
\end{lemma}

The same scaling allows us to realize homogeneous $\mathscr{A}$-free Young
measures on an arbitrary cube.

\begin{rmk}[Realization on cubes]\label{rescaling remark}
Let $Q=x_Q+r\XX$ be a cube. For $\pphi\in\CC^\infty_c(\XX,\V)$ the map
$\pphi^Q(x):=r^l\pphi\big((x-x_Q)/r\big)$ belongs to $\CC^\infty_c(Q,\V)$,
and for $w\in\V$, $\Phi\in\HH$ and $0\le m\le l-1$
$$
\mathscr{B}\pphi^Q(x)=(\mathscr{B}\pphi)\Big(\frac{x-x_Q}{r}\Big),
\qquad
\dashint_Q\Phi(w+\mathscr{B}\pphi^Q)\dd x=\int_\XX\Phi(w+\mathscr{B}\pphi)\dd x,
$$
$$
\|\nabla^m\pphi^Q\|_{\LL^1(Q)}=r^{\,l-m+n}\,\|\nabla^m\pphi\|_{\LL^1(\XX)} .
$$
Consequently, for every $(\nu^0,\nu^\infty)\in\mathbf{Y}^{\mathscr{A}}_z$
and all $\eps,\tau>0$ there is a map $\pphi^Q\in\CC^\infty_c(Q,\V)$ with
$\|\pphi^Q\|_{\WW^{l-1,1}(Q,\V)}<\tau$ and
$$
\big\|(\nu^0,\nu^\infty)-\boldsymbol{\eps}_{z+\mathscr{B}\pphi^Q}\big\|_{\mathrm{K}}<\eps,
\qquad\text{where}\quad
\boldsymbol{\eps}_{z+\mathscr{B}\pphi^Q}(\Phi)=\dashint_Q\Phi(z+\mathscr{B}\pphi^Q)\dd x .
$$
This is what is meant in Section~\ref{inhomogenization section} by
realizing the elementary Young measures on $Q$ rather than on $\XX$.
\end{rmk}

\begin{lemma}\label{weak star closed Y lemma}
The set $\mathbf{Y}^{\mathscr{A}}_z$ is weakly$^*$ closed in $\HH^*$.
\end{lemma}

\begin{lemma}\label{convex Y lemma}
The set $\mathbf{Y}^{\mathscr{A}}_z$ is convex.
\end{lemma}

\begin{proposition}\label{homogeneous case characterization}
Let $\nu=(\nu^0,\nu^\infty)$ be a homogeneous Young measure, that is,
$\nu^0 \in \mathscr{M}_1^+(\V)$ has finite first moment and
$\nu^\infty \in \mathscr{M}^+(\SG)$, and let $z \in \V$. Then
$\nu \in \mathbf{Y}^{\mathscr{A}}_z$ if and only if
$\overline{\nu^0}+\overline{\nu^\infty}=z$ and
$$
\int_\V f \, \dd \nu^0 + \int_{\SG} f^\infty \, \dd \nu^\infty \geq f(z)
\qquad\text{for all } f\in\mathbf{Q}.
$$
\end{proposition}

\begin{proof}
\emph{Only if.} Fix $\nu \in \mathbf{Y}^{\mathscr{A}}_z$ and a sequence
$(\pphi_j)$ as in \eqref{membership Y condition}. Testing
\eqref{membership Y condition} with the linear integrands $\ell_e$,
$e\in\V$, for which $\ell_e^\infty=\ell_e$, and using
$\int_\XX\mathscr{B}\pphi_j\dd x=0$, we get
$(\overline{\nu^0}+\overline{\nu^\infty})\cdot e=z\cdot e$ for all $e$. Next,
let $f\in\mathbf{Q}$. For every $\Phi \in \HH$ with $\Phi \geq f$, the
$\mathscr{A}$-quasiconvexity of $f$ gives
$$
\nu(\Phi) =\lim_{j \to \infty} \int_\XX \Phi(z+\mathscr{B} \pphi_j) \dd x
\geq \limsup_{j \to \infty} \int_\XX f(z+\mathscr{B} \pphi_j) \dd x \geq f(z).
$$
By Lemma~\ref{approximation lemma} we may take $\Phi=g_\delta$, and letting
$\delta\searrow0$ by dominated convergence, thanks to
\eqref{uniform growth of g delta}, we obtain
$$
\int_\V f \, \dd \nu^0 + \int_{\SG} f^\infty \, \dd \nu^\infty
= \lim_{\delta \searrow 0 } \nu(g_\delta) \geq f(z).
$$

\emph{If.} By Lemmas~\ref{weak star closed Y lemma} and~\ref{convex Y lemma},
$\mathbf{Y}^{\mathscr{A}}_z$ is a weakly$^*$ closed convex subset of
$\HH^*$; hence, by the Hahn--Banach separation theorem, it is the
intersection of the weakly$^*$ closed half-spaces containing it. Such a
half-space is of the form
$$
H=\{\ell \in \HH^*: \ell(\Phi) \geq t\}
$$
for some $\Phi \in \HH$ and $t \in \R$. If $\mathbf{Y}^{\mathscr{A}}_z\subseteq H$,
Lemma~\ref{elementary in Y lemma} gives
$$
t \leq \boldsymbol{\eps}_{z+ \mathscr{B} \pphi}(\Phi) = \int_{\XX} \Phi(z+\mathscr{B} \pphi) \dd x
\qquad\text{for all } \pphi \in \CC_c^\infty(\XX,\V),
$$
and taking the infimum over $\pphi$ we get $t \leq \Phi^{\mathrm{qc}}(z)$.
In particular $\Phi^{\mathrm{qc}}(z)>-\infty$, so that $\Phi^{\mathrm{qc}}$
is a real-valued $\mathscr{A}$-quasiconvex function with
$\Phi^{\mathrm{qc}}\le\Phi$, hence of linear growth
(cf.\ \cite{FonsecaMuller99,KristensenRaita22}), that is,
$\Phi^{\mathrm{qc}}\in\mathbf{Q}$; moreover
$(\Phi^{\mathrm{qc}})^\infty\le\Phi^\infty$. Applying the assumption to
$f=\Phi^{\mathrm{qc}}$,
$$
\nu(\Phi)=\int_\V \Phi \, \dd \nu^0 + \int_{\SG} \Phi^\infty \, \dd \nu^\infty
\geq \int_\V \Phi^{\mathrm{qc}} \, \dd \nu^0 + \int_{\SG} (\Phi^{\mathrm{qc}})^\infty \, \dd \nu^\infty
\geq \Phi^{\mathrm{qc}}(z) \geq t ,
$$
so that $\nu \in H$. As $H$ was arbitrary, $\nu \in \mathbf{Y}^{\mathscr{A}}_z$.
\end{proof}
\section{Sufficiency theorem}
\begin{proposition}\label{sufficiency proposition}
Let $\mathscr{A}$ be a linear homogeneous differential operator of order $k$ on $\mathbb{R}^n$ from $\mathbb{V}$ to $\mathbb{W}$. Assume it has constant rank and satisfies the spanning cone condition. Let $\Omega$ be a bounded open subset of $\mathbb{R}^n$, $v \in \mathscr{M}(\Omega, \mathbb{V})$ and $(v_j)$ a sequence in $\mathscr{M}(\Omega, \mathbb{V})$ satisfying $v_j \wstar v$ in $\mathscr{M}(\Omega, \mathbb{V})$, $\mathscr{A} v_j \to \mathscr{A} v$ in $\WW^{-k,p}_{\mathrm{loc}}(\Omega, \mathbb{W})$ for some $p \in (1,\frac{n}{n-1})$ and $(v_j)$ generates $\nu = (\nu_x, \lambda, \nu_x^\infty)$. Let $\lambda = \lambda^a \Leb \Lcorner \Omega + \lambda^s$ be the Lebesgue--Radon--Nikod\'ym decomposition with respect to $\Leb$. Then there exists an $\Leb$ negligible subset $N^a$ of $\Omega$ such that
\begin{equation}\label{Jensen inequality in sufficiency theorem}
    \int_{\mathbb{V}} f  \dif \nu_x + \lambda^a(x) \int_{\mathbb{S}_G} f^\infty  \dif \nu_x^\infty \geq f\left(\overline{\nu}_x + \lambda^a(x) \overline{\nu}_x^\infty\right)
\end{equation}
holds for all $x \in \Omega \setminus N^a$ and all $f \colon \mathbb{V} \to \mathbb{R}$ that are $\mathscr{A}$-quasiconvex and of linear growth. Furthermore, let $|v^s|_G$ and $p_G=\frac{\dif v^s}{\dif|v^s|_G}$ be the $G$-variation and the $G$-polar of $v^s$ (Section~\ref{notation section} and Lemma~\ref{lem:wellposed}), so that $v^s=p_G\,|v^s|_G$ with $p_G\in\SG$ $|v^s|_G$-almost everywhere, and let $\lambda^s = b_G\,|v^s|_G + \lambda^*$, $\lambda^*\perp|v^s|_G$, be the Lebesgue--Radon--Nikod\'ym decomposition of $\lambda^s$ with respect to $|v^s|_G$. Then

\begin{equation}\label{conditions on nu infty in sufficiency theorem}
    \begin{cases}
        b_G(x)\,\overline{\nu}_x^\infty = p_G(x) \in \Lambda_{\mathscr{A}}\cap\SG & \text{for } |v^s|_G \text{ almost all } x, \\[6pt]
        \overline{\nu}_x^\infty = 0 & \text{for } \lambda^* \Lcorner \Omega \text{ almost all } x.
    \end{cases}
\end{equation}
In particular $b_G\geq 1$ $|v^s|_G$-almost everywhere, that is, $\lambda^s\Lcorner\Omega\geq|v^s|_G$, and $\overline{\nu}^\infty_x=b_G(x)^{-1}p_G(x)$ lies on the segment from $0$ to $p_G(x)$: $G^\infty(\overline{\nu}^\infty_x)=b_G(x)^{-1}\leq1$, and $\overline{\nu}^\infty_x\in\SG$ if and only if $b_G(x)=1$.
\end{proposition}
\begin{proof}
The proof of (\ref{Jensen inequality in sufficiency theorem}) follows exactly from the same proof of Proposition 1.1 in \cite{kristensen raita}, since the role of $\mathbb{S}_G$, which is $\mathbb{S}_\V$ in the authors' case, is not played in the proof. The initial Lebesgue negligible set $N_f$ depends on the particularly chosen $f$, but then one uses the separability of the class of integrands $\mathbb{H}$ to deduce the validity for a common Lebesgue negligible set and all $f \in \mathbb{H}$, and finally the approximation Lemma \ref{approximation lemma} to upgrade it for all  $\mathscr{A}$-quasiconvex functions of linear growth. \\
To prove (\ref{conditions on nu infty in sufficiency theorem}), we recall that given
$$
\begin{cases}
   v_j \wstar v & \text{in }\mathscr{M}(\Omega,\V),\\   
   \mathscr{A}v_j \to \mathscr{A}v & \text{in }\WW^{-k,p}_{\mathrm{loc}}(\Omega,\mathbb W), \ 1<p<\frac{n}{n-1},
\end{cases}
$$
from the first convergence we can extract a subsequence not relabelled such that
$$
v_j \toY \nu, \quad \nu=((\nu_x),\lambda, (\nu_x^\infty)),
$$
meaning that
\begin{multline*}
\int_\Omega \Phi(x,v_j^a(x))\dif x+\int_{\Omega} \Phi^\infty\left(x, \frac{\dif v_j^s}{\dif|v_j^s|}(x) \right)\dif|v_j^s|(x)    \\
\underset{j \to \infty}{\longrightarrow}\int_{ \Omega}\int_\V \Phi(x,z)\dif \nu_x(z) \dif x + \int_{\overline \Omega}\int_{\mathbb{S}_G} \Phi^\infty(x,z)\dif \nu^\infty_x(z) \dif \lambda (x), \quad \text{for all } \Phi \in \E_G(\Omega,\V).
\end{multline*}
The convergence $\mathscr{A}v_j \to \mathscr{A}v$ in $\WW^{-k,p}_{\mathrm{loc}}(\Omega,\mathbb W)$ gives additional information on the Young measure $\nu$, in the form of the Jensen inequalities (\ref{Jensen inequality in sufficiency theorem}) and the structure (\ref{conditions on nu infty in sufficiency theorem}), the latter to be proved. \\
Fix $\Omega' \Subset \Omega$ with $(\Leb+\lambda)(\partial \Omega')=0$. Write
$$
\nu \Lcorner \Omega' =((\nu_x)_{x \in \Omega'},\lambda, (\nu_x^\infty)_{x \in \Omega'}).
$$
Choose $\eta \in \CC_c^\infty(\Omega)$ such that $\boldsymbol{1}_{\Omega'} \leq \eta \leq \boldsymbol{1}_{\Omega}$; it is clear by localization that $\eta v_j \in \mathscr{M}(\R^n,\V)$,  $\eta v \in \mathscr{M}(\R^n,\V)$ and that $\eta v_j \wstar \eta v \ \ \text{in }\mathscr{M}(\R^n,\V)$. Also, since $(\eta v_j) \Lcorner \Omega' = v_j \Lcorner \Omega' $, we claim
$$
v_j \Lcorner \Omega' \toY \nu \Lcorner \Omega'.
$$
To check the Young measure generation above, take any $\Phi \in \E_G(\Omega',\V)$. Since the transformed $T\Phi$ is bounded and (uniformly) continuous on the closed set $\overline{\Omega '} \times \overline{\B}$, by Tietze extension theorem we can find an extension $\widetilde {T \Phi}$ of $T \Phi$, bounded and uniformly continuos on $\overline{\Omega} \times \overline{\B}$, such that $\widetilde{T\Phi}(x,\hat z)=T\Phi(x,\hat z) $ for $(x,\hat z)\in \overline{\Omega '} \times \overline{\B}$. Transforming back with $T^{-1}$, we get that $\widetilde \Phi:=T^{-1}[\widetilde{T\Phi}]  \in \E_G(\Omega,\V)$ since $T$ is an isomorphism. We check that $\widetilde \Phi$ extends $\Phi$ in the expected way. From the explicit formula (\ref{isomorphism T inhomogeneous setting}) of $T$, since the transformation does not involve the $x$ variable but only $z$, we get that
$$
\widetilde \Phi(x,z)=(1+G^\infty(z))\widetilde{T\Phi}\left(x, \frac{z}{1+G^\infty(z)} \right)= \Phi(x,z) \quad (x,z) \in \overline{\Omega '} \times \V
$$
since $\widetilde {T \Phi}|_{\overline{\Omega'} \times \B}=T\Phi$, and
\begin{multline*}
\widetilde \Phi^\infty (x,z) = \lim_{t \to \infty}\frac{\widetilde \Phi (x,tz)}{t}=\lim_{t \to \infty}(1+G^\infty(tz))\left( 1- \frac{G^\infty(tz)}{1+G^\infty(tz)}\right)\Phi \left(x, \frac{\frac{tz}{1+G^\infty(tz)}}{1-\frac{G^\infty(tz)}{1+G^\infty(tz)}} \right)  \\=\lim_{t \to \infty}\frac{\Phi(x,tz)}{t}=\Phi^\infty(x,z), \quad (x,z) \in \overline{\Omega '} \times \mathbb S_G.
\end{multline*}
The extension given by Tietze extension theorem is not linear, however it is possible to find a \emph{linear} extension operator, see for example Theorem 1.8 (Borsuk--Dugundji) in \cite{brudnyi extension}. It also works for unbounded continuous functions on a closed set, thus one can extends directly $\Phi$ to $\widetilde \Phi$ before transforming and then check that $\widetilde \Phi$ is actually in $\E_G(\Omega,\V)$. \\
We construct now $\phi_\delta \in \CC_c^0(\Omega)$ be a cutoff function such that $\phi_\delta(x)\equiv 1$ in $\overline{\Omega'}$ and such that $\phi_\delta \searrow 0 $ on $\Omega \setminus \overline{\Omega'}$; this is possible since $\overline{\Omega'}$ is closed. Say $\text{supp}(\phi_\delta) \Subset B_\delta(\Omega')$, its closed $\delta$-neighbourhood. First note that for all $\delta >0$
\begin{multline*}
(\nu \Lcorner \Omega')(\Phi)=\int_{\Omega'} \langle \Phi,\nu_x \rangle \dif x + \int_{\Omega'}\langle \Phi^\infty, \nu_x^\infty \rangle \dif \lambda \\= \int_{\Omega}\phi_\delta \langle \widetilde \Phi, \nu_x \rangle \dif x +\int_{\Omegacl}\phi_\delta \langle \widetilde \Phi^\infty, \nu^\infty_x \rangle \dif \lambda.\\
-\int_{B_\delta(\Omega')\setminus\Omega'}\phi_\delta \langle \widetilde\Phi, \nu_x \rangle \dif x - \int_{B_\delta(\Omega')\setminus\Omega'}\phi_\delta \langle \widetilde \Phi^\infty, \nu^\infty_x \rangle \dif \lambda,
\end{multline*}
and likewise for each $j \in \N$
\begin{multline}
\boldsymbol{\eps}_{v_j \Lcorner \Omega'}(\Phi)=\int_{\Omega'} \Phi(x,v_j^a)\dif x + \int_{\Omega'} \Phi^\infty\left(x, \frac{\dif v_j^s}{\dif|v_j^s|} \right)\dif|v_j^s| \\ = \int_{\Omega}\phi_\delta \widetilde \Phi(x,v_j^a)\dif x + \int_{\Omega}\phi_\delta \widetilde\Phi^\infty\left(x, \frac{\dif v_j^s}{\dif|v_j^s|} \right)\dif|v_j^s| 
\\ - \int_{B_\delta(\Omega')\setminus \Omega'}\phi_\delta \widetilde\Phi(x,v_j^a)\dif x - \int_{B_\delta(\Omega')\setminus\Omega'}\phi_\delta \widetilde\Phi^\infty\left(x, \frac{\dif v_j^s}{\dif|v_j^s|} \right)\dif|v_j^s|.
\end{multline}
Then, using the shorthand notation $\int_A \phi_\delta \langle \widetilde \Phi, \nu \rangle:=\int_A \phi_\delta \langle \widetilde \Phi, \nu_x \rangle \dif x + \int_A \phi_\delta \langle \widetilde\Phi^\infty, \nu_x^\infty \rangle \dif \lambda$ and likewise for $\int_A \phi_\delta \langle \widetilde \Phi, \boldsymbol{\eps}_{v_j} \rangle$ for any Borel set $A\subset\Omegacl$,
\begin{multline*}
\left| \boldsymbol{\eps}_{v_j \Lcorner \Omega'}(\Phi) -(\nu \Lcorner \Omega')(\Phi)\right|= \left| \boldsymbol{\eps}_{v_j}(\phi_\delta \widetilde \Phi) - \nu(\phi_\delta \widetilde\Phi) + \int_{B_\delta(\Omega')\setminus\Omega'}\phi_\delta \langle \widetilde \Phi, \nu \rangle - \int_{B_\delta(\Omega')\setminus\Omega'}\phi_\delta \langle \widetilde \Phi,\boldsymbol{\eps}_{v_j} \rangle\right| \\
\leq  \left| \boldsymbol{\eps}_{v_j}(\phi_\delta \widetilde \Phi)- \nu(\phi_\delta \widetilde\Phi)\right| + \left|\int_{B_\delta(\Omega')\setminus\Omega'}\phi_\delta \langle \widetilde \Phi, \nu \rangle \right| + \left| \int_{B_\delta(\Omega')\setminus\Omega'}\phi_\delta \langle \widetilde \Phi,\boldsymbol{\eps}_{v_j} \rangle\right|.
\end{multline*}
Since $\phi_\delta \widetilde\Phi \in \E_G(\Omega,\V)$ for all $\delta>0$, it holds $\lim_{j\to \infty} \boldsymbol{\eps}_{v_j}(\phi_\delta \widetilde\Phi)= \nu(\phi_\delta \widetilde\Phi)$ and therefore
\begin{multline*}
\limsup_{j \to \infty}\left| \boldsymbol{\eps}_{v_j \Lcorner \Omega'}(\Phi) -(\nu \Lcorner \Omega')(\Phi)\right| \leq 
\left|\int_{B_\delta(\Omega')\setminus\Omega'}\phi_\delta \langle \widetilde \Phi, \nu \rangle \right| \\ + \limsup_{j \to \infty}\left| \int_{B_\delta(\Omega')\setminus\Omega'}\phi_\delta \langle \widetilde \Phi,\boldsymbol{\eps}_{v_j} \rangle\right|.
\end{multline*}
We can estimate the first term on the right-hand side by $0\leq \phi_\delta \leq 1$, $|\widetilde \Phi^\infty| \leq \max |T \widetilde \Phi|=:c_\Phi$ and $|\widetilde \Phi| \leq c_\Phi(1+G^\infty)$ (as the Tietze extension can be chosen with the same supremum, $c_\Phi=\max|T\Phi|\leq\|T\Phi\|_{\text{LIP}}$):
\begin{multline*}
\left|\int_{B_\delta(\Omega')\setminus\Omega'}\phi_\delta \langle \widetilde\Phi, \nu_x \rangle \dif x + \int_{B_\delta(\Omega')\setminus\Omega'}\phi_\delta \langle \widetilde \Phi^\infty, \nu^\infty_x \rangle \dif \lambda\right|\\ \leq \int_{B_\delta(\Omega')\setminus\Omega'} \left| \langle \widetilde\Phi, \nu_x \rangle \right| \dif x  + c_\Phi \,\lambda(B_\delta(\Omega')\setminus\Omega') \\
\leq c_\Phi \int_{B_\delta(\Omega')\setminus\Omega'}\langle 1+G^\infty, \nu_x \rangle  \dif x + c_\Phi \,\lambda(B_\delta(\Omega')\setminus\Omega')\underset{\delta \searrow 0}{\longrightarrow 0},
\end{multline*}
the right-hand side tends to $0$ as $\delta \searrow 0$ since $\lambda(\partial \Omega')=0$ and since $\langle 1 + G^\infty, \nu_x \rangle \in \LL^1(\Omega)$ by the moment condition (4) of Lemma \ref{YM measure theoretic definition} and $\Leb(\partial \Omega')=0$. With the same calculations we can estimate the second term with the $\limsup_j$, using now that $|\widetilde \Phi^\infty| \leq c_\Phi \, G^\infty$,
\begin{multline*}
\left| \int_{B_\delta(\Omega')\setminus \Omega'}\phi_\delta \widetilde\Phi(x,v_j^a)\dif x +\int_{B_\delta(\Omega')\setminus\Omega'}\phi_\delta \widetilde\Phi^\infty\left(x, \frac{\dif v_j^s}{\dif|v_j^s|} \right)\dif|v_j^s| \right| \\
\leq c_\Phi \int_{B_\delta(\Omega')\setminus \Omega'} (1+G^\infty(v_j^a))\dif x   + c_\Phi\int_{B_\delta(\Omega')\setminus\Omega'} G^\infty\left( \frac{\dif v_j^s}{\dif|v_j^s|} \right)\dif|v_j^s|   \\
=c_\Phi\int_{B_\delta(\Omega')\setminus\Omega'} (1+G^\infty(\cdot))(v_j).
\end{multline*}
The scalar (positive) measures 
$$
(1+G^\infty(\cdot))(v_j):=(1+G^\infty(v_j^a))\Leb\Lcorner\Omega + G^\infty(v_j^s/|v_j^s|)|v_j^s|
$$
weakly$^*$ converge to $\langle 1+G^\infty,\nu_x\rangle \Leb \Lcorner\Omega + \lambda $, simply by testing the convergence $v_j \toY \nu$ with the test function $\eta(x)(1+G^\infty(z))$, with $\eta \in \CC_0(\Omega)$,
$$
\int_{\Omega}\eta(1+G^\infty(v_j^a))\dif x + \int_{\Omega} \eta G^\infty\left(\frac{\dif v_j^s}{\dif |v_j^s|}\right)\dif|v_j^s| \to \int_\Omega \eta \langle 1+G^\infty, \nu_x \rangle \dif x + \int_{\Omega} \eta \int_{\mathbb S_G}G^\infty \dif \nu_x^\infty \dif \lambda,
$$
and $\int_{\mathbb S_G}G^\infty \dif \nu_x^\infty=1$ for $\lambda$-almost every $x$ since $G^\infty=1$ on $\mathbb S_G$ by its definition. By the upper semicontinuity on compact sets given by the weak$^*$ convergence, applied to the compact set $B_\delta(\Omega')\setminus\Omega'$ (we took the closed $\delta$-neighbourhood of $\Omega'$) we get
$$
\limsup_{j \to \infty}\int_{B_\delta(\Omega')\setminus\Omega'} (1+G^\infty(\cdot))(v_j) \leq \int_{B_\delta(\Omega')\setminus\Omega'} \langle 1+G^\infty,\nu_x\rangle \dif x + \lambda(B_\delta(\Omega')\setminus\Omega').
$$
The right-hand side is the same quantity of above which tends to $0$ as $\delta \searrow 0$. We can recollect the estimates by writing
\begin{multline*}
 \limsup_{j \to \infty}\left| \boldsymbol{\eps}_{v_j \Lcorner \Omega'}(\Phi) -(\nu \Lcorner \Omega')(\Phi)\right| \leq  2c_\Phi \left(\int_{B_\delta(\Omega')\setminus\Omega'} \langle 1+G^\infty,\nu_x\rangle \dif x + \lambda(B_\delta(\Omega')\setminus\Omega') \right)    
\end{multline*}
and since $\delta >0$ is arbitrary, we get the desired convergence for all test $\Phi \in \E_G(\Omega',\V)$ yielding $v_j \Lcorner\Omega' \toY \nu \Lcorner \Omega'$.

Without loss of generality, we can take $\Omega'$ such that simultaneously
$$
(\Leb+\lambda)(\partial \Omega')=0,
$$
$$
(v_j)(\partial \Omega')=0 \quad \text{for all }j \in \N.
$$
This choice is possible since only a countable collection from a family of disjoint boundaries can have positive $v_j$ measure, and repeating this for all $j$ one gets a countable union of countable collections. \\
Since $\eta v_j \in \mathscr{M}(\R^n,\V)$ for each $j \in \N$, we can find sequences of smooth functions $\rho_\e \ast (\eta v_j) $ such that
$$
\rho_\e \ast  (\eta v_j) \overset{\text{$\langle \cdot \rangle$-strictly}}{\longrightarrow} \eta v_j \quad \text{in $\R^n$ as }\eps \to 0.
$$
Since elementary Young measures are closed with respect to $\langle \cdot \rangle$-strict convergence by an application of Reshetnyak continuity theorem, this yields
$$
\rho_\e \ast  (\eta v_j) \toY \eta v_j \quad \text{in $\R^n$ as }\eps \to 0.
$$
By the same strategy of proof for the localization of the Young measure generation in $\Omega'$, which can be done by the assumption $(v_j)(\partial \Omega')=0$ for each $j \in \N$, one has
$$
\big(\rho_\e \ast (\eta v_j)\big)\big|_{\Omega'} \toY \eta v_j \Lcorner \Omega'=v_j \Lcorner \Omega' \quad \text{as }\eps \to 0.
$$
As the notion of convergence for Young measures is metrizable on bounded sets, we can use the diagonal-sequence extraction principle (see for example Proposition 6 in \cite{laszlo}, Lemma 3 in \cite{Kri Rindler}, or \cite{kristensen raita}) to find a null sequence $\eps_j \searrow 0$ such that
$$
\big(\rho_{\e_j} \ast (\eta v_j)\big)\big|_{\Omega'} \toY \nu \Lcorner \Omega' \quad \text{as }j \to \infty. 
$$
Put $\tilde v:= \eta v$ and $\tilde v_j :=\rho_{\e_j} \ast (\eta v_j)$. Define also $w_j:=\rho_{\e_j} \ast \tilde v $. Note that $w_j \to \tilde v$ $\langle \cdot \rangle$-strictly on $\R^n$ by standard convolution results. Also $w_j - \tilde v_j \wstar 0$ in $\mathscr{M}(\R^n,\V)$, since for $\phi \in \CC^0_0(\R^n, \V)$
\begin{multline}
\int_{\R^n} \phi \big((\rho_{\e_j} \ast \eta v_j )- (\rho_{\e_j} \ast \eta v )\big) \dif x = \int_{\R^n} (\rho_{\e_j}\ast \phi ) \dif \, (\eta v_j - \eta v) \\ 
= \int_{\R^n} \big((\rho_{\e_j}\ast \phi) -\phi \big) \dif \, (\eta v_j - \eta v) + \int_{\R^n}\phi \dif \, (\eta v_j - \eta v), 
\end{multline}
the last quantity tends to $0$ as $j \to \infty$ by $\eta v_j \wstar \eta v$ in $\mathscr{M}(\R^n,\V)$, whereas the first quantity can be bounded by
$$
\left|\int_{\R^n} \big((\rho_{\e_j}\ast \phi) -\phi \big) \dif \, (\eta v_j - \eta v)\right| \leq \|(\rho_{\e_j}\ast \phi) -\phi\|_{\infty}|\eta v_j + \eta v|(\R^n) \to 0
$$
since $(\rho_{\e_j}\ast \phi) \to \phi$ uniformly in $\R^n$ and $|\eta v_j + \eta v|(\R^n)$ is uniformly bounded by a constant. With the same strategy it can be proved that both $w_j$ and $\tilde v_j$ converge to $\tilde v$ in $\mathscr{M}(\R^n,\V)$.\\
By the definition of $w_j$ and $\tilde v_j$, and by $\mathscr{A}v_j \to \mathscr{A}v$ in $\WW^{-k,p}_{\mathrm{loc}}(\Omega,\mathbb W)$ we deduce that
$$
\mathscr{A}(w_j - \tilde v_j) \to 0 \text{ in } \WW^{-k,p}(\R^n,\mathbb W).
$$
Indeed, for a given $\phi \in \WW_0^{k,p'}(\R^n,\mathbb{W})$,
we have
\begin{align*}
\langle \mathscr{A}(w_j - \tilde v_j), \phi \rangle &= \sum_{|\alpha|=k}(-1)^{k}\int_{\R^n}A_\alpha^T \partial^\alpha \phi \, \dif \,(w_j -\tilde v_j) \\
&=\sum_{|\alpha|=k}(-1)^{k}\int_{\R^n}A_\alpha^T \partial^\alpha \phi \, \big( \rho_{\eps_j} \ast (\eta v_j - \eta v)\big) \dif x \\ &=\sum_{|\alpha|=k}(-1)^{k}\int_{\R^n}\left( \rho_{\eps_j}\ast A_\alpha^T \partial^\alpha \phi \right) \dif\, (\eta v_j-\eta v) \\
&=\sum_{|\alpha|=k}(-1)^{k}\int_{\R^n}\left(  A_\alpha^T (\rho_{\eps_j}\ast \partial^\alpha \phi) \right) \dif\, (\eta v_j-\eta v)\\
&=\sum_{|\alpha|=k}(-1)^{k}\int_{\R^n}\left(  A_\alpha^T \partial^\alpha (\rho_{\eps_j}\ast  \phi) \right) \dif\, (\eta v_j-\eta v)= \langle \mathscr{A}(\eta v_j - \eta v),\rho_{\eps_j}\ast \phi \rangle
\end{align*}
or in short, $\mathscr{A}(w_j - \tilde v_j)= \rho_{\eps_j} \ast \mathscr{A}(\eta v_j - \eta v) $ since $\mathscr{A}$ has constant coefficients. From this equality,
\begin{align*}
\sup_{ \phi \in \WW_0^{k,p'}, \|\phi\|_{\WW_0^{k,p'}}\leq 1}|\langle \mathscr{A}(w_j - \tilde v_j), \phi \rangle|&= \sup_\phi|\langle \mathscr{A}(\eta v_j - \eta v),\rho_{\eps_j}\ast \phi \rangle| \\
&\leq \|\mathscr{A}(\eta v_j - \eta v)\|_{\WW^{-k,p}}\sup_{\phi}\| \rho_{\eps_j} \ast \phi \|_{\WW_0^{k,p'}} \\
&\leq \|\mathscr{A}(\eta v_j - \eta v)\|_{\WW^{-k,p}}\|\rho_{\eps_j}\|_{\LL^1} \sup_{\phi}\| \phi \|_{\WW_0^{k,p'}}\\ &\leq \|\mathscr{A}(\eta v_j - \eta v)\|_{\WW^{-k,p}} \to 0,
\end{align*}
where we used Young inequality for convolutions, $\|\rho_{\eps_j}\|_{\LL^1}=1$ and the assumption $\mathscr{A}v_j \to \mathscr{A}v$ in $\WW^{-k,p}_{\mathrm{loc}}(\Omega,\mathbb W)$, which we claim to yield $\mathscr{A}(\eta v_j) \to \mathscr{A}(\eta v)$ in $\WW^{-k,p}(\R^n,\mathbb W)$. The claim is proved by Leibniz rule
\begin{align*}
\mathscr{A}(\eta v_j)&=\eta \mathscr{A}v_j+ \sum_{|\alpha|=k}A_\alpha\sum_{\beta < \alpha}\binom{\alpha}{\beta}\partial^{\alpha-\beta}\eta \partial^\beta v_j \\
&\to \eta \mathscr{A}v +\sum_{|\alpha|=k}A_\alpha\sum_{\beta < \alpha}\binom{\alpha}{\beta}\partial^{\alpha-\beta}\eta \partial^\beta v = \mathscr{A}(\eta v),
\end{align*}
where the first convergence holds by definition of $\mathscr{A}v_j \to \mathscr{A}v$ in $\WW^{-k,p}_{\mathrm{loc}}(\Omega,\mathbb W)$, and for the second term we exploit that $v_j \wstar v$ in $\mathscr{M}(\Omega,\V)$ implies $\partial^\beta v_j \to \partial^\beta v $ in $\WW^{-k,p}_{\text{loc}}(\Omega,\V)$ for $|\beta|\leq k-1<k$, by virtue of the compact embedding of $\WW_0^{1,p'}(\Omega,\V) \hookrightarrow \CC_0(\Omega)$. Indeed, for a given $\phi \in \WW_0^{k,p'}(\R^n,\V)$, the assumption $p<\frac{n}{n-1}\leq \frac{n}{n-k}$ implies $p' > n \geq \frac{n}{k}$ with $\lfloor k -\frac{n}{p'} \rfloor=k-1$, so that by Sobolev embedding $\phi\in \CC^{k-1}_0(\R^n,\V)$, implying $\partial^\beta \phi \in \CC_0(\R^n,\V)$. Denoting $\xi= \partial^{\alpha-\beta}\eta \in \CC_c^\infty(\Omega)$, we want to prove for $|\beta|\leq k-1$
$$
\sup_{ \phi \in \WW_0^{k,p'}(\R^n,\V), \|\phi\|_{\WW_0^{k,p'}}\leq 1} | \langle \xi\,  \partial^\beta(v_j-v) , \phi \rangle  | = \sup_{\phi} \left|\int_{\Omega} \partial^\beta(\xi \phi)  \dif \, (v_j-v)   \right| \to 0 \quad \text{as }j \to \infty.
$$
Since $\partial^\beta(\xi \phi)$ involves at most derivatives of order $|\beta|$ of $\phi$, consequently at most $k-1$, and the derivatives of $\xi=\partial^{\alpha-\beta}\eta$ are bounded by a common constant $C=C(\eta)$, we get from the bounded family $\{ \phi \in \WW_0^{k,p'}(\R^n,\mathbb W): \|\phi\|_{\WW_0^{k,p'}}\leq 1\}$ that
$$
\|\partial^\beta(\xi \phi) \|_{\WW^{1,p'}_0(\Omega,\V)}\leq C(\eta)\| \phi\|_{\WW_0^{k,p'}(\R^n,\V)} \leq C(\eta).  
$$
Therefore the family $\{ \partial^\beta(\xi \phi): \phi \in \WW_0^{k,p'}(\R^n,\V), \ \|\phi\|_{\WW_0^{k,p'}}\leq 1\}$ is bounded in $\WW^{1,p'}_0(\Omega,\V)$, and thus it is compact in $\CC_0(\Omega,\V)$ by $p' >n$. The compactness of this family allows to conclude
$$
\sup_{\| \phi\|_{\WW_0^{k,p'}}\leq 1} \left|\int_{\Omega} \partial^\beta(\xi \phi)  \dif \, (v_j-v)   \right| \to 0 \quad \text{as }j \to \infty,
$$
by an $\eps$-net argument in $\CC_0(\Omega,\V)$ or a contradiction argument, contradicting $v_j \wstar v$ in $\mathscr{M}(\Omega,\V)$, i.e. if there were a sequence $(\partial^\beta(\xi \phi_k))_k$ and a subsequence $(v_{j_k})_k$ for which the integral above is always $> \delta$, then by compactness extract $\partial^\beta(\xi \phi_{k_h}) \to g$ in $\CC_0(\Omega,\V)$ and get the contradiction of $\int_\Omega g \, \dif \, (v_{j_{k_h}}-v) > \delta$ for all $h$. 
\\ 
We recall now an inequality due to Murat \cite{murat}. There exists $\text{P}_\mathscr{A}\colon \LL^p(\R^n,\V) \to \LL^p(\R^n,\V)$ continuous orthogonal projection onto $\text{ker}\mathscr{A}$ for each $p\in(1,\infty)$ and there exists a constant $c_p>0$ such that
\begin{equation}\label{lemma Murat}
\|\pphi - \text{P}_{\mathscr{A}}\pphi \|_{\LL^p} \leq c_p \|\mathscr{A}\pphi\|_{\WW^{-k,p}}
\end{equation}
holds for all $\pphi \in \LL^p(\R^n,\V)$. In Lemma 3.8 in \cite{kristensen raita}, the authors combine this inequality with local potentials $\mathscr{B}$, constructed in \cite{raita potentials}, to deduce the existence of a constant $C_p>0$ such that for all vector fields $\pphi \in \LL^p(\R^n,\V)$ it holds
$$
\inf_{\psi \in \CC_c^\infty(\R^n,\V), \mathscr{A}\psi=0} \| \pphi - \psi \|_{\LL^p}\leq C_p \|\mathscr{A}\pphi\|_{\WW^{-k,p}}.
$$
The construction used to find the function $\psi$ is linear in the choice of $\pphi$, as the projection $\text{P}_{\mathscr{A}}\pphi$ is linear in the choice of $\pphi$ as well, since it involves Fourier localization and convolutions with smooth mollifiers.
Since in particular $w_j \in \LL^p(\R^n,\V)$, find $\psi_j \in \CC_c^\infty(\R^n,\V)$ with $\mathscr{A}\psi_j=0$ and 
$$
\|w_j-\psi_j\|_{\LL^p}\leq (C_p+1)\|\mathscr{A}w_j\|_{\WW^{-k,p}}.
$$
Since the choice of $\psi_j$ is linear in $w_j$, we apply it to the difference $w_{j+m}-w_j$ to find $\zeta \in \CC_c^\infty(\R^n,\V)$ which satisfies $\zeta=\psi_{j+m}-\psi_j $ and
$$
\|w_{j+m}-w_j - \zeta\|_{\LL^p}\leq (C_p+1)\|\mathscr{A}(w_{j+m}-w_j)\|_{\WW^{-k,p}}
$$
which becomes
$$
\|(w_{j+m}-\psi_{j+m}) - (w_j - \psi_j)\|_{\LL^p}\leq (C_p+1)\|\mathscr{A}w_{j+m}-\mathscr{A}w_j\|_{\WW^{-k,p}}.
$$
By noticing $\mathscr{A}w_j=\rho_{\eps_j} \ast \mathscr{A}\tilde v$ and exploiting the representation of distributions in $\WW^{-k,p}$ as sum of distributional derivatives of $\LL^p$ maps, $\sum_{|\alpha|\leq k}\partial^\alpha f_\alpha$ with each $f_\alpha \in \LL^p$, it can be proved that $\mathscr{A}w_j$ converges to $\mathscr{A}\tilde v$ in $\WW^{-k,p}(\R^n,\mathbb W)$. In particular, in the inequality above the convergence $\mathscr{A}w_j \to \mathscr{A}\tilde v$ in $\WW^{-k,p}(\R^n,\mathbb W)$ implies that the sequence $(w_j-\psi_j)_j$ is Cauchy in $\LL^p(\R^n,\V)$, so it converges $w_j-\psi_j \to h$ for some $h \in \LL^p(\R^n,\V)$, in particular also in $\mathscr{M}(\R^n,\V)$ since both $w_j$ and $\psi_j$ have compact support to have $\LL^p$ embedded in $\mathscr{M}$. But we know already $w_j \wstar \tilde v$ in $\mathscr{M}(\R^n,\V)$, therefore
$$
\psi_j= w_j+ (\psi_j-w_j) \wstar \tilde v-h \quad \text{in } \mathscr{M}(\R^n,\V). 
$$
Each $\psi_j$ satisfies $\mathscr{A}\psi_j=0$, so this stays true at the limit
$$
\mathscr{A}(\tilde v-h)=0 \quad \text{in }\CC_c^\infty(\R^n,\mathbb W)^*
$$
because for every $\phi \in \CC_c^\infty(\R^n,\mathbb W)$ one has
$$
\langle \mathscr{A}(\tilde v-h), \phi \rangle = \langle \mathscr{A}(\tilde v-h-\psi_j), \phi \rangle = \int_{\R^n} \bigg( \sum_{|\alpha|=k}(-1)^k A_\alpha^T \partial^\alpha \phi\bigg) \, \dif\,(\tilde v-h -\psi_j)\to 0,
$$
as $\tilde v -h-\psi_j$ is tested against a $\CC_c^\infty(\R^n,\V)$ function.
Now we have that $h$ is a $\LL^p$ map so it gives no contribution to the singular part of $\tilde v$, meaning $(\tilde v -h)^s = \tilde v^s$. Therefore by the De Philippis--Rindler result \cite{dephiliprindler} applied to $\tilde v-h$ we get
$$
\frac{\tilde v^s}{|\tilde v|^s} \in \Lambda_{\mathscr{A}} \quad |\tilde v^s|\text{ almost everywhere}.  
$$
Finally, we restrict the measures to $\Omega'$. Since $\tilde v=v$ on $\Omega'$, the polar $\frac{v^s}{|v^s|}$ lies in $\Lambda_{\mathscr{A}}$ $|v^s|$-almost everywhere in $\Omega'$, hence so does the $G$-polar $p_G=\frac{v^s}{|v^s|}\big/G^\infty\big(\frac{v^s}{|v^s|}\big)$, which is a positive multiple of it and takes values in $\SG$ (Lemma~\ref{lem:wellposed}(c)); moreover $|v^s|_G$ and $|v^s|$ have the same null sets (Lemma~\ref{lem:wellposed}(b)). We find
$$
v \Lcorner \Omega'= \tilde v \Lcorner \Omega'= \overline{\nu}_x \Leb \Lcorner \Omega' + \overline{\nu}_x^\infty \lambda \Lcorner\Omega' ,
$$
and, writing $v^s=p_G\,|v^s|_G$ and $\lambda \Lcorner \Omega'= \lambda^a\Leb \Lcorner \Omega' + b_G\, |v^s|_G\Lcorner \Omega' + \lambda^* \Lcorner \Omega'$, where the three measures on the right-hand side are mutually singular, we can identify the terms in the equality of measures
$$
\big(v^a-\overline{\nu}_x-\lambda^a\overline{\nu}_x^\infty\big) \Leb \Lcorner \Omega' + \big( p_G-b_G\,\overline{\nu}_x^\infty \big) |v^s|_G \Lcorner \Omega' - \overline{\nu}_x^\infty \lambda^* \Lcorner \Omega'=0 .
$$
We obtain
\begin{align*}
p_G&= b_G\,\overline{\nu}_x^\infty \qquad |v^s|_G \text{ almost everywhere in }\Omega', \\
\overline{\nu}_x^\infty&=0 \qquad\qquad\ \ \lambda^* \text{ almost everywhere in }\Omega',
\end{align*}
and, since $\Omega$ is the union of an increasing sequence of sets $\Omega'$ as above, \eqref{conditions on nu infty in sufficiency theorem} follows. It remains to prove that $b_G\geq1$. As $p_G\neq0$, the first identity forces $b_G>0$, so that $\overline{\nu}^\infty_x=b_G^{-1}p_G\in\Lambda_{\mathscr{A}}$, $\Lambda_{\mathscr{A}}$ being a cone, and $|v^s|_G\ll\lambda$. Theorem~\ref{KK theorem}, applied to the positively $1$-homogeneous and $\Lambda_{\mathscr{A}}$-convex function $G^\infty$, and the identity $\langle G^\infty,\nu^\infty_x\rangle=1$ then give
$$
1=G^\infty(p_G)=b_G\,G^\infty\big(\overline{\nu}^\infty_x\big)\leq b_G\,\langle G^\infty,\nu^\infty_x\rangle=b_G
\qquad |v^s|_G\text{ almost everywhere}.
$$
Finally, since $\lambda^*$ is also the singular part of $\lambda^s$ with respect to $|v^s|$ and $\frac{\lambda^s}{|v|^s}=b_G\,G^\infty\big(\frac{v^s}{|v^s|}\big)$, the first identity in \eqref{conditions on nu infty in sufficiency theorem} can be equivalently written in terms of the Euclidean polar as
$$
\frac{\lambda^s}{|v|^s}(x)\,\overline{\nu}_x^\infty = \frac{v^s}{|v|^s}(x) \qquad\text{for } |v^s| \text{ almost all } x ,
$$
which is the form used in Section~\ref{convergence of energy section}.
\end{proof}
\section{Inhomogenization via dyadic martingale }\label{inhomogenization section}
Let $(\phi_t)_{t>0}$ be a standard mollifier on $\R^n$, obtained by the $\LL^1$ dilation by $t>0$ of a nonnegative smooth and compactly supported function $\phi:\R^n \to \R$, with $\LL^1$ norm equal to $1$, $\phi_t=t^{-n}\phi(x/t)$. It is convenient to assume that the support of $\phi$ is the closed unit cube $\overline{\XX}=[-\frac{1}{2},\frac{1}{2}]^n$ and that $\phi>0$ on its interior $\XX$. For later reference we put $M= \max |\nabla \phi|$. In the following, we shall also use the $\ell^\infty$ metric on $\R^n$, so for $x\in\R^n$
$$
\|x\|=\|x\|_{\ell^\infty}=\max\{|x_1|,...,|x_n|\}. 
$$
We state here the Propositions on the inhomogenization procedures, where the homogeneous Young measures are merged, or inhomogenized, by an approximation and
gluing procedure. The relative proofs are found later in this section. Although the result is not new, the proofs presented here use a novel procedure based on a dyadic martingale argument, in place of the Lusin-type approximation. 

\begin{proposition}[Singular part]\label{inhomogenization singular part}
Let $\Omega$ be a bounded open set in $\R^n$, and assume that $\nu= (\nu_x,\lambda,\nu_x^\infty )$   is a Young measure as defined in Lemma \ref{YM measure theoretic definition}, with the additional property that $\lambda(\partial \Omega )=0$. Let $v \in \mathscr{M}(\Omega,\V)$ be the barycentre of $\nu$,  i.e.
$$
v=\overline{\nu}= \overline{\nu}_x \Leb \Lcorner \Omega + \overline{\nu}_x^\infty \lambda= \big(\overline{\nu}_x+ \lambda^a(x) \overline{\nu}_x^\infty \big) \Leb \Lcorner \Omega + \overline{\nu}_x^\infty\lambda^s ,
$$
where $\overline{\nu}_x= \int_\V z \dif \nu_x(z)$ and $\overline{\nu}_x^\infty=\int_{\mathbb{S}_G}  z \dif \nu^\infty_x(z)$ and $\lambda=\lambda^a \Leb \Lcorner \Omega + \lambda^s$ is the Lebesgue--Radon--Nikod\'ym decomposition of $\lambda$ with respect to $\Leb$. Suppose that 
$$\mathscr{A} v=0 \quad \text{in the distributional sense},
$$ 
where $\mathscr{A}$ is a linear homogeneous differential operator of order $k$ on $\R^n$ from $\V$ to $\mathbb{W}$, with constant rank and satisfying the spanning cone condition.
Suppose also that the Jensen inequality
$$
\int_\V f \dif \nu_x + \lambda^a(x) \int_{\mathbb{S}_G} f^\infty \dif \nu_x^\infty \geq f(\overline{\nu}_x+ \lambda^a(x)\overline{\nu}_x^\infty) \quad \text{\emph{for $\Leb \Lcorner \Omega$ almost all $x$}}, 
$$
holds for all $\mathscr{A}$-quasiconvex functions $f:\V \to \R$ with linear growth for which one can find $r=r(f)$ so that $f(z)=f^\infty(z)$ for $|z|>r$ (cfr. Section \ref{approximation qc function section} for the approximation with such functions). \\
Given $\eps >0$ we can find $t_\eps>0$ such that for each $t\in (0,t_\eps]$ there exists $\pphi=\pphi_t \in \CC_c^\infty(\Omega,\V)$ with $\|\pphi\|_{\WW^{l-1,1}(\Omega,\V)} < \eps$ so that
$$
\bigg| \int_{\Omega} \eta(x)\Phi(0) \dif x + \int_{\Omega} \eta(x) \langle \Phi^\infty, \nu_x^\infty \rangle \dif \lambda^s  - \int_{\Omega} \eta \Phi\big( \phi_t \ast (\overline{\nu}_{(\cdot)}^\infty \lambda^s \Lcorner \Omega)+ \mathscr{B} \pphi \big) \dif x \bigg|< \eps
$$
holds uniformly in $\eta: \Omegacl \to \R$ and $\Phi: \V \to \R$ of class $\HH$ with
\begin{equation}\label{conditions on eta Phi}
\| \eta \|_{\emph{LIP}} \leq 1 \ \emph{and} \ \|T \Phi\|_{\emph{LIP}}\leq 1.
\end{equation}
where $\| \eta \|_{\emph{LIP}}= \emph{lip}(\eta)+\max_{\Omegacl}|\eta|$.
\end{proposition}
\begin{remo}
Proposition \ref{inhomogenization singular part} states that we can generate the singular part $(\delta_0,\lambda^s, \nu_x^\infty)=:\nu^s$ of the Young measure $\nu$ with a generating sequence in $\CC_c^\infty(\Omega,\V)$ given by the mollification of the (singular part of the) barycentre measure $\overline{\nu}_{(\cdot)}^\infty \lambda^s \Lcorner \Omega $ plus a sequence of potentials $\mathscr{B} \pphi$, by letting $\eps \to 0$.
Using the Kantorovich metric, this Young measure generation amounts to
$$
\|\boldsymbol{\eps}_{u_t} -\nu^s\|_\text{K} \equiv \|(T^*)^{-1}(\boldsymbol{\eps}_{u_t} -\nu^s)\|_\text{K} < \eps.
$$
\end{remo}
\begin{proposition}[AC part]\label{inhomogenization ac part}
Under the assumptions of Proposition \ref{inhomogenization singular part}, given $\eps >0$ we can find $t_\eps>0$ such that for each $t\in (0,t_\eps]$ there exists $\pphi=\pphi_t \in \CC_c^\infty(\Omega,\V)$ with $\|\pphi\|_{\WW^{l-1,1}(\Omega,\V)} < \eps$ so that
$$
\bigg| \int_{\Omega} \eta(x)\big( \langle \Phi, \nu_x \rangle +\lambda^a(x)\langle \Phi^\infty, \nu_x^\infty \rangle \big)\dif x  - \int_{\Omega} \eta \Phi\big( \phi_t \ast (\overline{\nu}_{(\cdot)} + \lambda^a \overline{\nu}_{(\cdot)}^\infty)+ \mathscr{B} \pphi \big) \dif x \bigg|< \eps
$$
holds uniformly in $\eta: \Omegacl \to \R$ and $\Phi: \V \to \R$ of class $\HH$ with
\begin{equation}\label{conditions on eta phi on AC inhomogenization proof}
\| \eta \|_{\emph{LIP}} \leq 1 \ \emph{and} \ \|T \Phi\|_{\emph{LIP}}\leq 1.
\end{equation} 
\end{proposition}
\subsection{Dyadic martingale and Young measures}\ \\
Work on a fixed half-open cube $Q_0 \supset \Omegacl$, and let $\mathscr{D}_k$ be the collection of dyadic subcubes $Q=Q_k \in \mathscr{D}_k$ of $k$-th generation, with side $l_k=2^{-k}l(Q_0)$. These partition $Q_0$, so each $x$ lies in exactly one cube $Q_k(x)\in \mathscr{D}_k$. 
Define $\mathcal{F}_k$ to be the $\sigma$-algebra generated by these cubes, $\mathcal F_k=\sigma(\mathscr D_k)$. \\
Let $E_k$ be the dyadic conditional expectation at generation $k$. For a measurable function $g$, it is defined by $E_k g:=\E[g| \, \mathcal F_k]$, the (almost everywhere) unique $\mathcal{F}_k$-measurable function which satisfies $\int_A g = \int _A E_k g$ for all $A \in \mathcal{F}_k$. \\
The action of $E_k$ is defined on functions, measures and Young measures in the following way. \\
\indent \textit{On functions.} Let $g\in \LL^1(Q_0,\V)$, set
$$
E_k g:= \sum_{Q \in \mathscr{D}_k} g_Q \boldsymbol{1}_{Q}, \quad g_Q:= \frac{1}{|Q|}\int_Q g \dif x. 
$$
This is exactly $E_k g = \E[g| \, \mathcal F_k]$ for normalized Lebesgue measure on $Q_0$. Two facts are used:
\begin{itemize}
    \item[$\bullet$] $E_k \eta$ is defined at every point, not just a.e., so it can be integrated against singular measures. For Lipschitz $\eta$, $\sup_\Omega|\eta - E_k \eta| \leq 2^{-k} \text{lip}(\eta) $ holds everywhere.
    \item[$\bullet$] For $g \in \WW^{1,1}$, applying Poincaré cube by cube (the inequality scales with their diameter $C(n)2^{-k}$) gives
    $$ 
    \|g-E_kg\|_{\LL^1}= \sum_{Q \in \mathscr{D}_k} \|g-g_Q\|_{\LL^1(Q)} \leq C(n)2^{-k}\sum_Q \|\nabla g\|_{\LL^1(Q)} = C(n)2^{-k}\|\nabla g\|_{\LL^1}. 
    $$
\end{itemize} \ \\
\indent \textit{On measures.} For a finite, possibly vector-valued, Borel measure $\mu$ on $Q_0 $, $E_k \mu$ is defined by duality for bounded Borel $g$  
$$
\int (E_k g) \dif \mu = \int g \dif \, (E_k \mu)
$$
or in other words set by definition
$$
E_k \mu := \sum_{Q \in \mathscr{D}_k} \frac{\mu(Q)}{|Q|}\boldsymbol{1}_Q \Leb,
$$
and $E_k(g \Leb)=(E_k g)\Leb$. The duality formula gives $E_k \mu \wstar \mu$, since
$$
\left|\int \eta \dif\, (E_k\mu-\mu) \right| \leq 2^{-k} \text{lip}(\eta)|\mu|(Q_0).
$$
\indent \textit{On Young measures.} Regard $\nu=(\nu_x,\lambda,\nu_x^\infty)$ as $\HH^*$-valued measures. For all $\Phi \in \HH$
$$
\nu(A)[\Phi]= \int_A \langle \Phi, \nu_x\rangle \dif x + \int_A \langle \Phi^\infty, \nu_x^\infty \rangle \dif \lambda(x).
$$
Hence by approximation the pairing formula $\ddual{\eta\otimes \Phi}{ \nu}= \int \eta \dif \nu[\Phi]$ holds. Now we can define $E_k\nu $ with the same formula for measures
$$
E_k\nu:= \sum_{Q \in \mathscr{D}_k} \frac{\nu(Q)}{|Q|} \boldsymbol{1}_Q \Leb.
$$
To read off the components of the density of $E_k\nu$ on each cube $Q \in \mathscr{D}_k$, again regarding $\nu$ as $\HH^*$-valued measures and by Fubini we can write
\begin{align*}
\frac{1}{|Q|}\nu(Q)[\Phi]= \frac{1}{|Q|}&\int_Q \langle \Phi, \nu_x \rangle \dif x+\frac{1}{|Q|}\int_Q \langle \Phi^\infty , \nu_x^\infty \rangle \dif \lambda(x)\\
&=\langle \Phi, \nu_Q \rangle + \frac{\lambda(Q)}{|Q|}\langle \Phi^\infty, \nu^\infty_Q \rangle
\end{align*}
where (with the meaning of the action on $\Phi$ understood as above)
$$
\nu_Q:=\frac{1}{|Q|}\int_Q \nu_x \dif x, \quad \nu_Q^\infty:= \frac{1}{\lambda(Q)}\int_Q \nu_x^\infty \dif \lambda(x).
$$
\textbf{Note that $\nu_x^\infty$ is averaged against its concentration measure $\lambda$ and not $\Leb$}; if $\lambda(Q)=0$ take any probability measure, since it does not contribute to the calculations. By expanding for every pairing $\eta \otimes \Phi$ we have the expression
$$\ddual{\eta\otimes\Phi}{E_k\nu}= \sum_{Q\in \mathscr{D}_k}\left( \int_Q \eta \langle \Phi, \nu_Q \rangle \dif x + \int_Q \eta \frac{\lambda(Q)}{|Q|} \langle \Phi^\infty, \nu_Q^\infty \rangle \dif x \right).
$$  
Since such pairings determine the Young measure, $E_k\nu$ is the Young measure that equals the homogeneous Young measure $(\nu_Q, \frac{\lambda(Q)}{|Q|}\Leb \Lcorner Q, \nu_Q^\infty )$ on each cube $Q$. The global triple associated to $E_k\nu$, noting that each $x \in \Omega$ belongs to exactly one cube $Q=Q_k(x)$, can be written as
\begin{itemize}
    \item$(E_k\nu)_x= \nu_{Q_k(x)}$, \\
    \item $\lambda^k= E_k \lambda \ll \Leb$, \\
    \item $(E_k\nu)^\infty_x= \nu^\infty_{Q_k(x)}$.
\end{itemize}
Since everything is linear in $\Phi$, the corresponding facts for measures transfer verbatim to Young measures. For $\Phi\in\HH$ denote by $\nu[\Phi]$ the finite signed measure $A\mapsto\nu(A)[\Phi]$, so that $\ddual{\eta\otimes\Phi}{\nu}=\int\eta\dif\nu[\Phi]$ and $(E_k\nu)[\Phi]=E_k(\nu[\Phi])$.
\begin{itemize}
\item \emph{Duality.} For every $\eta$ and every $\Phi\in\HH$
$$
\ddual{\eta\otimes\Phi}{E_k\nu}=\ddual{E_k\eta\otimes\Phi}{\nu}.
$$
This is the identity $\int\eta\dif \, (E_k\mu)=\int E_k\eta\dif\mu$ applied to $\mu=\nu[\Phi]$; explicitly, both sides equal $\sum_{Q\in\mathscr D_k}\eta_Q\,\nu(Q)[\Phi]$, where $\eta_Q=\dashint_Q\eta\dif x$.
\item \emph{Barycentres commute.} $\overline{E_k\nu}=E_k\overline{\nu}$. Indeed, for an orthonormal basis $(e_i)$ of $\V$ the linear functions $z_i=\langle z,e_i\rangle$ belong to $\HH$, and so using the duality above and the one for measures
$$
\int \eta \dif \, ( \overline{E_k \nu}_i) = \ddual{\eta \otimes z_i}{ E_k \nu}= \ddual{E_k\eta \otimes z_i}{\nu}= \int E_k \eta \dif \overline{\nu}_i = \int \eta \dif \,( E_k \overline{\nu} _i).
$$
\item \emph{Tower property.} $E_jE_k=E_{\min\{j,k\}}$. If $j\le k$, every $Q\in\mathscr D_j$ is the disjoint union of the cubes of $\mathscr D_k$ it contains, so $E_k\nu(Q)=\nu(Q)$ and $E_jE_k\nu=E_j\nu$. If $j\ge k$, $E_k\nu$ has constant ($\HH^*$-valued) density with respect to $\Leb$ on each cube of $\mathscr D_k$, hence on each of its dyadic subcubes, so $E_jE_k\nu=E_k\nu$. Thus $(E_k\nu)_k$ is a martingale of piecewise homogeneous Young measures with respect to the dyadic filtration $\mathcal F_k=\sigma(\mathscr D_k)$.
\end{itemize}
\subsection{Jensen's inequality for averages of concentration-angle measures}\ \\
The following lemma is used in the proof of the inhomogenization of the
singular part (Proposition~\ref{inhomogenization singular part}). It
states that any average of the concentration-angle measures
$\nu^\infty_x$ at singular points again satisfies the Jensen inequality.
This is in contrast with the absolutely continuous part, where averages
of measures satisfying the Jensen inequality at different barycentres
need not satisfy it. The proof combines three ingredients: the structure
theorem for $\mathscr{A}$-free measures of De Philippis--Rindler
\cite[Thm.~1.1]{DePhilippisRindler16}, which places the barycentres of
the $\nu^\infty_x$ in the wave cone $\Lambda_{\mathscr{A}}$; the
automatic convexity result of Kirchheim--Kristensen
\cite[Thm.~1.1]{KirchheimKristensen16}, which gives the Jensen inequality
at each singular point; and the inequality $f(z+w)\le f(z)+f^\infty(w)$ for $w\in\Lambda_{\mathscr{A}}$, which passes the inequality to
averages even though their barycentres may leave $\Lambda_{\mathscr{A}}$.
Applied on each dyadic cube $Q$, the lemma shows that the conditional expectation of the mollified Young measure $\nu^s_t$ with respect to the dyadic filtration satisfies the
Jensen inequality on $Q$; this is what makes the dyadic martingale
approach work for the singular part.
\begin{lemma}\label{integral jensen inequality lemma}
Let $\m$ be a finite positive measure in $\Omega$ and $x \mapsto \sigma_x$ a $\m$-measurable family of probability measures on $\SG$ with $\overline{\sigma}_x \in \Lambda$ for $\m$-a.e. $x$. Then for every $f\in\mathbf{Q}$
$$
f \left (\int_\Omega \overline{\sigma}_x \dif \m(x)  \right)\leq f(0) + \int_\Omega \langle f^\infty,\sigma_x\rangle \dif \m(x).
$$
\end{lemma}
\begin{proof}
We divide the proof in steps.
\begin{enumerate}
    \item Iterating $f(z+w)\leq f(z) + f^\infty(w)$ for $w \in \Lambda$ gives
    $$
    f\left(\sum_{i} w_i\right) \leq f(0)+ \sum_{i} f^\infty(w_i),
    $$
    with all $w_i \in \Lambda$.
    \item Approximate $x\mapsto \overline{\sigma}_x$ by $\Lambda$-valued simple functions $\sum_i c_i \boldsymbol{1}_{E_i} $. First of all $\overline{\sigma}_x \in \Lambda$ only $\m$-a.e., so on the exceptional set redefine it  equal to $0 \in \Lambda$ without loss of generality. Then, as $\overline{\sigma}: \Omega \to \V$ is measurable and bounded (all the probability measures $\sigma_x$ are supported on the compact $\SG$), take any simple function $s= \sum_i a_i \boldsymbol{1}_{E_i} $ with $\|s-\overline{\sigma} \|_\infty < \delta$. For each $a_i$ take $c_i \in \Lambda$ such that $|a_i-c_i|=\text{dist}(a_i,\Lambda)$, as $\Lambda$ is a closed set. For all $x \in E_i$ we have $\overline{\sigma}_x \in \Lambda$, so that $\text{dist}(a_i,\Lambda) \leq |a_i-\overline{\sigma}_x|<\delta $. Hence $|c_i-\overline{\sigma}_x|< 2\delta$ and the simple function $\sum_i c_i \boldsymbol{1}_{E_i}$ is $\Lambda$-valued and uniformly $2\delta$-close to $\overline{\sigma}$. This proves the approximation by $\Lambda$-valued simple functions.
    Since $\Lambda$ is a cone, $w_i=\m(E_i)c_i \in \Lambda$ as well, and therefore we can apply (1) to the $w_i=\m(E_i)c_i$ to get
    $$
    f\left(\sum_i \m(E_i)c_i \right)\leq f(0) + \sum_i \m(E_i)f^\infty(c_i).
    $$
    \item Using the Lipschitz continuity of $f$ and $f^\infty$ we can pass to the limit
    $$
    f\left(\int_\Omega \overline{\sigma}_x \dif \m(x) \right) \leq f(0) + \int_\Omega f^\infty (\overline{\sigma}_x)\dif \m(x).
    $$
    Note that changes of $\overline{\sigma}$ on $\m$-null sets affect neither $\int \overline{\sigma}\dif \m$ nor $\int f^\infty(\overline{\sigma})\dif \m$, so the formula above remains unaffected by changes of $\overline{\sigma}$ as done in its approximation by simple functions.
    \item Finish with Kirchheim–Kristensen pointwise: $f^\infty(\overline{\sigma}_x)\leq \langle f^\infty , {\sigma}_x \rangle$ for $\m$-a.e. $x$.
\end{enumerate}    
\end{proof}
\begin{figure}[tbp]
\centering
\resizebox{\textwidth}{!}{%
\begin{tikzpicture}[
  declare function={
    bump(\s) = (abs(\s)<0.5)*(8/3)*cos(180*\s)^4;
    phit(\s) = bump(\s/0.16)/0.16;
    rhoP(\s) = 0.75*phit(\s-0.18)+0.20*phit(\s-0.47)+0.80*phit(\s-0.55)+0.24*phit(\s-0.82);
    rhoM(\s) = 0.25*phit(\s-0.18)+0.60*phit(\s-0.47)+0.20*phit(\s-0.55)+0.36*phit(\s-0.82);
    avgP(\lo,\hi) = (rhoP(\lo+(\hi-\lo)/16)+rhoP(\lo+3*(\hi-\lo)/16)+rhoP(\lo+5*(\hi-\lo)/16)+rhoP(\lo+7*(\hi-\lo)/16)+rhoP(\lo+9*(\hi-\lo)/16)+rhoP(\lo+11*(\hi-\lo)/16)+rhoP(\lo+13*(\hi-\lo)/16)+rhoP(\lo+15*(\hi-\lo)/16))/8;
    avgM(\lo,\hi) = (rhoM(\lo+(\hi-\lo)/16)+rhoM(\lo+3*(\hi-\lo)/16)+rhoM(\lo+5*(\hi-\lo)/16)+rhoM(\lo+7*(\hi-\lo)/16)+rhoM(\lo+9*(\hi-\lo)/16)+rhoM(\lo+11*(\hi-\lo)/16)+rhoM(\lo+13*(\hi-\lo)/16)+rhoM(\lo+15*(\hi-\lo)/16))/8;
  },
  font=\small,
  >={Stealth[length=3pt,width=2.6pt]},
  dim/.style={{Bar[width=4pt]}-{Bar[width=4pt]},line width=0.4pt}
]
\def\Lx{7.2}
\def\Hb{1.0}
\def\Pr{2.45}
\def\kp{0.036}
\def\ka{0.62}
 
\draw[gridc] (0,0) -- (0,\Hb) (\Lx,0) -- (\Lx,\Hb);
\draw[basec] (0,0) -- (\Lx,0) (0,\Hb) -- (\Lx,\Hb);
\draw[samp,line width=2pt] (0,{\Hb/2}) -- (\Lx,{\Hb/2});
\foreach \y/\up/\dn in {0.18/0.75/0.25, 0.47/0.20/0.60, 0.55/0.80/0.20, 0.82/0.24/0.36} {
  \draw[conc!50,densely dotted] ({\Lx*\y},0) -- ({\Lx*\y},\Hb);
  \fill[conc] ({\Lx*\y},{\Hb/2}) circle (1.4pt);
  \draw[conc,line width=1.1pt,-{Stealth[length=2.6pt,width=3.2pt]}] ({\Lx*\y},\Hb) -- ++(0,{\ka*\up});
  \draw[conc,line width=1.1pt,-{Stealth[length=2.6pt,width=3.2pt]}] ({\Lx*\y},0) -- ++(0,{-\ka*\dn});
}
\node[anchor=west,font=\footnotesize] at (\Lx,\Hb) {$+\infty$};
\node[anchor=west,font=\footnotesize] at (\Lx,{\Hb/2}) {$0$};
\node[anchor=west,font=\footnotesize] at (\Lx,0) {$-\infty$};
\node[samp,anchor=south west,font=\footnotesize,inner sep=1.5pt] at (0.05,{\Hb/2+0.05}) {$\delta_0$};
\node[conc,anchor=west,font=\footnotesize,inner sep=1pt] at (0.05,{\Hb+0.25}) {$\lambda^s\,\nu^\infty_y$};
\node[anchor=east] at (-0.35,{\Hb/2}) {$\nu^s$};
 
\pgfmathsetmacro{\yA}{-\Pr}
\fill[conc!25] (0,{\yA+\Hb}) -- plot[domain=0:1,samples=361] ({\Lx*\x},{\yA+\Hb+\kp*rhoP(\x)}) -- (\Lx,{\yA+\Hb}) -- cycle;
\fill[conc!25] (0,\yA) -- plot[domain=0:1,samples=361] ({\Lx*\x},{\yA-\kp*rhoM(\x)}) -- (\Lx,\yA) -- cycle;
\draw[conc,line width=0.6pt] plot[domain=0:1,samples=361] ({\Lx*\x},{\yA+\Hb+\kp*rhoP(\x)});
\draw[conc,line width=0.6pt] plot[domain=0:1,samples=361] ({\Lx*\x},{\yA-\kp*rhoM(\x)});
\draw[gridc] (0,\yA) -- (0,{\yA+\Hb}) (\Lx,\yA) -- (\Lx,{\yA+\Hb});
\draw[basec] (0,\yA) -- (\Lx,\yA) (0,{\yA+\Hb}) -- (\Lx,{\yA+\Hb});
\draw[samp,line width=2pt] (0,{\yA+\Hb/2}) -- (\Lx,{\yA+\Hb/2});
\draw[dim] ({\Lx*0.10},{\yA+0.76}) -- ({\Lx*0.26},{\yA+0.76}) node[right,font=\footnotesize,inner sep=2pt] {$t$};
\node[anchor=east] at (-0.35,{\yA+\Hb/2}) {$\nu^s_t$};
 
\pgfmathsetmacro{\yB}{-2*\Pr}
\foreach \j in {0,...,31} {
  \pgfmathsetmacro{\xl}{\j/32}
  \pgfmathsetmacro{\xr}{(\j+1)/32}
  \pgfmathsetmacro{\rpp}{avgP(\xl,\xr)}
  \pgfmathsetmacro{\rmm}{avgM(\xl,\xr)}
  \fill[conc!72] ({\Lx*\xl+0.012},{\yB+\Hb}) rectangle ({\Lx*\xr-0.012},{\yB+\Hb+\kp*\rpp});
  \fill[conc!72] ({\Lx*\xl+0.012},\yB) rectangle ({\Lx*\xr-0.012},{\yB-\kp*\rmm});
}
\draw[black!70,line width=0.35pt] plot[domain=0:1,samples=361] ({\Lx*\x},{\yB+\Hb+\kp*rhoP(\x)});
\draw[black!70,line width=0.35pt] plot[domain=0:1,samples=361] ({\Lx*\x},{\yB-\kp*rhoM(\x)});
\foreach \i in {0,...,32} \draw[gridc,line width=0.3pt] ({\Lx*\i/32},\yB) -- ({\Lx*\i/32},{\yB+\Hb});
\draw[basec] (0,\yB) -- (\Lx,\yB) (0,{\yB+\Hb}) -- (\Lx,{\yB+\Hb});
\draw[samp,line width=2pt] (0,{\yB+\Hb/2}) -- (\Lx,{\yB+\Hb/2});
\draw[dim] ({\Lx*4/32},{\yB+0.76}) -- ({\Lx*5/32},{\yB+0.76}) node[right,font=\footnotesize,inner sep=2pt] {$2^{-k}$};
\node[anchor=east] at (-0.35,{\yB+\Hb/2}) {$E_k\nu^s_t$};
 
\draw[hlc,line width=0.8pt,rounded corners=2pt] ({\Lx*0.5-0.03},{\yB-0.58}) rectangle ({\Lx*17/32+0.03},{\yB+\Hb+0.58});
\node[hlc,anchor=north,inner sep=2pt] at ({\Lx*33/64},{\yB-0.58}) {$Q$};
 
\pgfmathsetmacro{\rQp}{avgP(0.5,17/32)}
\pgfmathsetmacro{\rQm}{avgM(0.5,17/32)}
\pgfmathsetmacro{\pQ}{\rQp/(\rQp+\rQm)}
\def\xz{1.2}
\def\Wz{3.8}
\def\Hz{1.9}
\pgfmathsetmacro{\zt}{\yB-1.45}
\pgfmathsetmacro{\zb}{\zt-\Hz}
\pgfmathsetmacro{\zm}{(\zt+\zb)/2}
\pgfmathsetmacro{\zQ}{\zm+0.3}
\pgfmathsetmacro{\Xh}{\xz+\Wz+0.65}
 
\draw[hlc,line width=0.8pt,rounded corners=2pt] ({\xz-0.1},{\zb-0.1}) rectangle ({\xz+\Wz+0.1},{\zt+0.1});
\draw[hlc,line width=0.4pt] ({\Lx*0.5-0.03},{\yB-0.58}) -- ({\xz-0.1},{\zt+0.1});
\draw[hlc,line width=0.4pt] ({\Lx*17/32+0.03},{\yB-0.58}) -- ({\xz+\Wz+0.1},{\zt+0.1});
 
\draw[gridc] (\xz,\zb) -- (\xz,\zt) ({\xz+\Wz},\zb) -- ({\xz+\Wz},\zt);
\draw[basec] (\xz,\zb) -- ({\xz+\Wz},\zb) (\xz,\zt) -- ({\xz+\Wz},\zt);
\draw[basec,densely dotted] ({\xz+\Wz},\zt) -- (\Xh,\zt);
\draw[basec,densely dotted] ({\xz+\Wz},\zb) -- (\Xh,\zb);
\draw[samp!60,densely dotted] ({\xz+\Wz},\zm) -- (\Xh,\zm);
\draw[dashed,line width=0.4pt] (\xz,\zQ) -- (\Xh,\zQ);
 
\def\bet{0.28}
\def\w{0.045}
\pgfmathsetmacro{\cz}{\xz+\bet+0.25}
\pgfmathsetmacro{\sp}{(\Wz-2*\bet-0.5)/8}
\draw[line width=0.7pt,line join=round]
  (\xz,\zQ) -- ({\xz+\bet},\zQ) -- ({\xz+\bet+0.08},\zm)
  \foreach \i/\d in {0/1,1/-1,2/1,3/1,4/-1,5/1,6/1,7/-1,8/1} {
    -- ({\cz+\i*\sp-\w},\zm) -- ({\cz+\i*\sp},{\zm+\d*\Hz/2}) -- ({\cz+\i*\sp+\w},\zm)
  }
  -- ({\xz+\Wz-\bet-0.08},\zm) -- ({\xz+\Wz-\bet},\zQ) -- ({\xz+\Wz},\zQ);
\node[anchor=east,font=\footnotesize] at ({\xz-0.25},\zt) {$+\infty$};
\node[anchor=east,font=\footnotesize] at ({\xz-0.25},\zb) {$-\infty$};
\node[anchor=east,font=\footnotesize] at ({\xz-0.25},\zm) {$u_Q=z_Q+\mathscr B\pphi^Q$};
\draw[decorate,decoration={brace,mirror,amplitude=3pt},line width=0.4pt]
  ({\xz+\bet},{\zb-0.22}) -- ({\xz+\Wz-\bet},{\zb-0.22})
  node[midway,below=3pt,font=\footnotesize] {$\operatorname{supp}\pphi^Q\Subset Q$};
 
\draw[basec] (\Xh,\zb) -- (\Xh,\zt);
\fill[samp] (\Xh,{\zm-0.045}) rectangle ({\Xh+0.85},{\zm+0.045});
\fill[conc] (\Xh,{\zt-0.05}) rectangle ({\Xh+1.6*\pQ},{\zt+0.05});
\fill[conc] (\Xh,{\zb-0.05}) rectangle ({\Xh+1.6*(1-\pQ)},{\zb+0.05});
\fill (\Xh,\zQ) -- ++(-0.13,0.065) -- ++(0,-0.13) -- cycle;
\node[right,inner sep=2pt] at (\Xh,\zQ) {$z_Q$};
\node[anchor=south west,inner sep=2pt] at ({\Xh-0.15},{\zt+0.12}) {$\sigma_Q=(\delta_0,r_Q,\nu^\infty_Q)$};
 
\def\xa{8.25}
\begin{scope}[font=\footnotesize,every node/.style={anchor=west,align=left,inner sep=0pt}]
  \draw[black!45,line width=0.6pt,->] (\xa,{\Hb/2}) -- (\xa,{\yA+\Hb/2+0.05});
  \node at ({\xa+0.25},{(\Hb+\yA+\Hb)/2-\Hb/2}) {\textit{mollify} at scale $t$\\[1pt]
    error $\lambda^s(\Omega\setminus K)+Ct\,\lambda^s(\Omega)$};
  \draw[black!45,line width=0.6pt,->] (\xa,{\yA+\Hb/2-0.05}) -- (\xa,{\yB+\Hb/2+0.05});
  \node at ({\xa+0.25},{(\yA+\yB)/2+\Hb/2}) {\textit{average} on $Q\in\mathscr D_k$, $2^{-k}\ll t$\\[1pt]
    error $C\,2^{-k}\lambda^s(\Omega)$\\[1pt]
    $\sigma_Q$ is Jensen (Lemma): no sampling};
  \draw[black!45,line width=0.6pt,->] (\xa,{\yB+\Hb/2-0.05}) -- (\xa,\zm);
  \node at ({\xa+0.25},{(\yB+\Hb/2+\zm)/2}) {\textit{generate} on each $Q$\\[1pt]
    $\|\boldsymbol\varepsilon_{u_Q}-\sigma_Q\|_{\mathrm K}<\varepsilon$\\[1pt]
    barycenter: $\|w_t-E_kw_t\|_{L^1}\le C\frac{2^{-k}}{t}\lambda^s(K)$};
\end{scope}
\end{tikzpicture}}
\caption{Schematic of the proof of
Proposition~\ref{inhomogenization singular part} ($n=1$; the vertical axis is
the compactified fibre). The singular part $\nu^s$ is mollified at scale $t$
and then averaged on the dyadic cubes $Q\in\mathscr{D}_k$, $2^{-k}\ll t$. On
each cube the average $\sigma_Q$ satisfies the Jensen inequality by
Lemma~\ref{integral jensen inequality lemma}, so no sampling is needed, and it
is generated by $u_Q=z_Q+\mathscr{B}\pphi^Q$ with
$\operatorname{supp}\pphi^Q\Subset Q$.}
\label{fig:inhom-singular}
\end{figure}
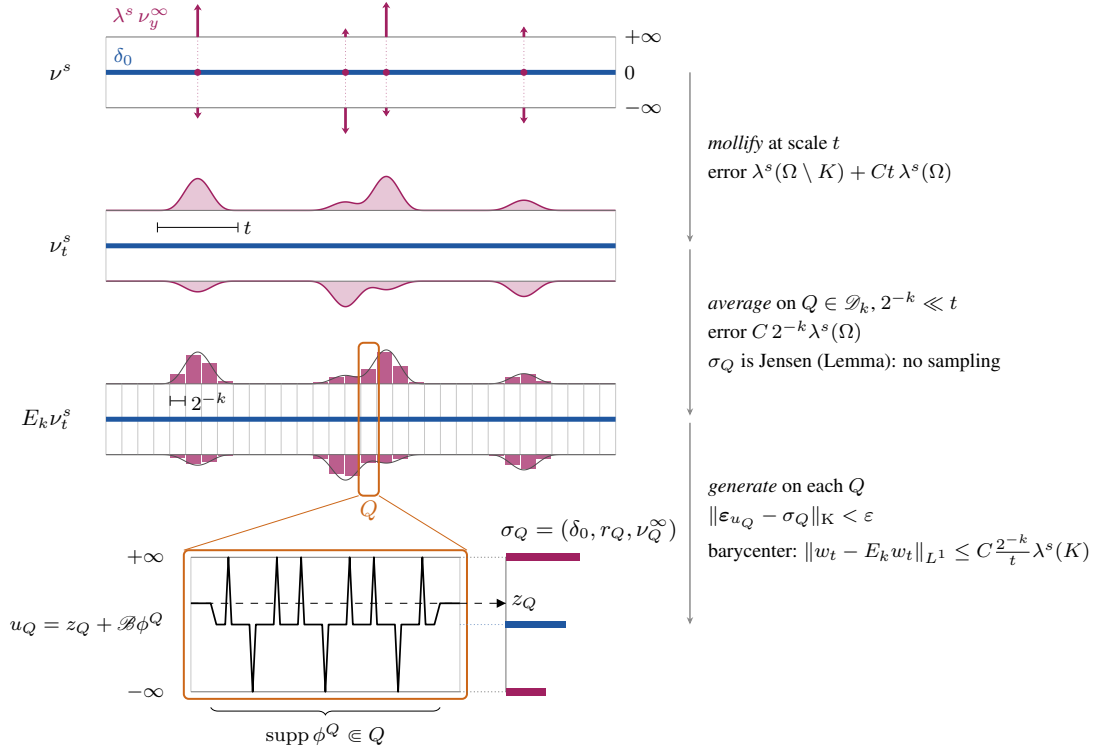
\begin{proof}[Proof of Proposition \ref{inhomogenization singular part}] We prove the singular part of the inhomogenization, sketched in Figure~\ref{fig:inhom-singular}. Set the notation:
\begin{itemize}
    \item[$\bullet$] $d(\cdot,\cdot)$ is the sup distance over the test class $(\| \eta \|_{\text{LIP}} \leq 1 \ \text{and} \ \|T \Phi\|_{\text{LIP}}\leq 1)$
    \item[$\bullet$] $\nu^s=(\delta_0,\lambda^s,\nu_x^\infty)$ and $\boldsymbol{\eps}_u$ is the elementary Young measure associated to $u$.
    \item[$\bullet$] $g_Q$ is the average of $g$ over $Q\in \mathscr{D}_k$, and $E_k$ is the dyadic conditional expectation at generation $k$. In our case $k=p+q$, with $p,q \in \N$.
\end{itemize}
Fix a cube $Q_0 \supset \Omegacl $ so that cubes of generation $k$ fill exactly $Q_0$; all the Young measures throughout the following proof are extended tacitly to $(\delta_0,0)$ on $Q_0\setminus \Omega$ and the function $\eta$ extended as a Lipschitz function with same Lipschitz constant; the extra term carried $\int_{Q_0 \setminus \Omega} \eta \Phi(0) \dif x$ is the same for every Young measure being compared (with $\psi^s$ extended by 0) so it cancels out in each term $d(\cdot, \cdot)$. Therefore each term $d(\cdot,\cdot)$ is the same whether the pairings are computed on
$Q_0$ or on $\Omega$, and below we compute them on $\Omega$. \\
Take a compact set $K \subset \Omega$ such that $ \lambda^s(\Omega \setminus K) < \eps$ and $\Leb(K)=0$. This is just inner regularity of $\lambda^s$ together with $\lambda^s \perp \Leb$ and $\lambda^s(\partial \Omega)=0$. No Lusin is needed nor continuity of $ x \mapsto \nu_x^\infty$. \\
Fix $p\in\N$ so large that $t:=2^{-p}l(Q_0)=2^{-p}$ (normalizing $l(Q_0)$ to 1 after a dilation) satisfies
\begin{equation}\label{t small wrt K}
\tfrac32\, t<\operatorname{dist}(K,\mathbb{R}^n\setminus\Omega),
\end{equation}
which is possible because $K\subset\Omega$ is compact. Since $k=p+q$, the side of the
cubes of $\mathscr{D}_k$ is $l_k=2^{-q}t\le t$.
Mollify the Young measure:
\begin{equation}\label{definition of mu t, rho t, w t}
\mu_t(x):= \int_K \phi_t(x-y)\nu_y^\infty \dif \lambda^s(y), \quad \rho_t:= \phi_t \ast (\lambda^s \Lcorner K), \quad w_t:=\phi_t \ast (\overline{\nu}^\infty \lambda^s \Lcorner K) 
\end{equation}
and set $\nu_t^s=(\delta_0, \rho_t \Leb,\mu_t/ \rho_t )$. We emphasize the meaning of the parametrized measure $\mu_t(x)$ 
$$
\int \eta(x) \langle \Phi^\infty, \mu_t(x) \rangle \dif x =\int \eta(x) \int_K \left( \phi_t(x-y) \int_{\SG} \Phi^\infty \dif \nu_y^\infty \right)  \dif \lambda^s(y) \dif x,
$$
and the full action of $\nu_t^s$
$$
\ddual{\eta \otimes \Phi}{\nu_t^s}=\int \eta \Phi(0) \dif x + \int \eta(x) \langle \Phi^\infty, \mu_t(x) \rangle \dif x.
$$
Note that $\rho_t$ is the mass of $\mu_t$ (tested with $\Phi^\infty \equiv 1$), and $w_t$ is its first moment (tested with linear $\Phi$). Since $\operatorname{supp}\phi_t=t\overline{\XX}$, by (\ref{t small wrt K}) we have
$\operatorname{supp}\rho_t\cup\operatorname{supp}w_t\subset K+t\overline{\XX}\subset\Omega$;
in particular $\int_{Q_0}\rho_t\dif x=\lambda^s(K)$.\\
On each $Q \in \mathscr{D}_k$, the conditional expectation $E_k\nu_t^s$ equals the homogeneous Young measure 
$$
\sigma_Q=(\delta_0, r_Q \Leb \Lcorner Q, \sigma_Q^\infty ),  
$$ 
where by definition $r_Q:=\dfrac{(\rho_t \Leb)(Q)}{|Q|}=(\rho_t)_Q$ and $\sigma_Q^\infty$ is the average of $(\nu_t^s)_x^\infty= \mu_t(x)/\rho_t(x)$ against its concentration measure $\lambda_t^s=\rho_t \Leb$. Explicitly 
\begin{align*}
  &\bullet \ r_Q=(\rho_t)_Q=\frac{1}{|Q|}\int_Q \rho_t(x) \dif x =  \int_K \left(\frac{1}{|Q|}\int_Q \phi_t(x-y)\dif x \right) \dif \lambda^s(y)=m_Q(K),\\
   &\qquad \text{where}\quad \kappa_Q(y):= \frac{1}{|Q|}\int_Q \phi_t(x-y)\dif x, \quad m_Q:=\kappa_Q \lambda^s\Lcorner K;  \\
   &\bullet \ \sigma_Q^\infty=\frac{1}{\lambda^s_t(Q)}\int_Q (\nu_t^s)_x^\infty \dif \lambda_t^s (x)=\frac{1}{(\rho_t)_Q|Q|}\int_Q \mu_t(x) \dif x =: \frac{(\mu_t)_Q}{(\rho_t)_Q}. 
\end{align*}
The barycenter of $(\delta_0, r_Q, \sigma_Q^\infty)$ on each cube $Q$ is equal to $z_Q:=(w_t)_Q=r_Q \overline{\sigma}_Q^\infty$. 
Globally 
$$
E_k\nu_t^s=(\delta_0, (E_k\rho_t)\Leb, x \mapsto \sigma_{Q(x)}^\infty ). 
$$
Let
$$
\mathscr{D}_k^s:=\{Q\in\mathscr{D}_k:\ r_Q>0\}.
$$
If $Q\notin\mathscr{D}_k^s$, then $\rho_t\equiv0$ on $Q$, since $\rho_t$ is continuous
and nonnegative; as $|\overline{\nu}_y^\infty|\le c^{-1}$ (because $\overline{\nu}_y^\infty\in\overline{\mathrm{co}}(\SG)$) gives $|w_t|\le c^{-1}\rho_t$, also
$w_t\equiv0$ and $\mu_t\equiv0$ on $Q$. Thus on such cubes $z_Q=0$ and
$E_k\nu_t^s=(\delta_0,0)$.\\
We now apply Lemma \ref{integral jensen inequality lemma} (for each cube $Q$) with the positive measure $m_Q(\dif y)$, which satisfies $m_Q \ll \lambda^s$ by its definition. Hence the barycentres $\overline{\nu}_y^\infty \in \Lambda$ for $m_Q$-a.e. $y$, since it is true $\lambda^s$-a.e., and likewise for the $m_Q$-measurability of $y \mapsto \nu_y^\infty$. The quantities involved in the statement of the lemma are
$$
\int_K \overline{\nu}_y^\infty \dif m_Q(y)= \frac{1}{|Q|}\int_Q \left(\int_K\phi_t(x-y) \overline{\nu}_y^\infty \dif \lambda^s(y) \right) \dif x=(w_t)_Q=z_Q,
$$
$$
\int_K \langle f^\infty , \nu_y^\infty \rangle \dif m_Q(y)= \frac{1}{|Q|}\int_Q \langle f^\infty, \mu_t(x) \rangle \dif x = \langle f^\infty, (\mu_t)_Q\rangle = r_Q \langle f^\infty, \sigma_Q^\infty \rangle. 
$$
The Lemma gives
$$
f(z_Q)= f\Big(\int_K \overline{\nu}_y^\infty \dif m_Q(y)\Big)
\leq f(0) + \int_K \langle f^\infty , \nu_y^\infty \rangle \dif m_Q(y)
= f(0) + r_Q \langle f^\infty, \sigma_Q^\infty \rangle.
$$
This is Jensen for $(\delta_0, r_Q, \sigma_Q^\infty)$ whose barycenter is $z_Q$. In other words, $E_k \nu_t^s$ is Jensen on each cube $Q$. For $Q\in\mathscr{D}_k^s$, Proposition \ref{homogeneous case characterization} applied
with the homogeneous Young measure $(\delta_0, r_Q ,\sigma_Q^\infty)$ and barycentre
$z_Q$ gives $\pphi^Q \in \CC_c^\infty(Q,\V)$ with
$\|\pphi^Q\|_{\WW^{l-1,1}(Q,\V)}< \eps\, r_Q|Q|$, a positive number since $r_Q>0$,
such that (realizing the elementary Young measure on $Q$ rather than on $\XX$)
$$
\|(\delta_0, r_Q \sigma_{Q}^\infty) -\boldsymbol{\eps}_{r_Q\overline{\sigma}_{Q}^\infty + \mathscr{B}\pphi^Q} \|_{\text{K}} <\eps,
$$
(denote $(\delta_0, r_Q \sigma_{Q}^\infty)\equiv (\delta_0, r_Q\Leb\Lcorner Q, \sigma_{Q}^\infty)$ ) which amounts in terms of the Kantorovich metric to 
\begin{equation}\label{phi^Q kantorovich distance}
\sup_{\|T\Phi\|_{\text{LIP}}\leq 1} \bigg|\Phi(0)+ r_Q \langle \Phi^\infty, \sigma_Q^\infty \rangle - \dashint_Q \Phi(r_Q \overline{\sigma}_{Q}^\infty + \mathscr{B}\pphi^Q)\dif x \bigg|< \eps. 
\end{equation}
For $Q\notin\mathscr{D}_k^s$ we set $\pphi^Q:=0$; since there $r_Q=0$ and $z_Q=0$, the
left-hand side of (\ref{phi^Q kantorovich distance}) vanishes identically, so
(\ref{phi^Q kantorovich distance}) holds for every $Q\in\mathscr{D}_k$.

Every cube of $\mathscr{D}_k^s$ satisfies $\overline Q\subset\Omega$. Indeed, if $r_Q>0$
there is $x\in Q$ with $\rho_t(x)>0$, hence, since
$\operatorname{supp}\phi_t=t\overline{\XX}$, some $y\in K$ with $\|x-y\|\le t/2$.
For every $x'\in\overline Q$,
$$
\|x'-y\|\le\|x'-x\|+\|x-y\|\le l_k+\tfrac t2\le\tfrac32\,t
<\operatorname{dist}(K,\mathbb{R}^n\setminus\Omega)
$$
by (\ref{t small wrt K}), so that $x'\in\Omega$.

We define
$$
\pphi:=\sum_{Q\in\mathscr{D}_k^s}\pphi^Q .
$$
This is a finite sum of smooth functions supported in the pairwise disjoint cubes
$Q$, whose closures lie in $\Omega$. Hence
$\pphi\in\CC_c^\infty(\Omega,\V)$, $\mathscr{B}\pphi=\mathscr{B}\pphi^Q$ a.e.\ on every
$Q\in\mathscr{D}_k$ (in particular $\mathscr{B}\pphi=0$ on the cubes outside
$\mathscr{D}_k^s$), and
$$
\|\pphi\|_{\WW^{l-1,1}(\Omega,\V)}=\sum_{Q\in\mathscr{D}_k^s}\|\pphi^Q\|_{\WW^{l-1,1}(Q,\V)}
<\eps\sum_{Q\in\mathscr{D}_k}r_Q|Q|=\eps\int_{Q_0}\rho_t\dif x=\eps\,\lambda^s(K)
\le\eps\,\lambda^s(\Omega).
$$ We check that the map
\begin{equation}\label{psi^s map definition Martinglale}
\psi^s:=\phi_t \ast (\overline{\nu}_{(\cdot)}^\infty \lambda^s \Lcorner \Omega)+ \mathscr{B}\pphi
\end{equation}
has the required properties, namely $d(\nu^s,\boldsymbol{\eps}_{\psi^s})< \eps$. 
Let $S:=\bigcup_{Q\in\mathscr{D}_k^s}Q\subset\Omega$. On every cube outside
$\mathscr{D}_k^s$ the functions $\rho_t$, $w_t$ and $\mu_t$ vanish, hence so do $E_k\rho_t$,
$E_kw_t$ and $\pphi$. Since these cubes cover $\Omega\setminus S$, the Young measures
$\nu_t^s$, $E_k\nu_t^s$ and $\boldsymbol{\eps}_{E_kw_t+\mathscr{B}\pphi}$ all equal
$(\delta_0,0)$ on $\Omega\setminus S$, and in the differences of their pairings below the
contributions of $\Omega\setminus S$ cancel. The Young measures $\nu^s$ and
$\boldsymbol{\eps}_{\psi^s}$ need not be trivial on $\Omega\setminus S$, but they only
enter the mollification and barycenter steps, where they are first compared with the
objects above.\\
Write
$$
d(\nu^s,\boldsymbol{\eps}_{\psi^s})\leq \underbrace{d(\nu^s,\nu^s_t)}_{\text{mollification}}+\underbrace{d(\nu_t^s,E_k \nu^s_t)}_{\text{martingale}}+ \underbrace{d(E_k \nu^s_t, \boldsymbol{\eps}_{E_k w_t + \mathscr{B}\pphi})}_{\text{homogeneous generation}}+\underbrace{d(\boldsymbol{\eps}_{E_k w_t + \mathscr{B}\pphi}, \boldsymbol{\eps}_{\psi^s})}_{\text{barycenter}}.
$$
Each term is the error introduced by one step of the construction, and its estimate reflects that step; we therefore treat the four terms in separate paragraphs, named after the corresponding steps. In each case we bound $|\ddual{\eta \otimes \Phi}{\cdot}|$ uniformly in the admissible pair $(\eta,\Phi)$; of $\eta$ we only use $\|\eta\|_\infty\leq 1$ and $\text{lip}(\eta)\leq 1$.  \\
\textbf{Mollification.} Recall the action $\ddual{\eta \, \otimes \, \Phi}{\nu_t^s}=\int \eta \Phi(0) \dif x + \int \eta \langle \Phi^\infty, \mu_t(x) \rangle \dif x$, and $\nu^s=(\delta_0,\lambda^s,\nu_x^\infty)$. The term $\int \eta \Phi(0) \dif x$ is in common between $\nu^s_t$ and $\nu^s$. 
 By Fubini,
$$
\int_\Omega \eta(x)\langle \Phi^\infty, \mu_t(x)\rangle \dif x
=\int_K \Big(\int_\Omega \phi_t(x-y)\eta(x)\dif x\Big)\langle \Phi^\infty,\nu_y^\infty\rangle \dif\lambda^s(y).
$$
For $y\in K$ the support $y+t\overline{\XX}$ of $\phi_t(\cdot-y)$ lies in $\Omega$ by
(\ref{t small wrt K}), so 
\begin{multline*}
\left| \int_\Omega \eta(y) \langle \Phi^\infty,\nu_y^\infty\rangle \dif\lambda^s(y)-\int_\Omega \eta(x)\langle \Phi^\infty, \mu_t(x)\rangle \dif x\right|\\
=\left|\int_{\Omega \setminus K} \eta(y) \langle \Phi^\infty,\nu_y^\infty\rangle \dif\lambda^s(y)+\int_K \left(\int_\Omega \phi_t(x-y)(\eta(y)-\eta(x)) \dif x \right) \langle \Phi^\infty,\nu_y^\infty\rangle \dif\lambda^s(y)\right|, 
\end{multline*}
where we used $\int_\Omega\phi_t(x-y)\dif x=1$ for $y\in K$. 
Since $|\int_\Omega \phi_t(x-y)(\eta(y)-\eta(x)) \dif x|\leq t/2$, $|\langle\Phi^\infty,\nu_y^\infty\rangle|\le1$  and $\lambda^s(\Omega\setminus K)<\eps$,
we obtain
$$
d(\nu^s,\nu_t^s)\le \tfrac t2\,\lambda^s(K)+\lambda^s(\Omega\setminus K)< t\,\lambda^s(\Omega)+\eps .
$$
\textbf{Martingale.} By the duality formula of $E_k$ applied to Young measures its action on $\nu_t^s$ can be transferred to $\eta$, so that
$$
\ddual{\eta \otimes \Phi}{\nu_t^s}-\ddual{\eta \otimes \Phi}{E_k\nu_t^s}=\int_{\Omega}(\eta -E_k \eta) \langle \Phi^\infty, \mu_t(x)\rangle \dif x.
$$
For Lipschitz $\eta$, $\sup_\Omega|\eta - E_k \eta| \leq 2^{-k} \text{lip}(\eta) $ holds everywhere, and $|\langle\Phi^\infty,\mu_t\rangle|\le\rho_t$, so the term is at most
$$
\bigg|\int_{\Omega}(\eta -E_k \eta) \langle \Phi^\infty, \mu_t(x)\rangle \dif x\bigg | \leq 2^{-k} \int_\Omega\rho_t\dif x\le 2^{-k}\lambda^s(\Omega).
$$
\textbf{Homogeneous generation.} Explicitly
$$
\ddual{\eta \otimes \Phi}{E_k\nu_t^s}
= \int_{\Omega\setminus S}\eta\Phi(0)\dif x
+ \sum_{Q \in \mathscr{D}_k^s}\Big(\int_Q \eta\dif x\Big)\big(\Phi(0)+r_Q \langle \Phi^\infty, \sigma_Q^\infty \rangle\big),
$$
$$
\ddual{\eta \otimes \Phi}{\boldsymbol{\eps}_{E_k w_t + \mathscr{B}\pphi}}= \int_\Omega \eta \Phi(E_k w_t + \mathscr{B}\pphi) \dif x = \int_{\Omega\setminus S} \eta \Phi(0)\dif x+ \sum_{Q \in \mathscr{D}^s_k}\int_Q \eta \Phi( (w_t)_Q+ \mathscr{B}\pphi^Q)\dif x.
$$

Recall $z_Q=(w_t)_Q=r_Q \overline{\sigma}_Q^\infty$. Add and subtract $\eta_Q=\dashint_Q \eta$ to the second expression above, 
\begin{multline*}
 \sum_{Q \in \mathscr{D}^s_k}  \int_Q \eta \Phi( r_Q \overline{\sigma}_Q^\infty + \mathscr{B}\pphi^Q)\dif x\\=\sum_{Q \in \mathscr{D}^s_k}\left(\int_Q \left(\eta -\eta_Q \right)\Phi( r_Q \overline{\sigma}_Q^\infty + \mathscr{B}\pphi^Q)\dif  x  + \left(\int_Q \eta\right)\dashint_Q \Phi( r_Q \overline{\sigma}_Q^\infty + \mathscr{B}\pphi^Q)\dif  x \right).
\end{multline*}
Therefore
\begin{multline*}
\ddual{\eta \otimes \Phi}{E_k\nu_t^s}-\ddual{\eta \otimes \Phi}{\boldsymbol{\eps}_{E_k w_t + \mathscr{B}\pphi}} \\=    \sum_{Q \in \mathscr{D}^s_k} \left(\int_Q \eta\right)\left( \Phi(0)+ r_Q \langle \Phi^\infty, \sigma_Q^\infty \rangle - \dashint_Q \Phi(r_Q \overline{\sigma}_{Q}^\infty + \mathscr{B}\pphi^Q)\dif x\right) \\
+\sum_{Q \in \mathscr{D}^s_k}\int_Q \left(\eta_Q -\eta \right)\Phi( r_Q \overline{\sigma}_Q^\infty + \mathscr{B}\pphi^Q)\dif  x=:I_1+I_2.
\end{multline*}
By the homogeneous generation we have $\sup_{\Phi}|\Phi(0)+ r_Q \langle \Phi^\infty, \sigma_Q^\infty \rangle - \dashint_Q \Phi(r_Q \overline{\sigma}_{Q}^\infty + \mathscr{B}\pphi^Q)\dif x|< \eps$, hence
$$
|I_1|\leq \eps \sum_{Q \in \mathscr{D}^s_k} \left|\int_Q \eta \right| \leq \eps \Leb(\Omega).
$$
Moreover, using $\int_Q(\eta-\eta_Q)\dif x=0$, $\text{lip}(\Phi)\leq L$ (see \eqref{lip constant L}) and $\dashint_Q|u|\dif x\leq(1+c^{-1})\|\boldsymbol{\eps}_{u}\|_{\text{K}}$ for $u\in\LL^1(Q,\V)$ (as $(1+c^{-1})^{-1}|\cdot|$ is admissible), with constants that will be absorbed by a constant $C$,
\begin{multline*}
|I_2|\leq \sum_{Q \in \mathscr{D}^s_k} \bigg|\int_Q\left(\eta -\eta_Q \right)\Phi(r_Q \overline{\sigma}_{Q}^\infty + \mathscr{B}\pphi^Q)\dif x \bigg| =\sum_Q \bigg|\int_Q\left(\eta -\eta_Q \right)\big(\Phi(r_Q \overline{\sigma}_{Q}^\infty + \mathscr{B}\pphi^Q)-\Phi(0)\big)\dif x \bigg|\\ \leq \sum_Q  \text{lip}(\eta)2^{-k}L(1+c^{-1})\|\boldsymbol{\eps}_{r_Q\overline{\sigma}_{Q}^\infty + \mathscr{B}\pphi^Q} \|_{\text{K}} \Leb(Q) \leq C2^{-k} \sum_Q (\|(\delta_0, r_{Q},\sigma_{Q}^\infty)\|_{\text{K}}+\eps) \Leb(Q) \\ \leq C2^{-k}\sum_Q(1+r_Q+\eps)\Leb(Q)= C2^{-k} \sum_Q \left((1+\eps)\Leb(Q)+ \int_Q \phi_t \ast (\lambda^s\Lcorner K) \dif x\right) \\
\leq C2^{-k}\left((1+\eps)\Leb(\Omega)+\int_\Omega \phi_t \ast (\lambda^s \Lcorner K) \dif x \right) \leq 2^{-k}C((1+\eps)\Leb(\Omega)+\lambda^s(\Omega)).
\end{multline*}
where in the second inequality we used the triangle inequality with (\ref{phi^Q kantorovich distance}).\\
\textbf{Barycenter.} We estimate $d(\boldsymbol{\eps}_{E_k w_t + \mathscr{B}\pphi}, \boldsymbol{\eps}_{\psi^s})$ where $\psi^s=\phi_t \ast (\overline{\nu}_{(\cdot)}^\infty \lambda^s \Lcorner \Omega)+ \mathscr{B}\pphi$. First of all we can reduce from $\overline{\nu}_{(\cdot)}^\infty \lambda^s\Lcorner \Omega$
to $\overline{\nu}_{(\cdot)}^\infty \lambda^s \Lcorner K$, because by
(\ref{conditions on eta Phi}), $\text{lip}(\Phi)\leq L$ and $|\overline{\nu}_y^\infty|\leq c^{-1}$,
$$
\bigg|\int_\Omega \eta\, \Phi(\psi^s) \dif x
-\int_\Omega \eta\, \Phi\big(w_t+ \mathscr{B}\pphi\big) \dif x\bigg|
\leq \tfrac{L}{c} \int_\Omega \phi_t \ast (\lambda^s \Lcorner \Omega \setminus K)\dif x
\leq \tfrac{L}{c}\lambda^s(\Omega \setminus K) < \tfrac{L}{c}\eps .
$$
Since $w_t$, $E_kw_t$ and $\mathscr{B}\pphi$ vanish on $\Omega\setminus S$, the remaining
comparison only involves the cubes of $\mathscr{D}_k^s$, and the sums below run over
$Q\in\mathscr{D}_k^s$.
 Note that $w_t=\phi_t \ast (\overline{\nu}^\infty \lambda^s \Lcorner K)$, hence by Poincaré inequality on each cube $Q$, with constant $C(n)2^{-k}$, 
\begin{multline*}
\left|\sum_Q \left( \int_Q \eta \Phi(E_kw_t+ \mathscr{B}\pphi^Q )\dif x    -\int_Q \eta \Phi(\phi_t \ast (\overline{\nu}_{(\cdot)}^\infty \lambda^s \Lcorner K)+ \mathscr{B}\pphi^Q  )\dif x\right) \right|\\
\leq \sum_Q \int_Q L|E_k w_t -w_t|\dif x \leq C2^{-k}\sum_Q \int_Q |\nabla w_t|\dif x \\=C2^{-k}\sum_Q \int_Q \left| \int_K \nabla\phi_t(x-y)\overline{\nu}_y^\infty \dif \lambda^s(y)  \right|\dif x \\ \leq C M  \frac{2^{-k}}{t}\sum_Q \int_Q \int_K (\boldsymbol{1}_\XX)_t(x-y) \dif \lambda^s(y) \dif x \\ =   C M  \frac{2^{-k}}{t} \int_\Omega  (\boldsymbol{1}_\XX)_t \ast (\lambda^s\Lcorner K) \dif x\leq C M  \frac{2^{-k}}{t} \lambda^s(K).
\end{multline*}
We recall that $\phi_t(x)= t^{-n}\phi(x/ t)$, and denoting $M=\max|\nabla \phi|$, we used the pointwise bound $|\nabla \phi_t|\leq M t^{-1}(\boldsymbol{1}_\XX)_t$. In this last step lies the use of a two-scale structure $k=p+q$, $2^{-k}/t=2^{-q}$.
\textbf{Conclusion.} Collecting the four estimates and using $2^{-k}=l_k\le t$, we obtain
$$
d(\nu^s,\boldsymbol{\eps}_{\psi^s})\leq C\,(\eps+t+2^{-q}),
$$
with $C$ depending only on $n$, $\phi$, $L$, $c$, $\Leb(\Omega)$ and $\lambda^s(\Omega)$. The
parameters are chosen in this order: $\eps$ and $K$; then $p$, so large that
(\ref{t small wrt K}) holds and $t\le\eps$; then $q$, so large that $2^{-q}\le\eps$;
finally the maps $\pphi^Q$, which depend on $t$ and $k=p+q$. This gives
$d(\nu^s,\boldsymbol{\eps}_{\psi^s})\leq 3C\eps$, and running the construction with
$\eps/(3C)$ in place of $\eps$ yields $d(\nu^s,\boldsymbol{\eps}_{\psi^s})<\eps$, while
the bound on $\|\pphi\|_{\WW^{l-1,1}}$ only improves.
\end{proof}

\begin{figure}[tbp]
\centering
\resizebox{\textwidth}{!}{%
\begin{tikzpicture}[
  declare function={
    fa(\s)  = 0.42 + 0.20*sin(360*(0.85*\s + 0.05));
    fb(\s)  = 1.20 + 0.22*cos(360*(1.15*\s + 0.10));
    fth(\s) = 0.50 + 0.22*sin(360*(0.70*\s + 0.15));
    amin(\lo,\hi) = min(fa(\lo),fa(\lo+(\hi-\lo)/8),fa(\lo+2*(\hi-\lo)/8),fa(\lo+3*(\hi-\lo)/8),fa(\lo+4*(\hi-\lo)/8),fa(\lo+5*(\hi-\lo)/8),fa(\lo+6*(\hi-\lo)/8),fa(\lo+7*(\hi-\lo)/8),fa(\hi));
    amax(\lo,\hi) = max(fa(\lo),fa(\lo+(\hi-\lo)/8),fa(\lo+2*(\hi-\lo)/8),fa(\lo+3*(\hi-\lo)/8),fa(\lo+4*(\hi-\lo)/8),fa(\lo+5*(\hi-\lo)/8),fa(\lo+6*(\hi-\lo)/8),fa(\lo+7*(\hi-\lo)/8),fa(\hi));
    bmin(\lo,\hi) = min(fb(\lo),fb(\lo+(\hi-\lo)/8),fb(\lo+2*(\hi-\lo)/8),fb(\lo+3*(\hi-\lo)/8),fb(\lo+4*(\hi-\lo)/8),fb(\lo+5*(\hi-\lo)/8),fb(\lo+6*(\hi-\lo)/8),fb(\lo+7*(\hi-\lo)/8),fb(\hi));
    bmax(\lo,\hi) = max(fb(\lo),fb(\lo+(\hi-\lo)/8),fb(\lo+2*(\hi-\lo)/8),fb(\lo+3*(\hi-\lo)/8),fb(\lo+4*(\hi-\lo)/8),fb(\lo+5*(\hi-\lo)/8),fb(\lo+6*(\hi-\lo)/8),fb(\lo+7*(\hi-\lo)/8),fb(\hi));
  },
  font=\small,
  >={Stealth[length=3pt,width=2.6pt]}
]
\def\L{7.2}
\def\H{1.7}
\def\G{0.55}
\def\tk{0.2}
\def\ins{0.035}
\def\pad{0.08}

\draw[basec] (0,0) -- (\L,0) node[right,text=black] {$Q_0$};
\draw[basec,->] (0,0) -- (0,{\H+0.2}) node[above,text=black] {$\mathbb V$};
\fill[fieldc] plot[domain=0:1,samples=160] ({\L*\x},{fa(\x)+0.5*\tk*fth(\x)})
  -- plot[domain=1:0,samples=160] ({\L*\x},{fa(\x)-0.5*\tk*fth(\x)}) -- cycle;
\fill[fieldc] plot[domain=0:1,samples=160] ({\L*\x},{fb(\x)+0.5*\tk*(1-fth(\x))})
  -- plot[domain=1:0,samples=160] ({\L*\x},{fb(\x)-0.5*\tk*(1-fth(\x))}) -- cycle;
\node[anchor=east] at (-0.35,{0.5*\H}) {$x\mapsto\nu^a_x$};

\foreach \k/\fr in {1/{0.64,0.28},
                    2/{0.30,0.72,0.55,0.20},
                    3/{0.42,0.75,0.18,0.62,0.33,0.58,0.47,0.25}}{
  \pgfmathsetmacro{\yb}{-\k*(\H+\G)}
  \pgfmathtruncatemacro{\N}{2^\k}
  \pgfmathsetmacro{\sd}{1/\N}
  \foreach \f [count=\j from 0] in \fr {
    \pgfmathsetmacro{\xl}{\j*\sd}
    \pgfmathsetmacro{\xr}{(\j+1)*\sd}
    \fill[bandc] ({\L*\xl+\ins},{\yb+amin(\xl,\xr)-\pad}) rectangle ({\L*\xr-\ins},{\yb+amax(\xl,\xr)+\pad});
    \fill[bandc] ({\L*\xl+\ins},{\yb+bmin(\xl,\xr)-\pad}) rectangle ({\L*\xr-\ins},{\yb+bmax(\xl,\xr)+\pad});
  }
  \foreach \i in {0,...,\N} \draw[gridc] ({\L*\i/\N},\yb) -- ({\L*\i/\N},{\yb+\H});
  \draw[basec] (0,\yb) -- (\L,\yb);
  \draw[fieldthin] plot[domain=0:1,samples=120] ({\L*\x},{\yb+fa(\x)});
  \draw[fieldthin] plot[domain=0:1,samples=120] ({\L*\x},{\yb+fb(\x)});
  \foreach \f [count=\j from 0] in \fr {
    \pgfmathsetmacro{\xl}{\j*\sd}
    \pgfmathsetmacro{\xr}{(\j+1)*\sd}
    \pgfmathsetmacro{\xq}{\xl+\f*\sd}
    \pgfmathsetmacro{\aq}{fa(\xq)}
    \pgfmathsetmacro{\bq}{fb(\xq)}
    \pgfmathsetmacro{\wa}{28.45*\tk*fth(\xq)}
    \pgfmathsetmacro{\wb}{28.45*\tk*(1-fth(\xq))}
    \draw[samp,line width=\wa pt] ({\L*\xl+\ins},{\yb+\aq}) -- ({\L*\xr-\ins},{\yb+\aq});
    \draw[samp,line width=\wb pt] ({\L*\xl+\ins},{\yb+\bq}) -- ({\L*\xr-\ins},{\yb+\bq});
    \draw[samp,densely dotted,line width=0.4pt] ({\L*\xq},\yb) -- ({\L*\xq},{\yb+\bq});
    \fill[samp] ({\L*\xq},\yb) circle (1.3pt);
  }
  \node[anchor=east] at (-0.35,{\yb+0.5*\H}) {$S_{\k}\nu^a$};
}

\pgfmathsetmacro{\yt}{-3*(\H+\G)}
\pgfmathsetmacro{\xQ}{0.875+0.25*0.125}
\draw[hlc,line width=0.8pt,rounded corners=2pt] ({\L*7/8-0.05},{\yt-0.09}) rectangle ({\L+0.05},{\yt+\H+0.07});
\node[hlc,anchor=south,inner sep=1.5pt] at ({\L*15/16},{\yt+\H+0.07}) {$Q$};
\node[samp,anchor=north,inner sep=2pt] at ({\L*\xQ},{\yt-0.09}) {$x_Q$};

\pgfmathsetmacro{\Xz}{\L+1.1}
\def\Wz{3.2}
\def\sz{1.35}
\pgfmathsetmacro{\Hz}{\sz*\H}
\pgfmathsetmacro{\yz}{\yt+0.45}
\pgfmathsetmacro{\Xh}{\Xz+\Wz+0.6}
\pgfmathsetmacro{\Aq}{fa(\xQ)}
\pgfmathsetmacro{\Bq}{fb(\xQ)}
\pgfmathsetmacro{\Tq}{fth(\xQ)}
\pgfmathsetmacro{\lvA}{\yz+\sz*\Aq}
\pgfmathsetmacro{\lvB}{\yz+\sz*\Bq}
\pgfmathsetmacro{\lvZ}{\yz+\sz*(\Tq*\Aq+(1-\Tq)*\Bq)}

\draw[hlc,line width=0.8pt,rounded corners=2pt] ({\Xz-0.1},{\yz-0.12}) rectangle ({\Xz+\Wz+0.1},{\yz+\Hz+0.1});
\draw[hlc,line width=0.4pt] ({\L+0.05},{\yt+\H+0.07}) -- ({\Xz-0.1},{\yz+\Hz+0.1});
\draw[hlc,line width=0.4pt] ({\L+0.05},{\yt-0.09}) -- ({\Xz-0.1},{\yz-0.12});

\draw[basec] (\Xz,\yz) -- ({\Xz+\Wz},\yz);
\draw[gridc] (\Xz,\yz) -- (\Xz,{\yz+\Hz});
\draw[gridc] ({\Xz+\Wz},\yz) -- ({\Xz+\Wz},{\yz+\Hz});

\draw[samp!60,densely dotted] (\Xz,\lvA) -- (\Xh,\lvA);
\draw[samp!60,densely dotted] (\Xz,\lvB) -- (\Xh,\lvB);
\draw[dashed,line width=0.4pt] (\Xz,\lvZ) -- (\Xh,\lvZ);

\pgfmathsetmacro{\bet}{0.28}
\pgfmathsetmacro{\per}{(\Wz-2*\bet)/6}
\def\tw{0.03}
\draw[line width=0.7pt,rounded corners=0.6pt]
  (\Xz,\lvZ) -- ({\Xz+\bet},\lvZ) -- ({\Xz+\bet+2*\tw},\lvA)
  \foreach \i in {0,...,4} {
    -- ({\Xz+\bet+\i*\per+\Tq*\per-\tw},\lvA)
    -- ({\Xz+\bet+\i*\per+\Tq*\per+\tw},\lvB)
    -- ({\Xz+\bet+(\i+1)*\per-\tw},\lvB)
    -- ({\Xz+\bet+(\i+1)*\per+\tw},\lvA)
  }
  -- ({\Xz+\bet+5*\per+\Tq*\per-\tw},\lvA)
  -- ({\Xz+\bet+5*\per+\Tq*\per+\tw},\lvB)
  -- ({\Xz+\Wz-\bet-2*\tw},\lvB)
  -- ({\Xz+\Wz-\bet},\lvZ) -- ({\Xz+\Wz},\lvZ);
\node[anchor=north west,font=\footnotesize,inner sep=2pt] at (\Xz,{\yz+\Hz+0.08}) {$u_Q=z_Q+\mathscr B\pphi^Q$};
\draw[decorate,decoration={brace,mirror,amplitude=3pt},line width=0.4pt]
  ({\Xz+\bet},{\yz-0.2}) -- ({\Xz+\Wz-\bet},{\yz-0.2})
  node[midway,below=3pt,font=\footnotesize] {$\operatorname{supp}\pphi^Q\Subset Q$};

\def\bl{1.3}
\draw[basec,->] (\Xh,\yz) -- (\Xh,{\yz+\Hz+0.15});
\fill[samp] (\Xh,{\lvA-0.05}) rectangle ({\Xh+\bl*\Tq},{\lvA+0.05});
\fill[samp] (\Xh,{\lvB-0.05}) rectangle ({\Xh+\bl*(1-\Tq)},{\lvB+0.05});
\fill (\Xh,\lvZ) -- ++(-0.13,0.065) -- ++(0,-0.13) -- cycle;
\node[right,inner sep=2pt] at (\Xh,\lvZ) {$z_Q$};
\node[above,inner sep=2pt] at ({\Xh+0.45},{\yz+\Hz+0.15}) {$\tau_Q=\nu^a_{x_Q}$};

\begin{scope}[shift={({\Xz},1.35)},font=\footnotesize]
  \draw[samp,line width=2.4pt] (0,0) -- (0.5,0);
  \node[anchor=west] at (0.65,0) {atoms of $\tau_Q$, thickness $\propto$ weight};
  \fill[bandc] (0,-0.62) rectangle (0.5,-0.38);
  \node[anchor=west] at (0.65,-0.5) {support of $E_k\nu^a$ on $Q$};
  \fill[samp] (0.25,-1.0) circle (1.3pt);
  \node[anchor=west] at (0.65,-1.0) {sampled $x_Q\in Q\cap J$};
  \draw[fieldthin] (0,-1.5) -- (0.5,-1.5);
  \node[anchor=west] at (0.65,-1.5) {the field $\nu^a_x$};
  \node[anchor=west,inner sep=0pt] at (0,-2.55)
    {\textit{sampling:}\ \ $\sum_{Q}\int_Q\|\nu^a_x-\nu^a_{x_Q}\|_{\mathrm K}\le 2\delta_k\to0$};
  \node[anchor=west,inner sep=0pt] at (0,-3.2)
    {\textit{generation:}\ \ $\|\boldsymbol\varepsilon_{u_Q}-\tau_Q\|_{\mathrm K}<\varepsilon$ on each $Q$};
\end{scope}
\end{tikzpicture}}
\caption{Schematic of the proof of
Proposition~\ref{inhomogenization ac part} ($n=1$; the vertical axis is the
fibre $\V$). Top: the field $x\mapsto\nu^a_x$, here two atoms whose positions
and weights vary with $x$ (thickness proportional to weight). Rows $k=1,2,3$:
on each $Q\in\mathscr{D}_k$ a point $x_Q\in Q\cap J$ is sampled and the
homogeneous Young measure $\tau_Q=\nu^a_{x_Q}$ is frozen on $Q$
($S_k\nu^a$). The shaded bands are the support of the martingale $E_k\nu^a$
on $Q$, which smears the atoms over the oscillation of the field on $Q$; they
shrink as $k$ grows, and
$\sum_Q\int_Q\|\nu^a_x-\nu^a_{x_Q}\|_{\mathrm{K}}\dif x\le2\delta_k\to0$.
Right: on each $Q$, $\tau_Q$ is generated by $u_Q=z_Q+\mathscr{B}\pphi^Q$,
$\pphi^Q\in\CC^\infty_c(Q)$, whose distribution of values on $Q$ is
$\varepsilon$-close to $\tau_Q$. The concentration part
$\lambda^a\nu^\infty_x$ is not drawn.}
\label{fig:inhom-ac}
\end{figure}
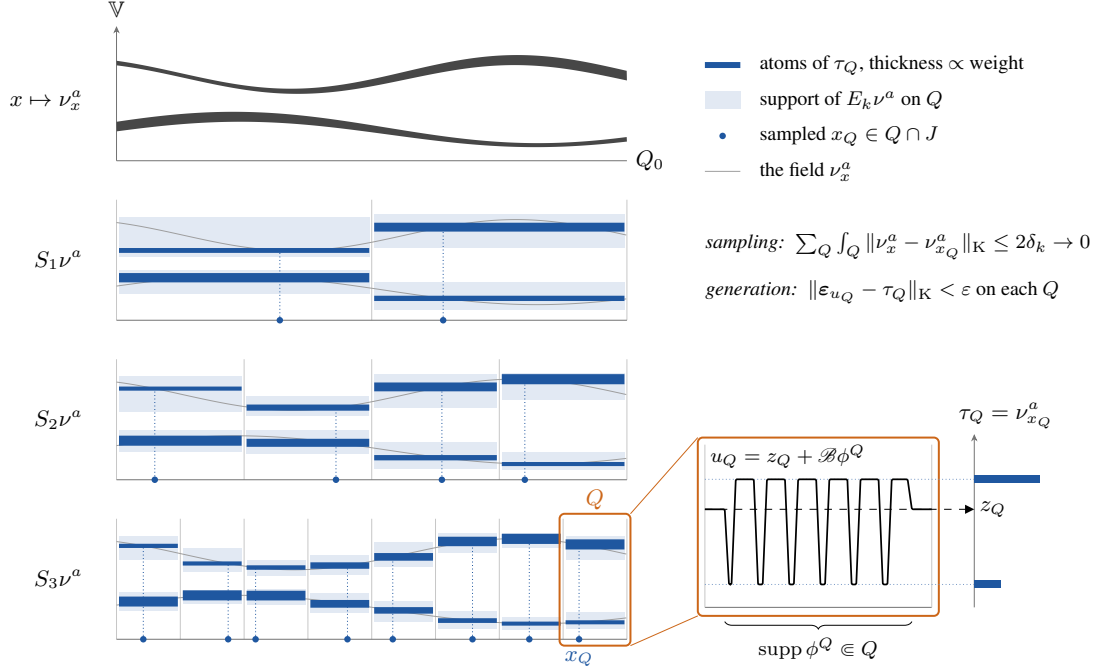
\begin{proof}[Proof of Proposition~\ref{inhomogenization ac part}]
We now turn to the proof of the absolutely continuous part of the inhomogenization, sketched in Figure~\ref{fig:inhom-ac}. Recall that $\Phi$ is admissible if $\|T\Phi\|_{\text{LIP}}\leq 1$, in which case $\text{lip}(\Phi)\leq L$ by \eqref{lip constant L}; we set $L_1:=L(1+c^{-1})$. Throughout, $\Phi,\Phi'$ denote admissible integrands, and for $g\in\LL^1(\Omega)$ and a cube $Q$ we write $(g)_Q:=\dashint_Q g\dif x$. We use the following notation:
\begin{itemize}
    \item $\nu^a:=(\nu_x,\lambda^a\Leb\Lcorner\Omega,\nu_x^\infty)$ is the absolutely continuous part of the Young measure at issue, and $\nu^a_{x}:=(\nu_{x},\lambda^a(x),\nu_{x}^\infty)$ is the homogeneous Young measure at fixed $x$;
    \item $f_\Phi(x):=\langle\Phi,\nu^a_x\rangle=\langle \Phi, \nu_x \rangle +\lambda^a(x) \langle \Phi^\infty, \nu_x^\infty \rangle$ is the action of $\nu^a_x$ on $\Phi$;
    \item $v^a(x):= \overline{\nu}_x + \lambda^a(x) \overline{\nu}_x^\infty\equiv \overline{\nu}^a(x)$ is the barycenter of $\nu^a_x$;
    \item $M^a(x):=\langle 1+G^\infty,\nu_x\rangle + \lambda^a(x)$; this is a function in $\LL^1(\Omega)$ with $M^a\geq 1$, which bounds $f_\Phi$ uniformly in $\Phi$ (see below).
\end{itemize}
For admissible $\Phi,\Phi'$ and a.e.\ $x,y\in\Omega$ the following hold.
\begin{itemize}
    \item Since $\Phi=(1+G^\infty)\,T\Phi(\iota_G)$, $\|T\Phi\|_\infty\leq 1$ and $\nu_x^\infty$ is carried by $\{G^\infty=1\}$, we have $|\Phi|\leq 1+G^\infty$, hence $|f_\Phi|\leq M^a$. In the same way,
    \begin{equation}\label{ac continuity in Phi}
    \begin{split}
    |f_\Phi(x)-f_{\Phi'}(x)|&\leq \|T\Phi-T\Phi'\|_\infty\big(\langle 1+G^\infty,\nu_x\rangle+\lambda^a(x)\langle 1,\nu_x^\infty\rangle\big)\\
    &=\|T\Phi-T\Phi'\|_\infty\,M^a(x).
    \end{split}
    \end{equation}
    Consequently, the supremum defining $\|\cdot\|_{\text{K}}$ may be taken over a countable, $\|T\cdot\|_\infty$-dense family of admissible integrands; in particular, all the functions of $x$ (and $y$) defined below through $\|\cdot\|_{\text{K}}$ are measurable.
    \item $|f_\Phi(x)-f_\Phi(y)|\leq \|\nu_x^a-\nu_y^a\|_\text{K}$, by definition of the Kantorovich norm.
    \item For every unit vector $e\in\V$ the integrand $(1+c^{-1})^{-1}\ell_e$, where $\ell_e(z):=z\cdot e$, is admissible, and $f_{\ell_e}=e\cdot v^a$. Choosing $e$ parallel to $v^a(x)-v^a(y)$, respectively to $v^a(x)$, we obtain
    $$
    |v^a(x)-v^a(y)|\leq (1+c^{-1})\|\nu_x^a-\nu_y^a\|_\text{K}\qquad\text{and}\qquad |v^a|\leq (1+c^{-1})M^a.
    $$
    \item $|\Phi(z)|\leq |\Phi(0)|+\text{lip}(\Phi)|z|\leq 1+L|z|$ for all $z \in \V$.
\end{itemize}
The core of the proof is the construction itself, so we present each choice (of sets, cubes, sample points and maps) and each step in a separate paragraph, in their logical order.

\textbf{Cubes.} Let $\mathscr{D}^{\Omega}_k:=\{Q \in \mathscr{D}_k: Q \subset\Omega\}$ and $B_k := \Omega \setminus \bigcup \mathscr{D}^{\Omega}_k$, and let $r_k:=\operatorname{diam}(Q)$ for $Q\in\mathscr{D}_k$. Since $\Omega$ is open, $B_k \searrow \emptyset$, and therefore $\int_{B_k} M^a \dif x\to 0$ as $k\to\infty$, because $M^a\in \LL^1(\Omega)$.

\textbf{Jensen set.} Since $\HH$ is separable, there is a set $J \subset \Omega$ of full measure such that for every $x\in J$ we have $M^a(x)<\infty$ and the Jensen inequality for $\nu^a_x$, with barycenter $v^a(x)$, holds simultaneously for all integrands in our class.

\textbf{Martingale convergence.} For $Q\in\mathscr{D}^{\Omega}_k$, let $(\nu^a)_Q$ be the conditional expectation of $\nu^a$ on $Q$, defined as in the singular part with $\lambda^a\Leb$ in place of $\lambda$, i.e.\ the homogeneous Young measure
$$
(\nu^a)_Q:=\big(\nu_Q,\,(\lambda^a)_Q,\,\nu_Q^{a,\infty}\big),
$$
where $\nu_Q:=\dashint_Q\nu_x\dif x$ and $\nu_Q^{a,\infty}:=\big(\int_Q\lambda^a\dif x\big)^{-1}\int_Q\lambda^a(x)\,\nu_x^\infty\dif x$ (any probability measure will do for $\nu_Q^{a,\infty}$ if $\int_Q\lambda^a\dif x=0$). By construction,
$$
\langle \Phi,(\nu^a)_Q \rangle= \langle \Phi, \nu_Q\rangle + (\lambda^a)_Q \langle \Phi^\infty, \nu_{Q}^{a,\infty} \rangle= (f_\Phi)_Q.
$$
We claim that
\begin{equation}\label{martingale convergence of delta k}
\delta_k :=\sum_{Q \in \mathscr{D}^{\Omega}_k} \int_Q \|\nu_x^a-(\nu^a)_Q\|_{\text{K}}\dif x \xrightarrow[k \to \infty]{} 0.
\end{equation}
This is a uniform-in-$\Phi$ version of the scalar statement
\begin{equation}\label{ac scalar martingale convergence}
\omega_k(g):=\sum_{Q\in\mathscr{D}^{\Omega}_k}\int_Q|g-(g)_Q|\dif x\xrightarrow[k\to\infty]{}0\qquad\text{for every } g\in\LL^1(\Omega),
\end{equation}
which is Doob's martingale convergence theorem in $\LL^1$ for the dyadic filtration, applied to $g\boldsymbol{1}_\Omega$ (on each $Q\in\mathscr{D}^{\Omega}_k$ its conditional expectation equals $(g)_Q$). Alternatively, \eqref{ac scalar martingale convergence} is immediate for $g\in\CC_c(\Omega)$ by uniform continuity, and it extends to $\LL^1(\Omega)$ by density, since $\omega_k(g)\leq\omega_k(\tilde g)+2\|g-\tilde g\|_{\LL^1(\Omega)}$.

\emph{Proof of \eqref{martingale convergence of delta k}.} By Ascoli--Arzel\`a, the set $\{T\Phi: \|T\Phi\|_{\text{LIP}}\leq 1\}$ is compact in $\CC(\overline{\BG})$. Hence, given $\tau>0$, there are finitely many admissible $\Phi_1,\dots, \Phi_N$ such that every admissible $\Phi$ satisfies $\|T\Phi-T\Phi_i\|_\infty\leq\tau$ for some $i$. For such an $i$, \eqref{ac continuity in Phi} and its average over $Q$ give, for $x\in Q\in\mathscr{D}^{\Omega}_k$,
\begin{align*}
|f_{\Phi}(x)-(f_\Phi)_Q| &\leq |f_\Phi(x)-f_{\Phi_i}(x)|+ |f_{\Phi_i}(x)-(f_{\Phi_i})_Q|+|(f_{\Phi_i})_Q-(f_\Phi)_Q|\\
&\leq \sum_{j=1}^N |f_{\Phi_j}(x)-(f_{\Phi_j})_Q| + \tau\big(M^a(x)+(M^a)_Q\big).
\end{align*}
The right-hand side does not depend on $\Phi$, so taking the supremum over admissible $\Phi$,
$$
\|\nu_x^a-(\nu^a)_Q\|_{\text{K}} \leq \sum_{j=1}^N |f_{\Phi_j}(x)-(f_{\Phi_j})_Q| + \tau\big(M^a(x)+(M^a)_Q\big).
$$
Integrating over $Q$ and summing over $Q \in \mathscr{D}^{\Omega}_k$ gives
$$
\delta_k \leq \sum_{j=1}^N \omega_k(f_{\Phi_j})+ 2 \tau \|M^a\|_{\LL^1(\Omega)}.
$$
Each $f_{\Phi_j}$ belongs to $\LL^1(\Omega)$, as $|f_{\Phi_j}|\leq M^a$, so $\omega_k(f_{\Phi_j})\to 0$ by \eqref{ac scalar martingale convergence}. Hence $\limsup_k \delta_k \leq 2 \tau \|M^a\|_{\LL^1(\Omega)}$, and \eqref{martingale convergence of delta k} follows since $\tau>0$ is arbitrary.

\textbf{Choice of the sample points.} Fix $k$ and $Q\in\mathscr{D}^{\Omega}_k$, and consider the measurable function
$$
h_Q(y) := \int_Q \|\nu_x^a-\nu^a_y\|_{\text{K}} \dif x, \qquad y\in Q.
$$
By the triangle inequality through $(\nu^a)_Q$, we have $h_Q(y)\leq \int_Q \|\nu_x^a-(\nu^a)_Q\|_{\text{K}} \dif x+\Leb(Q)\,\|(\nu^a)_Q-\nu^a_y\|_{\text{K}}$, and averaging over $y\in Q$,
$$
\dashint_Q h_Q(y) \dif y \leq 2\int_Q \|\nu_x^a-(\nu^a)_Q\|_{\text{K}} \dif x.
$$
The set $\{y\in Q: h_Q(y)\leq \dashint_Q h_Q\}$ has positive measure and $J$ has full measure, so we can pick a point $x_Q\in J$ in this set, that is,
$$
\int_Q \|\nu_x^a-\nu^a_{x_Q}\|_{\text{K}} \dif x \leq 2\int_Q \|\nu_x^a-(\nu^a)_Q\|_{\text{K}} \dif x.
$$
Summing over $Q\in\mathscr{D}^{\Omega}_k$,
\begin{equation}\label{ac sampling estimate}
\sum_{Q \in \mathscr{D}^{\Omega}_k } \int_Q \|\nu_x^a-\nu^a_{x_Q}\|_{\text{K}} \dif x \leq 2\delta_k.
\end{equation}

\textbf{Construction.} Fix $\eps\in(0,1]$ and $k$, to be chosen at the end, and let $Q\in\mathscr{D}^{\Omega}_k$. Since $x_Q \in J$, the homogeneous Young measure $\nu^a_{x_Q}$ satisfies the Jensen inequality, with barycenter $v^a(x_Q)$. By Proposition~\ref{homogeneous case characterization} on the homogeneous case (realizing the elementary Young measure on $Q$ rather than on $\XX$), we can select $\pphi^Q \in \CC_c^\infty(Q,\V)$ with $\|\pphi^Q\|_{\WW^{l-1,1}(Q,\V)}< \eps \Leb(Q)/\Leb(\Omega)$ such that, setting $u_Q:=v^a(x_Q)+\mathscr{B}\pphi^Q$,
\begin{equation}\label{ac homogeneous generation}
\Big|f_\Phi(x_Q)-\dashint_Q \Phi(u_Q)\dif y\Big|<\eps\qquad\text{for every admissible } \Phi.
\end{equation}
Set $\pphi:=\sum_{Q\in\mathscr{D}^{\Omega}_k}\pphi^Q$, so that $\|\pphi\|_{\WW^{l-1,1}(\Omega,\V)}<\eps$, $\mathscr{B}\pphi=\mathscr{B}\pphi^Q$ on each $Q\in\mathscr{D}^{\Omega}_k$, and $\mathscr{B}\pphi=0$ on $B_k$. Define:
\begin{itemize}
    \item $S_k \nu^a$, the Young measure equal to $\nu^a_{x_Q}$ on each $Q\in \mathscr{D}^{\Omega}_k$ and to $\boldsymbol{\eps}_{v^a}$ on $B_k$;
    \item $S_k v^a$, the function equal to $v^a(x_Q)$ on each $Q\in \mathscr{D}^{\Omega}_k$ and to $v^a$ on $B_k$;
    \item $\psi:= \phi_t \ast (v^a \Leb \Lcorner \Omega) + \mathscr{B}\pphi$, with $t>0$ to be chosen at the end.
\end{itemize}
Note that $S_k v^a+\mathscr{B}\pphi=u_Q$ on each $Q\in\mathscr{D}^{\Omega}_k$. By the triangle inequality,
\begin{align*}
d(\nu^a,\boldsymbol{\eps}_\psi) \leq{} \underbrace{d(\nu^a, S_k\nu^a)}_{\text{sampling}}+ \underbrace{d(S_k\nu^a, \boldsymbol{\eps}_{S_k v^a + \mathscr{B}\pphi})}_{\text{homogeneous generation}}+ \underbrace{d( \boldsymbol{\eps}_{S_k v^a + \mathscr{B}\pphi}, \boldsymbol{\eps}_{ v^a + \mathscr{B}\pphi})}_{\text{barycenter sampling}}+ \underbrace{d(\boldsymbol{\eps}_{ v^a + \mathscr{B}\pphi},\boldsymbol{\eps}_\psi )}_{\text{barycenter mollification}}.
\end{align*}
As in the proof of the singular part, each term is the error introduced by one step of the construction, and its estimate reflects that step; we therefore treat the four terms in separate paragraphs, named after the corresponding steps. In each case we bound $|\ddual{\eta \otimes \Phi}{\cdot}|$ uniformly in the admissible pair $(\eta,\Phi)$; of $\eta$ we only use $\|\eta\|_\infty\leq 1$ and $\text{lip}(\eta)\leq 1$.

\textbf{Sampling.} On the cubes we use \eqref{ac sampling estimate}; on $B_k$ we use $|f_\Phi|\leq M^a$, $|\Phi(v^a)|\leq 1+L|v^a|$, $|v^a|\leq (1+c^{-1})M^a$ and $M^a\geq 1$. Thus
\begin{align*}
\big|\ddual{\eta \otimes \Phi}{\nu^a-S_k\nu^a}\big|
&=\Big|\sum_{Q\in \mathscr{D}^{\Omega}_k} \int_Q \eta\,\big(f_\Phi-f_\Phi(x_Q)\big) \dif x + \int_{B_k}\eta\,\big(f_{\Phi}-\Phi(v^a)\big) \dif x\Big| \\
&\leq \sum_{Q\in \mathscr{D}^{\Omega}_k} \int_Q \|\nu_x^a-\nu^a_{x_Q}\|_{\text{K}}\dif x+ \int_{B_k} \big(M^a+1+L|v^a|\big)\dif x \\
&\leq 2\delta_k +(2+L_1)\int_{B_k}M^a \dif x.
\end{align*}

\textbf{Homogeneous generation.} On $B_k$ both Young measures are equal to $\boldsymbol{\eps}_{v^a}$, so the corresponding integrals cancel. On the cubes we argue as for the homogeneous generation in the singular part: with $\eta_Q:=\dashint_Q\eta\dif x$,
\begin{align*}
\ddual{\eta \otimes \Phi}{S_k\nu^a-\boldsymbol{\eps}_{S_kv^a + \mathscr{B}\pphi}}
&=\sum_{Q \in \mathscr{D}^{\Omega}_k} \Big(\int_Q \eta\dif x\Big)\Big(f_\Phi(x_Q) - \dashint_Q \Phi(u_Q)\dif y\Big)\\
&\quad+\sum_{Q \in \mathscr{D}^{\Omega}_k}\int_Q (\eta_Q -\eta)\,\Phi(u_Q)\dif x=:I_1+I_2.
\end{align*}
By \eqref{ac homogeneous generation},
$$
|I_1|\leq \eps \sum_{Q \in \mathscr{D}^{\Omega}_k} \int_Q|\eta|\dif x \leq \eps \Leb(\Omega).
$$
For $I_2$ we use $|\eta-\eta_Q|\leq \text{lip}(\eta)\operatorname{diam}(Q)\leq r_k$ on $Q$, then $|\Phi|\leq 1+G^\infty$, and finally \eqref{ac homogeneous generation} for the admissible integrand $1+G^\infty$, for which $f_{1+G^\infty}=M^a$:
$$
|I_2|\leq r_k\sum_{Q \in \mathscr{D}^{\Omega}_k}\int_Q \big(1+G^\infty(u_Q)\big)\dif x \leq r_k\sum_{Q \in \mathscr{D}^{\Omega}_k}\Leb(Q)\big(M^a(x_Q)+\eps\big).
$$
Since $M^a=f_{1+G^\infty}$, we also have $M^a(x_Q)\leq M^a(x)+\|\nu_x^a-\nu^a_{x_Q}\|_{\text{K}}$ for a.e.\ $x\in Q$; integrating over $Q$, summing over $Q\in\mathscr{D}^{\Omega}_k$ and using \eqref{ac sampling estimate} gives $\sum_{Q\in\mathscr{D}^{\Omega}_k}\Leb(Q)\,M^a(x_Q)\leq \|M^a\|_{\LL^1(\Omega)}+2\delta_k$, whence
$$
|I_2|\leq r_k\big(\|M^a\|_{\LL^1(\Omega)}+2\delta_k+\eps\Leb(\Omega)\big).
$$

\textbf{Barycenter sampling.} On $B_k$ there is nothing to prove. On the cubes, using $\text{lip}(\Phi)\leq L$, the bound $|v^a(x)-v^a(y)|\leq (1+c^{-1})\|\nu_x^a-\nu_y^a\|_{\text{K}}$ and \eqref{ac sampling estimate},
\begin{align*}
\big|\ddual{\eta\otimes\Phi}{\boldsymbol{\eps}_{S_k v^a + \mathscr{B}\pphi}-\boldsymbol{\eps}_{ v^a + \mathscr{B}\pphi}}\big|
&\leq L\sum_{Q \in \mathscr{D}^{\Omega}_k} \int_Q |v^a(x_Q)-v^a(x)|\dif x \\
&\leq L_1\sum_{Q \in \mathscr{D}^{\Omega}_k} \int_Q \|\nu_x^a-\nu^a_{x_Q}\|_{\text{K}} \dif x \leq 2L_1\delta_k.
\end{align*}

\textbf{Barycenter mollification.} Again by $\text{lip}(\Phi)\leq L$,
$$
\big|\ddual{\eta\otimes\Phi}{\boldsymbol{\eps}_{ v^a + \mathscr{B}\pphi}-\boldsymbol{\eps}_\psi}\big| \leq L\|\phi_t \ast (v^a \Leb \Lcorner \Omega) -v^a\|_{\LL^1(\Omega)} \xrightarrow[t\to 0]{} 0,
$$
since $v^a\in\LL^1(\Omega)$, as $|v^a|\leq (1+c^{-1})M^a$.

Finally, collecting the four estimates,
\begin{align*}
d(\nu^a,\boldsymbol{\eps}_\psi)\leq{}& (2+2L_1)\delta_k+(2+L_1)\int_{B_k}M^a\dif x+\eps\Leb(\Omega)+r_k\big(\|M^a\|_{\LL^1(\Omega)}+2\delta_k+\eps\Leb(\Omega)\big)\\
&+L\|\phi_t \ast (v^a \Leb \Lcorner \Omega) -v^a\|_{\LL^1(\Omega)}.
\end{align*}
We choose the parameters in the following order. First, by \eqref{martingale convergence of delta k}, $B_k\searrow\emptyset$ and $r_k\to 0$, we take $k$ so large that
$$
\delta_k\leq \eps,\qquad \int_{B_k}M^a\dif x\leq \eps,\qquad r_k\big(\|M^a\|_{\LL^1(\Omega)}+2+\Leb(\Omega)\big)\leq \eps;
$$
then we select the points $x_Q$ and the maps $\pphi^Q$, $Q\in\mathscr{D}^{\Omega}_k$, as above; finally, we take $t>0$ so small that the last term is at most $\eps$. With these choices,
$$
d(\nu^a,\boldsymbol{\eps}_\psi)\leq \big(6+3L_1+\Leb(\Omega)\big)\eps \qquad\text{and}\qquad \|\pphi\|_{\WW^{l-1,1}(\Omega,\V)}<\eps,
$$
and since $\eps\in(0,1]$ is arbitrary, the proof is complete.
\end{proof}
\begin{remark}
The averaged measures $(\nu^a)_Q$ enter only as intermediate objects
in the Kantorovich estimate. The homogeneous measures generated on each
cube are the point values $\nu_{x_Q}$, translated to have barycentre
zero; they satisfy the Jensen inequality because $x_Q$ is a Jensen
point and $\mathscr{A}$-quasiconvexity is invariant under translations.
This is essential, since averages of measures satisfying the Jensen
inequality at different barycentres need not satisfy it.
\end{remark}

\section{Characterization of concentration part of \texorpdfstring{$\mathscr{A}$}{A}-free Young measure}
In this section we want to prove a characterization theorem for the concentration part of $\mathscr{A}$-free Young measures in the new geometric representation of the compactification. The result will be the same of \cite{kristensen raita}, only now with the concentration angle Young measures being defined on the level set $\mathbb{S}_G$, the boundary of the set $\B=\{G^\infty < 1\}$ which is not necessarily convex. \\
Define the \emph{convex deficiency integrand} $\text{CD}_f: \V \to \R$ of a Lipschitz integrand $f:\V \to \R$ as
$$
\text{CD}_f(z) \equiv \sup \left\{\frac{f(w+tz)-f(w)}{t}: w \in \V, t >0 \right\}.
$$
Note that $\text{CD}_f$ is the smallest positively 1-homogeneous integrand such that
$$
f(w+z) \leq f(w) + \text{CD}_f(z)
$$
for all $z,w \in \V$.
\begin{lemma}\label{lemma: CD_f as sup of differentiability point}
Let $f: \V \to \R$ a Lipschitz integrand. Another formula for $\emph{CD}_f(z)$, $z \in \V$ is
$$
\emph{CD}_f(z)= \sup \left\{f'(w) \cdot z: f \emph{ is differentiable at }w \in \V \right\}
$$
\end{lemma}
\begin{proof}
It is a consequence of Rademacher's theorem, since $f$ is Lipschitz.   See \cite[Lemma ~6.1]{kristensen raita} for a proof.
\end{proof}
\begin{proposition}
Let $\mathscr{A}$ be a differential operator with constant rank and spanning cone condition. \\
Let $\lambda \in \mathscr{M}^+(\Omegacl)$ with $\lambda(\partial \Omega)=0$ and let $(\mu_x)_{x \in \Omega}$ a family of probability measures $\mu_x \in \mathscr{M}^+_1(\mathbb{S}_G)$ which is $\lambda$ measurable. Then $(\lambda,\mu_x)$ is the concentration part of an $\mathscr{A}$-free Young measure if and only if 
\begin{enumerate}[label=(\roman*)]
    \item $\mathscr{A}(\overline{\mu}_{(\cdot)} \lambda) \in \mathscr{A}(\LL^1(\Omega,\V))$, and \\
    \item  for all $\mathscr{A}$-quasiconvex functions $f:\V \to \R$ with linear growth for which one can find $r=r(f)$ so that $f(z)=f^\infty(z)$ for $|z|>r$ it holds
    $$
    \emph{CD}_f(\overline{\mu}_x) \leq \int_{\mathbb{S}_G} f^\infty \dif \mu_x, \quad \text{for } \lambda^a\Leb \text{ almost all } x \in \Omega. \label{item:CD_f jensen}
    $$
\end{enumerate}
\end{proposition}
\begin{proof}
We begin with proving that $\ref{item:CD_f jensen}$ holds if $(\lambda,\mu_x)$ is a concentration Young measure, namely we are given the Young measure $\nu=(\nu_x, \lambda, \nu_x^\infty)$ and $\mu_x=\nu_x^\infty$ for $\lambda$- almost all $x \in \Omegacl$. In fact,  $\ref{item:CD_f jensen}$ holds for all $\mathscr{A}$-quasiconvex functions of linear growth. To prove this, using Proposition \ref{sufficiency proposition} one has that for some negligible set $N^a \subset \Omega$ the Jensen inequality 
$$
F(\overline{\nu}_x + \lambda^a(x) \overline{\nu}^\infty_x) \leq \int_\V F \dif \nu_x + \lambda^a(x) \int_{\mathbb{S}_G} F^\infty \dif \nu_x^\infty, \quad \text{for } x \in \Omega \setminus N^a,
$$
holds for all quasiconvex integrands $F$ of linear growth. Suppose now we are at a point $x \in \Omega $ of \emph{diffuse concentration} where $t\equiv \lambda^a(x) >0$. Consider for $\eps >0$ and $\xi \in \V$ the integrand
$$
f_\eps(z)= \frac{F(\xi +\eps z)-F(\xi)}{\eps},
$$
which is itself quasiconvex with same recession integrand $f_\eps^\infty= F^\infty$. As $F$ is quasiconvex it is locally Lipschitz and differentiable $\Leb$ a.e.; if $\xi $ is one of such points of differentiability, then
$$
\lim_{\eps \to 0} f_\eps(z)= F'(\xi) \cdot z \quad \text{locally uniformly in } z.
$$
We can use $f_\eps$ as test function for the Jensen inequality above, and pass to the limit as $\eps \to 0$ to obtain 
\begin{align*}
F'(\xi) \cdot (\overline{\nu}_x + t \overline{\nu}^\infty_x) &\leq  \int_\V F'(\xi) \cdot z \dif \nu_x + t \int_{\mathbb{S}_G} F^\infty  \dif \nu_x^\infty\\
&= F'(\xi) \cdot \overline{\nu}_x +t \int_{\mathbb{S}_G} F^\infty  \dif \nu_x^\infty,
\end{align*}
and since $t>0$ it reads $F'(\xi) \cdot \overline{\nu}_x^\infty \leq \int_{\mathbb{S}_G} F^\infty \dif \nu_x^\infty$. Using now Lemma \ref{lemma: CD_f as sup of differentiability point}, the supremum over $\xi$ on the left-hand side yields the Jensen inequality
$$
\text{CD}_F(\overline{\nu}_x^\infty) \leq \int_{\mathbb{S}_G} F^\infty \dif \nu_x^\infty \quad \text{for }x \in \Omega \setminus N^a \ \text{with } \lambda^a(x)>0.
$$
The inequality thus holds for $\lambda^a \Leb$ almost all $x$, since we care about only of points of positive density. \\
To prove $(i)$, one notices that the barycentre $v= \overline{\nu}_{(\cdot)}\Leb + \overline{\nu}^\infty_{(\cdot)} \lambda$
is $\mathscr{A}$-free, i.e. $\mathscr{A}v=0$, so that
$$
\mathscr{A}(\overline{\nu}^\infty_{(\cdot)} \lambda)=\mathscr{A}(-\overline{\nu}_{(\cdot)} \Leb)
$$
with $-\overline{\nu} \in \LL^1(\Omega,\V)$. \\
To prove sufficiency, by $(i)$ there exists $v \in \LL^1$ such that the Young measure $\mu\equiv(\delta_{v(x)}, \lambda, \mu_x)$ has an $\mathscr{A}$-free barycenter, i.e. $\overline{\mu}= v \Leb \Lcorner \Omega + \overline{\mu}_{(\cdot)} \lambda = u \in \mathscr{M}(\Omega,\V)$ and $\mathscr{A}u=0$. Let $f$ be a special quasiconvex function of linear growth, then by $(ii)$ we have for $\lambda^a(x)>0$
\begin{align*}
f(v(x)+ \lambda^a(x)\overline{\mu}_x) &\leq f(v(x)) + \lambda^a(x)\text{CD}_f(\overline{\mu}_x) \\
&\leq f(v(x)) + \lambda^a(x) \int_{\mathbb{S}_G} f^\infty \dif \mu_x,
\end{align*}
and by the results of Section \ref{inhomogenization section} we have that $\mu$ is an $\mathscr{A}$-free Young measure.
\end{proof}
\section{Convergence of energy principles}\label{convergence of energy section}
Weak$^*$ convergence of a sequence of measures is compatible with both
oscillation and concentration, and the generated Young measure records how
much of each survives in the limit. In this section we ask how much of the
two is ruled out when, in addition, the energies converge. Precisely, we
consider in the gradient case
$$
\nabla u_j \Leb =v_j \Leb \wstar v= Du\in \mathscr{M}(\Omega, \M)
$$
or, in general,
$$
\begin{cases}
   v_j \wstar v & \text{in }\mathscr{M}(\Omega,\V),\\   
   \mathscr{A}v_j \to \mathscr{A}v & \text{in }\WW^{-k,p}_{\mathrm{loc}}(\Omega,\mathbb W), \ 1<p<\frac{n}{n-1},
\end{cases}
$$
and we assume the convergence of energies, with $G$ an
$\mathscr{A}$-quasiconvex integrand of linear growth with regular recession
function, as in \textup{(A3)},
$$
\int_\Omega G(v_j) \to \int_\Omega G(v) := \int_\Omega G(v^a)\dif x + \int_{\Omega} G^\infty (v^s) 
$$
the right-hand side being understood in the sense of Goffman and Serrin
(Section~\ref{notation section}). Up to a subsequence, $(v_j)$ generates a
Young measure $\nu=(\nu_x,\lambda,\nu^\infty_x)$, and since
$G\in\E_G(\Omega,\V)$ is an admissible test integrand, the generation of the
Young measure also gives
$$
\int_\Omega G(v_j) \to \int_\Omega \int_{\V} G(z) \dif \nu_x \dif x+ \int_{\Omegacl} \int_{\mathbb{S}_G} G^\infty(z) \dif \nu_x^\infty \dif \lambda.
$$
The aim is to read off the structure of $\nu$ from the equality of these two
limits. In the model case $G=G^\infty=|\cdot|$ the hypothesis is the strict
convergence of Reshetnyak's continuity theorem: no concentration angles
survive, $\nu$ is elementary on the singular part, with $\lambda^s=|v^s|$
and $\nu^\infty_x=\delta_{\frac{\dif v^s}{\dif|v^s|}(x)}$, and at
$\Leb$-a.e.\ point oscillation and diffuse concentration can only occur
along the half-line through $v^a(x)$ (Example~\ref{ex:reshetnyak}). For a general energy this cannot be expected
(Example~\ref{ex:staircase}), and the point is to find out what replaces it.

To this end we argue on the integrand $G$ itself, through the
compactification of $\V$ adapted to it (Section~\ref{compactification section}),
rather than through the Euclidean sphere. In this compactification the
concentration-angle measures are probability measures on the indicatrix
$\SG=\{G^\infty=1\}$, so that $\langle G^\infty,\nu^\infty_x\rangle=1$, and
the convergence of energies turns the Jensen inequalities satisfied by $\nu$
(Proposition~\ref{sufficiency proposition} and Theorem~\ref{KK theorem}),
tested with $f=G$, into equalities. At singular points this says that the
barycentre $\overline{\nu}^\infty_x$ of a probability measure carried by
$\SG$ lies on $\SG$ itself; the structure of $\nu^\infty_x$ is then dictated
by the convex geometry of $\overline{\mathrm{co}}(\SG)$ at
$\overline{\nu}^\infty_x$: $\nu^\infty_x$ is a Dirac mass at extreme points,
it is supported on a face where $\SG$ is flat, and it is not constrained at
all at points of $\SG$ lying in the interior of $\overline{\mathrm{co}}(\SG)$
(Figure~\ref{support sets for nu_x infty}). 
The discussion in the rest of this section is heuristic. The rigorous
statement, for a general operator $\mathscr{A}$ satisfying \textup{(A1)}, is
Proposition~\ref{prop:rigidity} in Section~\ref{G strict rigidity principle},
where moreover the third case is ruled out by the structure of the wave cone
(Lemma~\ref{lem:dent}). Here, for the sake of simplicity, we restrict to the
gradient case $\V=\M$, $\mathscr{A}=\operatorname{curl}$, $v_j=Du_j$ and
$v=Du$, so that $\nu$ is a $\bv$ gradient Young measure, and we assume that
no concentration occurs at $\partial\Omega$. Nothing essential is lost in
this simplification, since the computations are the same as in the proof of
Proposition~\ref{prop:rigidity}: the convergence of energies is rewritten as
the vanishing of a sum of integrals whose integrands are nonnegative by the
Jensen inequalities satisfied by $\nu$ (Proposition~\ref{sufficiency proposition}), tested with $G$ and
$G^\infty$; hence each integrand vanishes, which gives the pointwise identities
\eqref{equality 1 G infty and nux} and \eqref{equality 2 G infty and nux}.

More specifically, by Proposition \ref{sufficiency proposition} we have the Jensen inequality relations 
\begin{equation}\label{Jensen for G, AC part, cvg of energy principle}
\int_{\M} G(z)\dif \nu_x(z)+ \lambda^a \int_{\mathbb S_G} G^\infty(z) \dif \nu_x^\infty \geq G\left( \overline{\nu}_x+ \lambda^a \overline{\nu}_x^\infty \right) \quad \text{for $\Leb$ almost all }x \in \Omega,
\end{equation}
where $\lambda= \lambda^a \Leb + \lambda^s=\lambda^a \Leb + \dfrac{\lambda^s}{|D^su|}|D^s u|+ \lambda^*$ is the Lebesgue--Radon--Nikod\'ym decomposition of $\lambda$, and
\begin{equation}\label{Jensen for G, singular part, cvg of energy principle}
\int_{\mathbb S_G } G^\infty(z) \dif \nu_x^\infty(z) \geq G^\infty(\overline{\nu}_x^\infty) \quad \text{for $\lambda^s$ almost all }x \in \Omega,  
\end{equation}
or by the first condition of (\ref{conditions on nu infty in sufficiency theorem}), in the equivalent form given at the end of the proof of Proposition~\ref{sufficiency proposition}, we can write
\begin{equation}\label{D^s u and nu_x^infty relation}
\frac{D^s u}{|D^s u|}(x)=\frac{\lambda^s}{|D^s u |}\overline{\nu}_x^\infty \quad \text{for $|D^s u|$ almost all }x \in \Omega.
\end{equation}
On the other side, the convergence of energy assumption 
$$
\int_\Omega G(\nabla u(x))\dif x + \int_{\Omega} G^\infty \left(\frac{D^s u}{|D^s u|}(x) \right)\dif |D^s u| 
=\int_\Omega \int_{\M} G(z) \dif \nu_x \dif x+ \int_{\Omegacl} \int_{\mathbb{S}_G} G^\infty(z) \dif \nu_x^\infty \dif \lambda
$$
can be rewritten, since $\lambda(\partial\Omega)=0$, as
\begin{multline*}
\int_\Omega \left( \langle G,\nu_x \rangle + \lambda^a \langle G^\infty,\nu_x^\infty \rangle -G(\nabla u(x)) \right)\dif x
+ \int_\Omega \left(  \frac{\lambda^s}{|D^s u|}\langle G^\infty,\nu_x^\infty \rangle -G^\infty\left( \frac{D^s u}{|D^s u|}(x) \right) \right) \dif |D^s u| \\
+ \int_\Omega \langle G^\infty,\nu_x^\infty\rangle\dif\lambda^* = 0 .
\end{multline*}
The three integrands are nonnegative: the first by
(\ref{Jensen for G, AC part, cvg of energy principle}), the second by
(\ref{Jensen for G, singular part, cvg of energy principle}) and
(\ref{D^s u and nu_x^infty relation}), and the third because
$\langle G^\infty,\nu_x^\infty\rangle=1$ in the $G$-normalisation. Hence
\begin{align}
\quad &  \langle G,\nu_x \rangle + \lambda^a \langle G^\infty,\nu_x^\infty \rangle = G(\nabla u(x))= G\left( \overline{\nu}_x+ \lambda^a \overline{\nu}_x^\infty \right) \quad \Leb\text{ a.e.} \label{equality 1 G infty and nux}\\
\quad & \langle G^\infty,\nu_x^\infty \rangle= G^\infty\left(\overline{\nu}_x^\infty\right) \quad |D^s u|\text{ a.e.} \label{equality 2 G infty and nux}\\
\quad & \lambda^*=0, \quad\text{that is,}\quad \lambda^s=\frac{\lambda^s}{|D^su|}\,|D^su| .
\end{align}
where for (\ref{equality 2 G infty and nux}) we used $ G^\infty\left( \dfrac{D^s u}{|D^s u|}(x) \right)= \dfrac{\lambda^s}{|D^s u|}G^\infty(\overline{\nu}_x^\infty)$ by (\ref{D^s u and nu_x^infty relation}) and the positive homogeneity of $G^\infty$.\\ \ \\
From (\ref{equality 2 G infty and nux}) we get
$$
G^\infty(\overline{\nu}_x^\infty)= \int_{\mathbb S_G} \underbrace{G^\infty}_{=1 \text{ on }\mathbb S_G} \dif \nu_x^\infty =\nu_x^\infty(\mathbb S_G)=1 
$$
since $\nu_x^\infty$ is a probability measure on $\mathbb S_G$. This means that $|D^s u|$ a.e. 
$$
\overline{\nu}_x^\infty \in \mathbb S_G.
$$
Recall that $\SG=\{G^\infty=1\}$. Applying $G^\infty$ to
\eqref{D^s u and nu_x^infty relation} now gives
$\frac{\lambda^s}{|D^su|}=G^\infty\bigl(\frac{D^su}{|D^su|}\bigr)$, so that
$\overline{\nu}^\infty_x$ is the point where the ray through
$\frac{D^su}{|D^su|}(x)$ meets $\SG$, that is, the $G$-polar of $D^su$
(Definition~\ref{def:Gvar}). On the other hand,
$$
\theta_x:=\overline{\nu}_x^\infty=\int_{\SG}z\dif\nu_x^\infty
\in\overline{\mathrm{co}}(\SG),
$$
so the structure of $\nu^\infty_x$ is governed by the position of $\theta_x$
in the convex body $\overline{\mathrm{co}}(\SG)=\{G^{\infty\ast\ast}\le1\}$.
We use the notation of Section~\ref{wave cone section}: $K_G$ and $N_G$ are
the contact set and the dent (Definition~\ref{def:dent}); $\mathcal{L}_\theta$
is the set of linear $\ell\le G^\infty$ with $\ell(\theta)=1$, each of which
gives a supporting hyperplane $H=\{\ell=1\}$ of $\overline{\mathrm{co}}(\SG)$
through $\theta$; and $F_\theta$ is the trace on $\SG$ of the smallest exposed
face of $\overline{\mathrm{co}}(\SG)$ containing $\theta$
(Definition~\ref{def:face}).
\begin{figure}[H]
  \centering
\begin{tikzpicture}[scale=1.5]
    \draw[thick] (-0.8, 1.8)
        .. controls (-2.5, 1.5) and (-2.5, -1.0) .. (-0.8, -1.0)
        .. controls (-0.2, -1.0) and (-0.5, 0.0) .. (0.5, 0.0)
        .. controls (1.3, 0.0) and (1.2, -1.0) .. (1.8, -0.8)
        .. controls (2.3, -0.4) and (2.25, 0.1) .. (2.3, 0.4)
        -- (2.03, 0.86)
        .. controls (1.4, 0.9) and (1.4, 1.2) .. (1.76, 1.34)
        -- (1.5, 1.8)
        .. controls (1.0, 1.9) and (-0.2, 2.0) .. (-0.8, 1.8);
    \draw[dashed, gray, thin, shorten <=-1.0cm, shorten >=-1.0cm] (2.3, 0.4) -- (1.5, 1.8)
        node[pos=1.5, black, anchor=south west] {$H$};
    \node[rotate=12, anchor=south, inner sep=3pt] at (-1.1, 1.78) {$G^\infty = 1$};
    \draw[line width=1.8pt, samp] (2.3, 0.4) -- (2.03, 0.86);
    \draw[line width=1.8pt, samp] (1.76, 1.34) -- (1.5, 1.8);
    \draw[hlc, thick, fill=white] (2.18, 0.604) circle (1.8pt);
    \draw[hlc, thick, fill=white] (1.557, 1.699) circle (1.8pt);
    \node at (1.8,0.6) {$F_{\theta_2}$};
    \fill (-1.9,1) circle (1pt) node[right]{$\theta_1$};
    \fill (1.695,1.455) circle (1pt) node[left]{$\theta_2$};
    \fill (0.1,-0.07) circle (1pt) node[above]{$\theta_3$};
\end{tikzpicture}
\caption{The three positions of $\theta_x=\overline{\nu}^\infty_x$ on
$\SG=\{G^\infty=1\}$: $\theta_1$ is an extreme point of
$\overline{\mathrm{co}}(\SG)$ (Case~1); $\theta_2$ lies on a face, whose trace
$F_{\theta_2}=\SG\cap H$ on $\SG$ consists of the two thick segments, and the
two circles are the atoms of a possible $\nu^\infty_x$ with barycentre
$\theta_2$ (Case~2); $\theta_3$ lies in the dent $N_G$ (Case~3).}
\label{support sets for nu_x infty}
\end{figure}
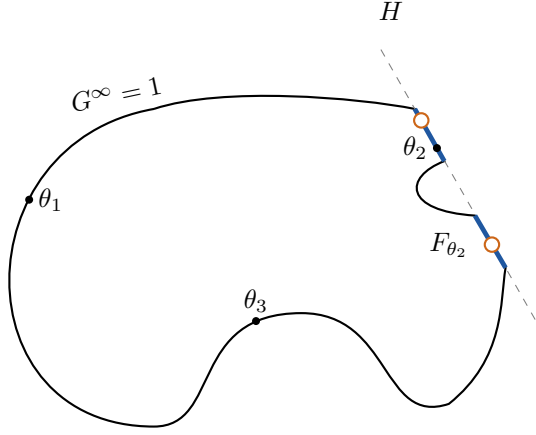
At $|D^su|$-a.e.\ $x$ there are three possibilities, depending on where
$\theta_x$ lies.
\begin{enumerate}[label=\emph{Case~\arabic*.}, leftmargin=*]
\item $\theta_x$ is an \emph{extreme point} of $\overline{\mathrm{co}}(\SG)$,
a point of strict convexity of $G^\infty$. An extreme point is the barycentre
of no probability measure on $\overline{\mathrm{co}}(\SG)$ other than its
Dirac mass, so $\nu^\infty_x=\delta_{\theta_x}$.
\item $\theta_x$ lies on a proper face of $\overline{\mathrm{co}}(\SG)$
without being extreme. For every $\ell\in\mathcal{L}_{\theta_x}$, with
supporting hyperplane $H=\{\ell=1\}$,
$$
1=\ell(\theta_x)=\int_{\SG}\ell\dif\nu^\infty_x
\le\int_{\SG}G^\infty\dif\nu^\infty_x=1,
$$
so the continuous function $G^\infty-\ell\ge0$ vanishes on
$\operatorname{supp}\nu^\infty_x$, that is,
$\operatorname{supp}\nu^\infty_x\subseteq\SG\cap H$. Intersecting over
$\ell\in\mathcal{L}_{\theta_x}$ gives
$\operatorname{supp}\nu^\infty_x\subseteq F_{\theta_x}$, and no further
restriction holds: rigidity degrades from a point to a face. Case~1 is
characterized by extremality, whereas $F_{\theta_x}=\{\theta_x\}$ holds
precisely at the (strictly stronger) \emph{exposed} points
(Remark~\ref{rmk:sharp}).
\item $\theta_x$ lies in the interior of $\overline{\mathrm{co}}(\SG)$, that
is, $\theta_x\in N_G$. Then $\mathcal{L}_{\theta_x}=\emptyset$, and the
support of $\nu^\infty_x$ is not restricted: an interior point of
$\overline{\mathrm{co}}(\SG)$ is the barycentre of many probability measures
on $\SG$.
\end{enumerate}
Surprisingly, Case~3 is vacuous, because of the structure of the wave cone
$\WC$: by the De Philippis--Rindler theorem
\cite[Thm.~1.1]{DePhilippisRindler16} the barycentre $\theta_x$ lies in $\WC$,
while by the automatic convexity result of Kirchheim and Kristensen
\cite[Thm.~1.1]{KirchheimKristensen16} the dent does not meet $\WC$
(Lemma~\ref{lem:dent}). The next section, after fixing the notation and the
most general assumptions, makes this precise.

\section{The \texorpdfstring{$G$}{G}-compactification and the wave cone}\label{wave cone section}

As in Section~\ref{notation section}, $\V$ is a finite dimensional real inner product space and
$$
  \mathscr{A} \;=\; \sum_{|\alpha| = k} A_{\alpha}\,\partial^{\alpha},
  \qquad A_{\alpha}\in\operatorname{Lin}(\V,\mathbb{W}),
$$
is a homogeneous linear differential operator of order $k$ with constant
coefficients, with principal symbol
$\mathscr{A}(\xi) := \sum_{|\alpha|=k}\xi^{\alpha}A_{\alpha}$ and wave cone
$$
  \WC \;:=\; \bigcup_{\xi\in\mathbb{R}^{n}\setminus\{0\}} \ker \mathscr{A}(\xi).
$$

\begin{itemize}
  \item[\textbf{(A1)}] $\WC$ spans $\V$.
  \item[\textbf{(A2)}] $G^{\infty}\colon \V \to [0,\infty)$ is positively
    $1$-homogeneous, $\mathscr{A}$-quasiconvex, and coercive: there are
    $0 < c \le C < \infty$ with $c|z| \le G^{\infty}(z) \le C|z|$ on $\V$.
  \item[\textbf{(A3)}] $G\colon \V\to\mathbb{R}$ is continuous,
    $\mathscr{A}$-quasiconvex, of linear growth, and admits $G^{\infty}$ as
    \emph{regular} recession function:
    $$
      G^{\infty}(z)
      = \lim_{\substack{z'\to z\\ t\to\infty}}\frac{G(t z')}{t}
      \qquad\text{for every } z\in\V .
    $$
\end{itemize}

Assumption \textup{(A1)} holds for $\mathscr{A} = \operatorname{curl}$,
$\operatorname{div}$, $\operatorname{curl}\operatorname{curl}$, but fails for
degenerate operators such as $v\mapsto\nabla v_{1}$; it is used in an
essential way in Lemma~\ref{lem:dent}. Assumption \textup{(A2)} guarantees
that the gauge below is a genuine compactification, and \textup{(A3)} says
precisely that $G\in\E_{G}(\Omega,\V)$, so that $G$ is an admissible
test integrand for the pairing introduced below.

\begin{remark}[$G$ cannot be convex in the regime of interest]\label{rmk:noconv}
The recession function of a convex function is convex, so \textup{(A3)} forces
$G$ to be non-convex as soon as $N_{G}\neq\emptyset$
(Definition~\ref{def:dent}). The interesting regime therefore requires a
genuinely $\mathscr{A}$-quasiconvex, non-convex integrand of linear growth;
non-convex positively $1$-homogeneous quasiconvex functions do exist
\cite{Muller92}, and by Lemma~\ref{lem:dent} their non-convexity is confined
to directions off the wave cone.
\end{remark}

\subsection{The compactification}

Recall from Section~\ref{compactification section} that
$$
  \BG = \{G^{\infty}< 1\},
  \qquad
  \overline{\BG} = \{G^{\infty}\le 1\},
  \qquad
  \SG = \{G^{\infty} = 1\},
  \qquad
  \iota_{G}(z) = \frac{z}{1+G^{\infty}(z)} .
$$
By \textup{(A2)}, $\iota_{G}$ is a homeomorphism of $\V$ onto the open set
$\BG$, so $\overline{\BG}$ is a compactification of $\V$ whose boundary at
infinity is the compact set $\SG$. The corresponding class of test integrands
is
$$
  \E_{G}(\Omega,\V)
  = \bigl\{ f\in \CC(\overline{\Omega}\times \V) :
     f(x,z) = \bigl(1+G^{\infty}(z)\bigr)\,g\bigl(x,\iota_{G}(z)\bigr)
     \ \text{ for some } g\in \CC(\overline{\Omega}\times \overline{\BG})\bigr\},
$$
and $f^{\infty}(x,z) := \lim_{t\to\infty}t^{-1}f(x,tz)$ exists for $z\in\SG$,
with $f^{\infty} = g$ on $\overline{\Omega}\times\SG$. Since
$c(1+|z|)\wedge 1 \lesssim 1+G^{\infty}(z)\lesssim C(1+|z|)$, one has
$\E_{G}(\Omega,\V) = \E(\Omega,\V)$ as sets; only the
normalisation of the boundary data differs.

\begin{remark}[Equivalence with the ball compactification]\label{rmk:compact}
Write $\iota(z) := z/(1+|z|)$ for the embedding underlying the usual ball
compactification $\overline{B}$ of $\V$, whose boundary at infinity is
$\mathbb{S}_{\V}$. The two compactifications are \emph{equivalent}, i.e.\ the
identity of $\V$ extends to a homeomorphism $\overline{B}\to\overline{\BG}$: explicitly
$\Theta := \iota_{G}\circ\iota^{-1}$ is given, together with its inverse, by
$$
  \Theta(w) = \frac{w}{1-|w|+G^{\infty}(w)},
  \qquad
  \Theta^{-1}(v) = \frac{v}{1+|v|-G^{\infty}(v)} ,
$$
both denominators being strictly positive on $\overline{B}\setminus\{0\}$ and
on $\overline{\BG}\setminus\{0\}$ respectively, by \textup{(A2)}. Indeed, along the ray
$w = s\theta$ with $\theta\in\mathbb{S}_{\V}$ and $g := G^{\infty}(\theta)$,
one has $\Theta(s\theta) = s\theta/(1+s(g-1))$, whose derivative in $s$ equals
$(1+s(g-1))^{-2}>0$, so $\Theta$ maps $[0,\theta]$ increasingly onto
$[0,\theta/g]$; letting $s\uparrow1$ shows that $\Theta$ restricts on
$\mathbb{S}_{\V}$ to the boundary map
$$
  \Pi(z) := \frac{z}{G^{\infty}(z)} ,
$$
used in Remark~\ref{rmk:transfer} below.

It is worth recording that the naive radial map
$$
  \Psi(z) := |z|\,\frac{z}{G^{\infty}(z)},
  \qquad
  \Psi^{-1}(w) = G^{\infty}(w)\,\frac{w}{|w|},
$$
is a homeomorphism of $\overline{B}$ onto $\overline{\BG}$ restricting to $\Pi$ on
$\mathbb{S}_{\V}$ as well --- it is positively $1$-homogeneous, hence a
homeomorphism of all of $\V$, and the factor $|z|$ is exactly what cancels the
homogeneity of $G^{\infty}$ in the denominator. It is, however, \emph{not} an
equivalence of compactifications: for $z = t\theta$,
$$
  \Psi\bigl(\iota(z)\bigr) = \frac{t\,\theta}{(1+t)\,g}
  \;\neq\;
  \frac{t\,\theta}{1+t g} = \iota_{G}(z)
  \qquad\text{unless } g = 1 ,
$$
so $\Psi$ moves the interior copy of $\V$. Only $\Theta$ fixes it pointwise.
Since all that is used below is the boundary restriction $\Pi$, either map
serves for Remark~\ref{rmk:transfer}; $\Theta$ is the one that justifies calling
$\overline{\BG}$ and $\overline{B}$ the same compactification of $\V$.
\end{remark}

A generalized Young measure in this normalisation is a triple
$\nu = \bigl((\nu_{x})_{x\in\Omega},\lambda,(\nu^{\infty}_{x})_{x\in\overline{\Omega}}\bigr)$
with $\nu_{x}\in\mathscr{M}^{+}_{1}(\V)$, $\lambda\in\Meas^{+}(\overline{\Omega})$
and, crucially,
$$
  \nu^{\infty}_{x}\in\mathscr{M}^{+}_{1}(\SG)
  \qquad\text{for } \lambda\text{-a.e.\ } x,
$$
paired with $f\in\E_{G}(\Omega,\V)$ by
$$
  \ddual{f}{\nu}
  := \int_{\Omega}\!\int_{\V} f(x,z)\,\mathrm{d}\nu_{x}(z)\,\mathrm{d}x
   + \int_{\overline{\Omega}}\!\int_{\SG} f^{\infty}(x,z)\,
     \mathrm{d}\nu^{\infty}_{x}(z)\,\mathrm{d}\lambda(x).
$$

\begin{lemma}[Transfer from the spherical normalisation]\label{rmk:transfer}
Let $\Pi(z) := z/G^{\infty}(z)$ denote the restriction of $\Theta$ to
$\mathbb{S}_{\V}$, a homeomorphism onto $\SG$
(Remark~\ref{rmk:compact}), and let
$(\nu_{x},\lambda^{\mathbb{S}},\nu^{\infty,\mathbb{S}}_{x})$ be the triple representing
$\nu$ in the spherical normalisation of \cite{KristensenRindler10}. Set
$$
  m(x) := \int_{\mathbb{S}_{\V}} G^{\infty}\,\mathrm{d}\nu^{\infty,\mathbb{S}}_{x}
  \;\in\; [c,C] .
$$
Then the triple representing the \emph{same} $\nu$ in the $G$-normalisation is
\begin{equation}\label{eq:transfer}
  \nu^{G}_{x} = \nu_{x},
  \qquad
  \lambda = m\,\lambda^{\mathbb{S}},
  \qquad
  \nu^{\infty}_{x}
  = \Pi_{\#}\Bigl(\frac{G^{\infty}\,\nu^{\infty,\mathbb{S}}_{x}}{m(x)}\Bigr) .
\end{equation}
In particular $c\lambda^{\mathbb{S}}\le\lambda\le C\lambda^{\mathbb{S}}$, so $\lambda^{\mathbb{S}}$ and
$\lambda$ are mutually absolutely continuous with bounded densities and
share their Lebesgue decompositions; moreover
$\overline{\nu}^{\infty}_{x} = m(x)^{-1}\overline{\nu}^{\infty,\mathbb{S}}_{x}$ and
$\overline{\nu}^{\infty}_{x}\lambda = \overline{\nu}^{\infty,\mathbb{S}}_{x}\lambda^{\mathbb{S}}$.
\end{lemma}

\begin{proof}
Write $\boldsymbol\mu\in\Meas^{+}(\overline{\Omega}\times\overline{B})$ and
$\boldsymbol\mu_{G}\in\Meas^{+}(\overline{\Omega}\times\overline{\BG})$ for the measures
representing $\nu$ under the isometries
$\E(\Omega,\V)\cong \CC(\overline{\Omega}\times\overline{B})$ and
$\E_{G}(\Omega,\V)\cong \CC(\overline{\Omega}\times\overline{\BG})$, and disintegrate
$$
  \boldsymbol\mu = \kappa\otimes\eta_{x},
  \qquad \kappa := \pi_{\#}\boldsymbol\mu,
  \qquad \eta_{x}\in\mathscr{M}^{+}_{1}(\overline{B}),
$$
$\pi(x,w) := x$. Comparing $\int g\,\mathrm{d}\boldsymbol\mu$ with
$\ddual{f}{\nu}$ and using $1-|w| = (1+|\iota^{-1}(w)|)^{-1}$ on $B$ gives the
standard formulas
\begin{equation}\label{eq:explicit}
  \lambda^{\mathbb{S}} = \eta_{x}(\mathbb{S}_{\V})\,\kappa,
  \quad
  \nu^{\infty,\mathbb{S}}_{x} = \frac{\eta_{x}\llcorner\mathbb{S}_{\V}}
                                     {\eta_{x}(\mathbb{S}_{\V})},
  \quad
  \nu_{x} = \frac{(\iota^{-1})_{\#}\bigl[(1-|w|)\,\eta_{x}\llcorner B\bigr]}
                 {m_{0}(x)},
  \quad
  m_{0}(x) := \int_{B}(1-|w|)\,\mathrm{d}\eta_{x},
\end{equation}
together with the normalisation $\Leb\llcorner\Omega = m_{0}\,\kappa$. The same
formulas, with $\iota_{G}$, $\SG$ and $1-G^{\infty}$ in place of $\iota$,
$\mathbb{S}_{\V}$ and $1-|w|$, produce the $G$-triple from
$\boldsymbol\mu_{G} = \kappa_{G}\otimes\eta^{G}_{x}$.\\ First, we check that both the base measure and the fibre measures differ in the two disintegrations. By Remark~\ref{rmk:compact},
$\boldsymbol\mu_{G} = (\mathrm{id}\otimes\Theta)_{\#}(J\boldsymbol\mu)$ with
$J(w) = 1-|w|+G^{\infty}(w)$, so
$$
  \kappa_{G} = \pi_{\#}\bigl(J\boldsymbol\mu\bigr) = j\,\kappa,
  \qquad
  j(x) := \int_{\overline{B}} J\,\mathrm{d}\eta_{x},
  \qquad
  \eta^{G}_{x} = \frac{\Theta_{\#}\bigl(J\,\eta_{x}\bigr)}{j(x)} ,
$$
the last identity by uniqueness of the disintegration with probability
fibres, $\eta^{G}_{x}$ being a probability measure by construction. Since $J$
is continuous and positive on the compact set $\overline{B}$, one has
$0 < \delta \le j \le 1+C$. Thus \emph{both} the base measure and the fibre
measures change.\\
Now we check how the concentration part changes. As $\Theta$ maps $B$ onto
$\BG$ and $\mathbb{S}_{\V}$ onto $\SG$, and
$J|_{\mathbb{S}_{\V}} = G^{\infty}$,
$$
  \lambda
  = \eta^{G}_{x}(\SG)\,\kappa_{G}
  = \frac{(J\eta_{x})(\mathbb{S}_{\V})}{j(x)}\cdot j(x)\,\kappa
  = \Bigl(\int_{\mathbb{S}_{\V}}G^{\infty}\,\mathrm{d}\eta_{x}\Bigr)\kappa
  = m\,\eta_{x}(\mathbb{S}_{\V})\,\kappa
  = m\,\lambda^{\mathbb{S}} ,
$$
where the factor $j(x)$ introduced in $\kappa_{G}$ cancels against the one
normalising $\eta^{G}_{x}$, and the last two equalities use
\eqref{eq:explicit}. The same cancellation gives
$$
  \nu^{\infty}_{x}
  = \frac{\eta^{G}_{x}\llcorner\SG}{\eta^{G}_{x}(\SG)}
  = \frac{\Pi_{\#}\bigl(G^{\infty}\,\eta_{x}\llcorner\mathbb{S}_{\V}\bigr)}
         {\int_{\mathbb{S}_{\V}}G^{\infty}\,\mathrm{d}\eta_{x}}
  = \Pi_{\#}\Bigl(\frac{G^{\infty}\,\nu^{\infty,\mathbb{S}}_{x}}{m(x)}\Bigr).
$$

Finally, we check the oscillation part actually stays the same. For $w\in B$ and $v = \Theta(w)$ one has
$1-G^{\infty}(v) = (1-|w|)/J(w)$, hence
$\bigl(1-G^{\infty}(v)\bigr)J(w) = 1-|w|$; moreover
$\iota_{G}^{-1}\circ\Theta = \iota^{-1}$ because $\Theta = \iota_{G}\circ\iota^{-1}$.
Therefore
$$
  (\iota_{G}^{-1})_{\#}\Bigl[\bigl(1-G^{\infty}\bigr)\,
     \eta^{G}_{x}\llcorner\BG\Bigr]
  = \frac{1}{j(x)}\,(\iota^{-1})_{\#}\bigl[(1-|w|)\,\eta_{x}\llcorner B\bigr],
$$
and likewise $m^{G}_{0} = m_{0}/j$; normalising, $\nu^{G}_{x} = \nu_{x}$. (The
normalisation $\Leb\llcorner\Omega = m^{G}_{0}\kappa_{G} = m_{0}\kappa$ is
preserved as well.) This is as it must be: $\nu_{x}$ is a measure on $\V$
itself and cannot depend on the chosen compactification.

Finally $c\le m\le C$ by \textup{(A2)}, whence the stated comparison of
$\lambda^{\mathbb{S}}$ and $\lambda$, and the two barycentre identities follow from
\eqref{eq:transfer} and the $1$-homogeneity of the identity map.
\end{proof}

\begin{remark}\label{rmk:jensen-transfer}
The pair $(\kappa,\eta_{x})$ is an auxiliary normalisation: it is the product
$\lambda^{\mathbb{S}}\otimes\nu^{\infty,\mathbb{S}}_{x}$, not its two factors separately,
that is canonically attached to $\nu$. This is the structural reason for the
cancellation of $j(x)$ in Step 2 above: any quantity of the form
``fibre mass $\times$ base measure'' is blind to it.

As a consequence of Lemma~\ref{rmk:transfer} we may use the Jensen-type
conditions of \cite[Thm.~1]{KristensenRindler10}, (i)--(ii) in the
restatement \cite[Thm.~6.1]{KirchheimKristensen16}, and those of
\cite{KristensenRaita22} for general $\mathscr{A}$, directly in the
$G$-normalisation. Indeed, for $1$-homogeneous $f^{\infty}$ both sides of each
condition are divided by $m(x)$ under \eqref{eq:transfer}, and
$\lambda^{\mathbb{S}}\ll\lambda\ll\lambda^{\mathbb{S}}$ with bounded densities, so the two
normalisations have the same singular parts and the same null sets; hence the
conditions hold verbatim, and we use them in that form below.
\end{remark}

The following identity, immediate from $\nu^{\infty}_{x}\in\mathscr{M}^{+}_{1}(\SG)$,
is what makes the $G$-normalisation the natural one here:
\begin{equation}\label{eq:unitmass}
  \dual{G^{\infty}}{\nu^{\infty}_{x}} = 1
  \qquad\text{for }\lambda\text{-a.e.\ } x .
\end{equation}
In words: $\lambda$ \emph{is} the $G$-concentration measure, and the
concentration contribution to the $G$-energy is exactly
$\lambda(\overline{\Omega})$.

\subsection{Contact set, dent, faces}

\begin{definition}[Contact set and dent]\label{def:dent}
Let $G^{\infty\ast\ast}$ be the convex envelope of $G^{\infty}$ and set
$$
  K_{G} := \{ z\in \V : G^{\infty}(z) = G^{\infty\ast\ast}(z)\},
  \qquad
  N_{G} := \SG\setminus K_{G}.
$$
We call $K_{G}$ the \emph{contact set} and $N_{G}$ the \emph{dent}. By
$1$-homogeneity $\{G^{\infty\ast\ast}\le 1\} = \overline{\mathrm{co}}\,\BG
= \overline{\mathrm{co}}(\SG)$, the last equality because every point of $\BG$
lies on a segment joining two points of $\SG$. Thus $N_{G}$ consists exactly
of the points of $\SG$ lying in the interior of $\overline{\mathrm{co}}(\SG)$;
one has $N_{G} = \emptyset$ if
and only if $G^{\infty}$ is convex.
\end{definition}

\begin{definition}[Exposing functionals and faces]\label{def:face}
For $\theta\in\SG$ put
$$
  \mathcal{L}_{\theta}
  := \bigl\{\, \ell \in \V^{\ast} : \ell \le G^{\infty} \text{ on } \V,
     \ \ \ell(\theta) = 1 \,\bigr\},
  \qquad
  F_{\theta} := \bigcap_{\ell\in\mathcal{L}_{\theta}}
                \bigl\{ z\in\SG : \ell(z) = 1 \bigr\} .
$$
Since a linear $\ell$ satisfies $\ell\le G^{\infty}$ if and only if
$\ell\le G^{\infty\ast\ast}$, the set $\mathcal{L}_{\theta}$ is the
subdifferential of the gauge $G^{\infty\ast\ast}$ at $\theta$ whenever
$\theta\in K_{G}$, and $F_{\theta}$ is the trace on $\SG$ of the smallest
exposed face of $\overline{\mathrm{co}}(\SG)$ containing $\theta$. Note
$F_{\theta}\subseteq K_{G}$ automatically: $z\in F_{\theta}$ gives
$1 = \ell(z)\le G^{\infty\ast\ast}(z)\le G^{\infty}(z) = 1$.
\end{definition}

\begin{lemma}[The dent avoids the wave cone]\label{lem:dent}
Assume \textup{(A1)}, \textup{(A2)}. Then $\mathcal{L}_{\theta}\neq\emptyset$
for every $\theta\in\WC\cap\SG$, and consequently
$$
  \WC \subseteq K_{G},
  \qquad\text{in particular}\qquad
  N_{G}\cap\WC = \emptyset .
$$
\end{lemma}

\begin{proof}
An $\mathscr{A}$-quasiconvex function is $\WC$-convex \cite[Prop.~3.4]{FonsecaMuller99}.
By \textup{(A1)} the cone $\WC$ is balanced and spanning, so
\cite[Thm.~1.1]{KirchheimKristensen16} applied with $\mathcal{C}=\V$ and
$\mathcal{D}=\WC$ yields, for each $\theta\in\WC$, a linear
$\ell\colon \V\to\mathbb{R}$ with $\ell(\theta)=G^{\infty}(\theta)$ and
$\ell\le G^{\infty}$ \emph{on all of} $\V$. For $\theta\in\SG$ this is exactly
$\ell\in\mathcal{L}_{\theta}$. Since $\ell$ is linear,
$\ell\le G^{\infty\ast\ast}\le G^{\infty}$, so
$G^{\infty}(\theta) = \ell(\theta)\le G^{\infty\ast\ast}(\theta)\le G^{\infty}(\theta)$,
i.e.\ $\theta\in K_{G}$; the general case follows by homogeneity.
\end{proof}

\begin{remark}
The globality of the minorant is what is used below: $\nu^{\infty}_{x}$ is a
probability measure on all of $\SG$ and is not a priori concentrated near
$\theta$, so a purely local convexity statement at $\theta$ would carry no
information about its support.
\end{remark}

\section{A \texorpdfstring{$G$}{G}-strict rigidity principle}\label{G strict rigidity principle}

\begin{definition}[$G$-variation, $G$-polar, $G$-strict convergence]
\label{def:Gvar}
For $\mu\in\Meas(\Omega,\V)$, let $\sigma := \mathrm{d}\mu/\mathrm{d}|\mu|$ be the
polar of $\mu$, so that $\mu = \sigma|\mu|$ with $|\sigma| = 1$ $|\mu|$-a.e.
\cite[Cor.~1.29]{AFP}, and set
$$
  |\mu|_{G} := G^{\infty}(\sigma)\,|\mu| \in\Meas^{+}(\Omega),
  \qquad
  p_{G} := \frac{\mathrm{d}\mu}{\mathrm{d}|\mu|_{G}} .
$$
We say $\mu_{j}\to\mu$ \emph{$G$-strictly} if
$\mu_{j}\overset{\ast}{\rightharpoonup}\mu$ in $\Meas(\Omega,\V)$ and
$\mathcal{G}(\mu_{j})\to\mathcal{G}(\mu)$, where
\begin{equation}\label{eq:energy}
  \mathcal{G}(\mu)
  := \int_{\Omega} G\Bigl(\frac{\mathrm{d}\mu}{\mathrm{d}\Leb}\Bigr)\,\mathrm{d}x
   \;+\; \int_{\Omega} G^{\infty}
        \Bigl(\frac{\mathrm{d}\mu^{s}}{\mathrm{d}|\mu^{s}|}\Bigr)\,\mathrm{d}|\mu^{s}|
   = \int_{\Omega} G\Bigl(\frac{\mathrm{d}\mu}{\mathrm{d}\Leb}\Bigr)\,\mathrm{d}x
   \;+\; |\mu^{s}|_{G}(\Omega).
\end{equation}
\end{definition}

Thus the integrand itself acts on the absolutely continuous part and only its
recession function on the singular part, as usual for functionals of linear
growth; the second equality in \eqref{eq:energy} is the first identity in
Lemma~\ref{lem:wellposed}(c).
Equivalently $\mathcal{G}(\mu) = \ddual{G}{\boldsymbol{\eps}_{\mu}}$, where
$\boldsymbol{\eps}_{\mu}$ denotes the elementary generalized Young measure generated by
the constant sequence $\mu$. Its spherical triple is
$\bigl(\delta_{\mathrm{d}\mu/\mathrm{d}\Leb},\,|\mu^{s}|,\,
\delta_{\mathrm{d}\mu^{s}/\mathrm{d}|\mu^{s}|}\bigr)$, so by
Lemma~\ref{rmk:transfer} and Lemma~\ref{lem:wellposed}(b),(c) its
$G$-normalised triple is
$$
  \nu_{x} = \delta_{\mathrm{d}\mu/\mathrm{d}\Leb},
  \qquad
  \lambda = |\mu^{s}|_{G},
  \qquad
  \nu^{\infty}_{x} = \delta_{p_{G}(x)} .
$$
Taking $G = G^{\infty}$, which is admissible in \textup{(A3)}, the second
identity in Lemma~\ref{lem:wellposed}(c) gives
$\mathcal{G} = |\cdot|_{G}(\Omega)$, i.e.\ the exact $G$-analogue of strict
convergence of measures; a general $G$ gives the analogue of area-strict
convergence.

The next lemma records that Definition~\ref{def:Gvar} is unambiguous and has
the expected properties. We stress that only the positive $1$-homogeneity and
the coercivity of $G^{\infty}$ enter; convexity plays no role whatsoever, so
the non-convex case is on exactly the same footing as the classical one
treated in \cite{goffmanserrin}.

\begin{lemma}[Well-posedness of the $G$-variation]\label{lem:wellposed}
Assume \textup{(A2)} and let $\mu\in\Meas(\Omega,\V)$.
\begin{enumerate}
  \item[\textup{(a)}] For every $\rho\in\Meas^{+}(\Omega)$ with $|\mu|\ll\rho$,
    $$
      G^{\infty}\Bigl(\frac{\mathrm{d}\mu}{\mathrm{d}\rho}\Bigr)\,\rho
      \;=\; G^{\infty}(\sigma)\,|\mu| \;=\; |\mu|_{G} ,
    $$
    so $|\mu|_{G}$ does not depend on the dominating measure used to compute it.
  \item[\textup{(b)}] $c\,|\mu| \le |\mu|_{G}\le C\,|\mu|$; in particular
    $|\mu|_{G}$ and $|\mu|$ are mutually absolutely continuous and
    $(|\mu|_{G})^{s} = |\mu^{s}|_{G}$.
  \item[\textup{(c)}] $p_{G} = \sigma/G^{\infty}(\sigma)$ with
    $G^{\infty}(p_{G}) = 1$, i.e.\ $p_{G}(x)\in\SG$, for $|\mu|_{G}$-a.e.\ $x$,
    and
    $$
      \mu = p_{G}\,|\mu|_{G},
      \qquad
      \mu^{s} = p_{G}\,|\mu^{s}|_{G} .
    $$
    Consequently
    $$
      \int_{\Omega} G^{\infty}\Bigl(\frac{\mathrm{d}\mu^{s}}{\mathrm{d}|\mu^{s}|}\Bigr)
        \,\mathrm{d}|\mu^{s}| = |\mu^{s}|_{G}(\Omega),
      \qquad
      |\mu|_{G}(\Omega)
      = \int_{\Omega} G^{\infty}\Bigl(\frac{\mathrm{d}\mu}{\mathrm{d}\Leb}\Bigr)
        \,\mathrm{d}x \;+\; |\mu^{s}|_{G}(\Omega).
    $$
\end{enumerate}
\end{lemma}

\begin{proof}
(a) By the chain rule for Radon--Nikod\'ym derivatives
\cite[Sect.~1.3]{AFP},
$\mathrm{d}\mu/\mathrm{d}\rho = \sigma\,\mathrm{d}|\mu|/\mathrm{d}\rho$
$\rho$-a.e., and $\mathrm{d}|\mu|/\mathrm{d}\rho\ge 0$. Positive
$1$-homogeneity of $G^{\infty}$ therefore gives
$$
  G^{\infty}\Bigl(\frac{\mathrm{d}\mu}{\mathrm{d}\rho}\Bigr)
  = G^{\infty}(\sigma)\,\frac{\mathrm{d}|\mu|}{\mathrm{d}\rho}
  \qquad \rho\text{-a.e.},
$$
and multiplying by $\rho$ yields (a). Note that subadditivity is never
invoked: the scalar factor pulled out of $G^{\infty}$ is nonnegative.

(b) Being finite and $\mathscr{A}$-quasiconvex, $G^{\infty}$ is $\WC$-convex, hence
Lipschitz by \textup{(A1)} (Lemma~\ref{qc properties lemma}); thus $G^{\infty}\circ\sigma$ is Borel. Since
$|\sigma| = 1$ $|\mu|$-a.e., \textup{(A2)} gives
$c\le G^{\infty}(\sigma)\le C$ $|\mu|$-a.e., whence the two-sided bound, and
mutual absolute continuity follows. For the last claim, let $N$ be an
$\Leb$-null Borel set with $\mu^{s} = \mu\llcorner N$, so that
$\mu - \mu^{s} = \mu\llcorner N^{c}\ll\Leb$. Then
$\mu^{s} = \sigma\,|\mu|\llcorner N$ with $|\sigma| = 1$, hence
$|\mu^{s}| = |\mu|\llcorner N$ and
$$
  \frac{\mathrm{d}\mu^{s}}{\mathrm{d}|\mu^{s}|} = \sigma
  \qquad |\mu^{s}|\text{-a.e.}
$$
Definition~\ref{def:Gvar} applied to $\mu^{s}$ thus gives
$|\mu^{s}|_{G} = G^{\infty}(\sigma)\,|\mu|\llcorner N = |\mu|_{G}\llcorner N$.
Since $|\mu|_{G}\llcorner N$ is carried by the $\Leb$-null set $N$ while
$|\mu|_{G}\llcorner N^{c} = G^{\infty}(\sigma)\,|\mu|\llcorner N^{c}\ll\Leb$,
uniqueness of the Lebesgue decomposition yields
$(|\mu|_{G})^{s} = |\mu|_{G}\llcorner N = |\mu^{s}|_{G}$.

(c) By (b) we may take $\rho = |\mu|_{G}$ in the chain rule, so
$p_{G} = \sigma\,\mathrm{d}|\mu|/\mathrm{d}|\mu|_{G}
       = \sigma/G^{\infty}(\sigma)$, and $G^{\infty}(p_{G}) = 1$ by
$1$-homogeneity. Restricting $\mu = p_{G}|\mu|_{G}$ to the set $N$ of the proof
of (b) gives $\mu^{s} = p_{G}\,|\mu|_{G}\llcorner N = p_{G}\,|\mu^{s}|_{G}$.
The first displayed identity is Definition~\ref{def:Gvar} applied to
$\mu^{s}$, integrated over $\Omega$. For the second, split
$|\mu|_{G} = |\mu|_{G}\llcorner N^{c} + |\mu^{s}|_{G}$ by (b), and apply (a)
to $\mu\llcorner N^{c}$ with $\rho = \Leb$, which gives
$|\mu|_{G}\llcorner N^{c} = G^{\infty}(\mathrm{d}\mu/\mathrm{d}\Leb)\,\Leb$.
\end{proof}

\begin{remark}[Two classically equivalent definitions that decouple]
\label{rmk:nodual}
For non-convex $G^{\infty}$ the functional $|\cdot|_{G}(\Omega)$ must be
defined through the polar as above, and not by either of the two standard
equivalent recipes.
\begin{enumerate}
  \item[\textup{(a)}] \emph{No dual representation.} For $\varphi\in
    \CC_{c}(\Omega,\V^{\ast})$ one has
    $\sup\{\int_{\Omega}\varphi\cdot\mathrm{d}\mu :
      (G^{\infty})^{\circ}(\varphi)\le 1\}
      = |\mu|_{G^{\ast\ast}}(\Omega)$,
    the polar body seeing only $\overline{\mathrm{co}}(\SG)$. In
    particular $|\cdot|_{G}(\Omega)$ is not a dual norm and is not
    sequentially weak\nobreakdash-$\ast$ lower semicontinuous on
    $\Meas(\Omega,\V)$; its relaxation there is $|\cdot|_{G^{\ast\ast}}(\Omega)$.
  \item[\textup{(b)}] \emph{No partition definition.} A positively
    $1$-homogeneous function is subadditive if and only if it is convex, so
    there exist $z_{1},z_{2}\in \V$ with
    $G^{\infty}(z_{1}+z_{2}) > G^{\infty}(z_{1})+G^{\infty}(z_{2})$; refining a
    partition may then \emph{decrease} $\sum_{i}G^{\infty}(\mu(E_{i}))$, and
    the supremum over finite Borel partitions strictly exceeds
    $|\mu|_{G}(\Omega)$ in general.
\end{enumerate}
Neither failure is a defect of Definition~\ref{def:Gvar}: it is precisely the
phenomenon under study. Lower semicontinuity is restored \emph{along
$\mathscr{A}$-free sequences} \cite{ArroyoRabasaDePhilippisRindler20}, which is also
what makes the Jensen inequality with $G^{\infty}$, rather than with
$G^{\infty\ast\ast}$, available in the proof of
Proposition~\ref{prop:rigidity}.
\end{remark}

\begin{proposition}[$G$-strict rigidity of the concentration part]
\label{prop:rigidity}
Assume \textup{(A1)}--\textup{(A3)}, and let $\Omega\subset\mathbb{R}^{n}$ be a
bounded Lipschitz domain. Let $(\mu_{j})\subset\Meas(\Omega,\V)$ satisfy
$\mathscr{A}\mu_{j}=0$ and $\sup_{j}|\mu_{j}|(\Omega)<\infty$, and let $(\mu_{j})$
generate $\nu = \bigl(\nu_{x},\lambda,\nu^{\infty}_{x}\bigr)$ in the
$G$-normalisation, with $\lambda(\partial\Omega)=0$ and barycentre
$\overline{\nu} = \mu$. Decompose
$$
  \lambda \;=\; \lambda^{a}\,\Leb \;+\; b_{G}\,|\mu^{s}|_{G} \;+\; \lambda^{\ast},
  \qquad \lambda^{\ast}\perp\bigl(\Leb + |\mu^{s}|_{G}\bigr).
$$
Comparing the Lebesgue decompositions of the two sides of
$\mu = \overline{\nu}_{x}\Leb + \overline{\nu}^{\infty}_{x}\lambda$ and using
$\mu^{s} = p_{G}|\mu^{s}|_{G}$ (Lemma~\ref{lem:wellposed}(c)), one finds
$\mathrm{d}\mu/\mathrm{d}\Leb = \overline{\nu}_{x}+\lambda^{a}\overline{\nu}^{\infty}_{x}$ $\Leb$-a.e.,
$p_{G} = b_{G}\,\overline{\nu}^{\infty}_{x}$ $|\mu^{s}|_{G}$-a.e., and
$\overline{\nu}^{\infty}_{x} = 0$ $\lambda^{\ast}$-a.e. If
$\mu_{j}\to\mu$ $G$-strictly, then:
\begin{enumerate}
  \item[\textup{(i)}] $\lambda^{\ast} = 0$ and $b_{G}\equiv 1$; equivalently
    $\lambda^{s} = |\mu^{s}|_{G}$ and
    $$
      \theta_{x} := \overline{\nu}^{\infty}_{x} = p_{G}(x)\in\SG
      \qquad\text{for } |\mu^{s}|_{G}\text{-a.e.\ } x ,
    $$
    i.e.\ the barycentre returns to the level set $\SG$;
  \item[\textup{(ii)}] $\theta_{x}\in\WC\cap\SG$ and
    $$
      \operatorname{supp}\nu^{\infty}_{x}\;\subseteq\;F_{\theta_{x}}
      \;\subseteq\;K_{G}\cap\SG
      \qquad\text{for } |\mu^{s}|_{G}\text{-a.e.\ } x ;
    $$
  \item[\textup{(iii)}] if moreover
    \begin{equation}\tag{$\star$}\label{eq:star}
      \bigl\{\eta\in\mathscr{M}^{+}_{1}(F_{\theta}) : \overline\eta = \theta\bigr\}
      = \{\delta_{\theta}\}
      \qquad\text{for every }\theta\in\WC\cap\SG,
    \end{equation}
    (i.e., $\theta$ is an \emph{extreme} point), then $\nu^{\infty}_{x} = \delta_{p_{G}(x)}$ for $|\mu^{s}|_{G}$-a.e.\ $x$:
    the concentration of $\nu$ on the singular part of $\mu$ is elementary,
    and no Euclidean mass is lost there, i.e.
    $$
      \lim_{j\to\infty}|\mu_{j}|(\Omega)
      = \int_{\Omega}\Bigl[\dual{|\cdot|}{\nu_{x}}
        + \lambda^{a}\,\dual{|\cdot|}{\nu^{\infty}_{x}}\Bigr]\,\mathrm{d}x
        + |\mu^{s}|(\Omega) .
    $$
\end{enumerate}
\end{proposition}

\begin{proof}
Since $\lambda(\partial\Omega)=0$, generation gives
$\ddual{G}{\nu} = \lim_{j}\ddual{G}{\boldsymbol{\eps}_{\mu_{j}}}
 = \lim_{j}\mathcal{G}(\mu_{j})$, so the hypothesis reads
$\ddual{G}{\nu} = \mathcal{G}(\mu)$. By \eqref{eq:unitmass} the left-hand
side is
$$
  \ddual{G}{\nu}
  = \int_{\Omega}\dual{G}{\nu_{x}}\,\mathrm{d}x + \lambda(\Omega)
  = \int_{\Omega}\bigl[\dual{G}{\nu_{x}} + \lambda^{a}\bigr]\mathrm{d}x
  + \int_{\Omega} b_{G}\,\mathrm{d}|\mu^{s}|_{G}
  + \lambda^{\ast}(\Omega),
$$
while, by \eqref{eq:energy} and $G^{\infty}(p_{G}) = 1$,
$$
  \mathcal{G}(\mu)
  = \int_{\Omega}G\bigl(\overline{\nu}_{x}+\lambda^{a}\overline{\nu}^{\infty}_{x}\bigr)\mathrm{d}x
  + \int_{\Omega} G^{\infty}(p_{G})\,\mathrm{d}|\mu^{s}|_{G}.
$$
We compare the three groups of terms. Proposition~\ref{sufficiency proposition}
gives the Jensen inequality \eqref{Jensen inequality in sufficiency theorem},
which we apply to $f=G$, and, by its last assertion,
$G^{\infty}(\overline{\nu}^{\infty}_{x})=b_{G}(x)^{-1}\le1$ at
$|\mu^{s}|_{G}$-a.e.\ $x$ and $\overline{\nu}^{\infty}_{x}=0$ at
$\lambda^{\ast}$-a.e.\ $x$. Together with \eqref{eq:unitmass}, this gives
$$
  G\bigl(\overline{\nu}_{x}+\lambda^{a}\overline{\nu}^{\infty}_{x}\bigr)
  \le \dual{G}{\nu_{x}} + \lambda^{a}\,\dual{G^{\infty}}{\nu^{\infty}_{x}}
   = \dual{G}{\nu_{x}} + \lambda^{a}
  \quad \Leb\text{-a.e.},
$$
and, at $\lambda^{s}$-a.e.\ $x$,
$$
  G^{\infty}\bigl(\overline{\nu}^{\infty}_{x}\bigr)
  \le \dual{G^{\infty}}{\nu^{\infty}_{x}} = 1 .
$$
Applying $G^{\infty}$ to $p_{G} = b_{G}\overline{\nu}^{\infty}_{x}$ and using
$1$-homogeneity, $1 = b_{G}\,G^{\infty}(\overline{\nu}^{\infty}_{x})\le b_{G}$.
Finally $\lambda^{\ast}\ge 0$ has no counterpart on the right. Hence every
term on the left dominates the corresponding one on the right, and equality
forces
$$
  \lambda^{\ast}(\Omega) = 0,
  \qquad b_{G} = 1 \ \ |\mu^{s}|_{G}\text{-a.e.},
$$
which is (i); note $b_{G}=1$ is equivalent to
$G^{\infty}(\overline{\nu}^{\infty}_{x}) = 1 = \dual{G^{\infty}}{\nu^{\infty}_{x}}$,
i.e.\ to equality in the second Jensen inequality.

For (ii), by the structure theorem for $\mathscr{A}$-free measures
\cite[Thm.~1.1]{DePhilippisRindler16} the polar
$\mathrm{d}\mu^{s}/\mathrm{d}|\mu^{s}|$ lies in $\WC$ $|\mu^{s}|$-a.e.; as
$p_{G}$ is a positive multiple of it and $|\mu^{s}|_{G}$ and $|\mu^{s}|$ are
mutually absolutely continuous by \textup{(A2)}, we get
$\theta_{x} = p_{G}(x)\in\WC\cap\SG$. Lemma~\ref{lem:dent} provides
$\ell\in\mathcal{L}_{\theta_{x}}$, and then
$$
  1 = \ell(\theta_{x}) = \int_{\SG}\ell\,\mathrm{d}\nu^{\infty}_{x}
    \le \int_{\SG} G^{\infty}\,\mathrm{d}\nu^{\infty}_{x} = 1 ,
$$
so $\int_{\SG}(G^{\infty}-\ell)\,\mathrm{d}\nu^{\infty}_{x} = 0$ with a
continuous nonnegative integrand; hence
$\operatorname{supp}\nu^{\infty}_{x}\subseteq\{z\in\SG:\ell(z)=1\}$.
Intersecting over $\ell\in\mathcal{L}_{\theta_{x}}$ gives
$\operatorname{supp}\nu^{\infty}_{x}\subseteq F_{\theta_{x}}$, and
$F_{\theta_{x}}\subseteq K_{G}$ by Definition~\ref{def:face}.

For (iii), $\nu^{\infty}_{x}$ is a probability measure on $F_{\theta_{x}}$ with
barycentre $\theta_{x}$, hence equals $\delta_{\theta_{x}} = \delta_{p_{G}(x)}$
by \eqref{eq:star}. For the mass identity, test the generation with
$|\cdot|\in\E_{G}(\Omega,\V)$: by (i), $\lambda = \lambda^{a}\Leb +
|\mu^{s}|_{G}$, and on the singular part
$\dual{|\cdot|}{\delta_{p_{G}}}\,|\mu^{s}|_{G} = |p_{G}|\,|\mu^{s}|_{G}
= |\mu^{s}|$ by Lemma~\ref{lem:wellposed}(c).
\end{proof}

\begin{remark}[What is \emph{not} controlled]
$G^{\infty}$ is affine along rays, so in the purely $1$-homogeneous case
$G = G^{\infty}$ the first Jensen inequality is saturated by
$\nu_{x}=\tfrac12\delta_{0}+\tfrac12\delta_{2z_{0}}$, and no rigidity of
$\nu_{x}$ can be expected. For a general integrand $G$ as in \textup{(A3)}
this is no longer so: equality in the first Jensen inequality,
$$
  G\bigl(\overline{\nu}_{x}+\lambda^{a}\overline{\nu}^{\infty}_{x}\bigr)
  = \dual{G}{\nu_{x}} + \lambda^{a} ,
$$
which the proof of Proposition~\ref{prop:rigidity} also delivers, is a genuine
constraint. If $G$ is \emph{strictly $\mathscr{A}$-quasiconvex} at
$\mathrm{d}\mu/\mathrm{d}\Leb(x)$, in the sense that this equality forces
$\nu_{x} = \delta_{\mathrm{d}\mu/\mathrm{d}\Leb(x)}$ and
$\lambda^{a}\overline{\nu}^{\infty}_{x} = 0$, then the whole generated Young measure is
elementary and $\mu_{j}\to\mu$ area-strictly. By Remark~\ref{rmk:noconv} this
hypothesis cannot be supplied by strict convexity of $G$ in the regime
$N_{G}\neq\emptyset$; identifying a workable substitute is the remaining gap
in the oscillation half of the problem, and Proposition~\ref{prop:rigidity} as
stated concerns the concentration part only.
\end{remark}
\begin{remark}[Diffuse concentration]\label{rmk:diffuse}
Proposition~\ref{prop:rigidity} describes the concentration of $\nu$ on the
singular part of $\mu$ only, and its proof uses that there the barycentre
$\theta_x$ lies in the wave cone, where $G^\infty$ is convex by
Theorem~\ref{KK theorem}. At a point of diffuse concentration,
$\lambda^a(x)>0$, the barycentre
$\frac{\mathrm{d}\mu}{\mathrm{d}\Leb}(x)
=\overline{\nu}_x+\lambda^a\overline{\nu}^\infty_x$ need not lie in $\WC$,
and away from the wave cone $\WC$-convexity gives no automatic convexity,
not even automatic $\mathscr{A}$-quasiconvexity. Indeed, for
$\mathscr{A}=\operatorname{curl}$, Cassese \cite[Cor.~1.6]{Cassese26}
constructs a positively $1$-homogeneous rank-one
convex integrand $f\colon\R^{4\times4}\to[0,\infty)$ that is not quasiconvex
at a matrix of rank two, so Theorem~\ref{KK theorem}, which gives convexity
at the matrices of rank at most one, cannot be extended even to
quasiconvexity at matrices of rank two. As observed in
\cite[\S1.4]{Cassese26}, the obstruction comes from oscillation: among the
generalized gradient Young measures with that rank-two barycentre, some with
nontrivial oscillation lower the energy of $f$ strictly below its value at
the barycentre, whereas, by \cite[Thm.~1.10]{GennaioliRindler26}, purely
concentrating sequences of gradients cannot lower the energy of a rank-one
convex integrand that is nonnegative on the rank-one cone and has a strong
recession function, such as $f$. The latter result rests on the
decomposition of concentration Young measures into probability measures
with barycentres in $\WC$ \cite[Thms.~1.6 and~1.7]{GennaioliRindler26}, to
which Theorem~\ref{KK theorem} applies. In the absence of oscillation this
decomposition suffices: at points of diffuse concentration where $\nu_x$ is a
Dirac mass, the concentration-angle measure does not charge the dent either
(Proposition~\ref{prop:diffuse} below). With oscillation, the equality
\eqref{equality 1 G infty and nux} couples $\nu_x$ and $\nu^\infty_x$ at a
barycentre outside the wave cone. Beyond that equality, the proof of
Proposition~\ref{prop:diffuse} uses only the $\WC$-convexity of $G$, which by
Cassese's example cannot control oscillation, so the structure of
$\nu^\infty_x$ there remains tied to the oscillation problem left open in the
previous remark.
\end{remark}
\begin{proposition}[Diffuse concentration without oscillation]\label{prop:diffuse} Assume \textup{(A1)}--\textup{(A3)} and that $\mathscr{A}$ is of first order,
and let $\Omega$, $(\mu_j)$ and $\nu$ be as in
Proposition~\ref{prop:rigidity}, with $\mu_j\to\mu$ $G$-strictly. Then for
$\Leb$-a.e.\ $x\in\Omega$ such that $\lambda^a(x)>0$ and $\nu_x$ is a Dirac
mass:
\begin{enumerate}
  \item[\textup{(i)}] there is a Borel probability measure $\pi_x$ on
    $\mathscr{M}^+_1(\SG)$ such that
    $$
      \nu^\infty_x=\int\tau\,\mathrm{d}\pi_x(\tau),
      \quad\text{i.e.}\quad
      \dual{\varphi}{\nu^\infty_x}=\int\dual{\varphi}{\tau}\,\mathrm{d}\pi_x(\tau)
      \ \text{ for all }\varphi\in\CC(\SG),
    $$
    and, for $\pi_x$-a.e.\ $\tau$, $\overline{\tau}\in\WC\cap\SG$ and
    $\operatorname{supp}\tau\subseteq F_{\overline{\tau}}$. In particular
    $\operatorname{supp}\nu^\infty_x\subseteq K_G\cap\SG$: the
    concentration-angle measure does not charge the dent;
  \item[\textup{(ii)}] if moreover \eqref{eq:star} holds, then
    $\tau=\delta_{\overline{\tau}}$ for $\pi_x$-a.e.\ $\tau$, and
    $\operatorname{supp}\nu^\infty_x\subseteq\WC\cap\SG$.
\end{enumerate}
\end{proposition}

\begin{proof}
As in the proof of
Proposition~\ref{prop:rigidity}, $G$-strict convergence turns each of the
three groups of terms compared there into an equality. The first group (the Lebesgue part) is the
integral of a nonnegative function, so the equality at $x$ reads
\begin{equation}\label{eq:diffuse equality}
  G\bigl(\overline{\nu}_x+\lambda^a\overline{\nu}^\infty_x\bigr)
  =\dual{G}{\nu_x}+\lambda^a
  \qquad\text{for }\Leb\text{-a.e.\ }x\in\Omega .
\end{equation}

Now we want to apply the result in \cite[Thm.~1.6]{GennaioliRindler26} in the spherical normalisation. Let
$(\nu_x,\lambda^{\mathbb S},\nu^{\infty,\mathbb S}_x)$ be the spherical triple
of Lemma~\ref{rmk:transfer}, and put
$s(x):=\frac{\mathrm{d}\lambda^{\mathbb S}}{\mathrm{d}\Leb}(x)$, so that
$s(x)>0$ exactly when $\lambda^a(x)>0$. By the localization principle at regular points
\cite[Prop.~2.24]{ArroyoRabasaDePhilippisRindler20}, which is stated in the
spherical normalisation and applies since the $\mu_j$ are $\mathscr{A}$-free
and $\mathscr{A}$ has constant rank, for $\Leb$-a.e.\ $x$ the homogeneous
Young measure
$\bigl(\nu_x,\,s(x)\Leb\Lcorner Q,\,\nu^{\infty,\mathbb S}_x\bigr)$ on the
unit cube $Q$ is generated by a sequence
$(w_j)\subset\CC^\infty_{\mathrm{per}}(Q,\V)$ with $\mathscr{A}w_j=0$.
Now fix such an $x$ with $s(x)>0$ and $\nu_x=\delta_\xi$. Testing the generation with $\min\{|z-\xi|,1\}$, with
$|z-\xi|$, and with $|z-\xi|\,\psi\bigl(\frac{z-\xi}{|z-\xi|}\bigr)$ for
$\psi\in\CC(\mathbb S_{\V})$, whose recession functions are $0$, $|z|$ and
$|z|\,\psi(z/|z|)$, we find that $w_j\to\xi$ in measure, that
$c_j:=\norm{w_j-\xi}_{\LL^1(Q)}\to s(x)$, and that
$$
  \int_Q\psi\Bigl(\frac{w_j-\xi}{|w_j-\xi|}\Bigr)\frac{|w_j-\xi|}{c_j}\,\mathrm{d}y
  \;\longrightarrow\;\int_{\mathbb S_{\V}}\psi\,\mathrm{d}\nu^{\infty,\mathbb S}_x
  \qquad\text{for all }\psi\in\CC(\mathbb S_{\V}).
$$
Since $\mathscr{A}$ has constant coefficients, $u_j:=(w_j-\xi)/c_j$ is
$\mathscr{A}$-free, with $\norm{u_j}_{\LL^1(Q)}=1$, $u_j\to0$ in measure and
$u_j\wstar\overline{\nu}^{\infty,\mathbb S}_x\Leb\Lcorner Q$, which does not
charge $\partial Q$. Thus $(u_j)$ is a concentrating sequence in the sense of
\cite{GennaioliRindler26}, and the display says that
$\nu^{\infty,\mathbb S}_x$ is its concentration Young measure. By
\cite[Thm.~1.6]{GennaioliRindler26} there is a Borel probability measure $\pi$
on $\mathscr{M}^+_1(\mathbb S_{\V})$, concentrated on
$\{\sigma:\overline{\sigma}\in\WC\}$, such that
$\nu^{\infty,\mathbb S}_x=\int\sigma\,\mathrm{d}\pi(\sigma)$.

In order to get back back to the $G$-normalisation, for
$\sigma\in\mathscr{M}^+_1(\mathbb S_{\V})$ put
$m_\sigma:=\dual{G^\infty}{\sigma}\in[c,C]$ and
$\sigma^G:=\Pi_\#(G^\infty\sigma)/m_\sigma\in\mathscr{M}^+_1(\SG)$, with $\Pi$
as in Lemma~\ref{rmk:transfer}. Then
$\overline{\sigma^G}=\overline{\sigma}/m_\sigma$, which lies in $\WC$ whenever
$\overline{\sigma}$ does, $\WC$ being a cone. By \eqref{eq:transfer} and
linearity,
$$
  \nu^\infty_x=\frac{\Pi_\#\bigl(G^\infty\nu^{\infty,\mathbb S}_x\bigr)}{m(x)}
  =\int\sigma^G\,\frac{m_\sigma}{m(x)}\,\mathrm{d}\pi(\sigma)
  =\int\tau\,\mathrm{d}\pi_x(\tau),
  \qquad
  \pi_x:=(\sigma\mapsto\sigma^G)_\#\Bigl(\frac{m_\sigma}{m(x)}\,\pi\Bigr),
$$
and $\pi_x$ is a probability measure because
$\int m_\sigma\,\mathrm{d}\pi(\sigma)=\dual{G^\infty}{\nu^{\infty,\mathbb S}_x}=m(x)$.
Hence $\overline{\tau}\in\WC$ for $\pi_x$-a.e.\ $\tau$, and, testing with the
coordinate functions, $\theta:=\overline{\nu}^\infty_x=\int\overline{\tau}\,\mathrm{d}\pi_x(\tau)$.

Recall the convex deficiency
integrand (cf.\ Lemma~\ref{lemma: CD_f as sup of differentiability point})
$$
  \mathrm{CD}_G(z)=\sup\Bigl\{\frac{G(w+rz)-G(w)}{r}:w\in\V,\ r>0\Bigr\}.
$$
Since $G$ is Lipschitz (Lemma~\ref{qc properties lemma}(a)), $\mathrm{CD}_G$ is
finite, positively $1$-homogeneous and subadditive, hence convex and
continuous. Moreover $\mathrm{CD}_G= G^\infty$ on $\WC$, since
$G(w+rz)\leq G(w)+rG^\infty(z)$ for $w\in\V$, $z\in\WC$ and $r>0$ by
Lemma~\ref{qc properties lemma}(c). Put $t:=\lambda^a(x)>0$. As $\nu_x=\delta_\xi$,
\eqref{eq:diffuse equality} reads $G(\xi+t\theta)-G(\xi)=t$. Therefore, by
Jensen's inequality for $\mathrm{CD}_G$ and by Theorem~\ref{KK theorem}
applied to each $\tau$,
\begin{multline*}
  1=\frac{G(\xi+t\theta)-G(\xi)}{t}
  \leq\mathrm{CD}_G(\theta)
  \leq\int\mathrm{CD}_G(\overline{\tau})\,\mathrm{d}\pi_x(\tau)\\
  =\int G^\infty(\overline{\tau})\,\mathrm{d}\pi_x(\tau)
  \leq\int\dual{G^\infty}{\tau}\,\mathrm{d}\pi_x(\tau)=1 .
\end{multline*}
All the inequalities are equalities. Since
$\dual{G^\infty}{\tau}-G^\infty(\overline{\tau})\geq0$ for $\pi_x$-a.e.\ $\tau$,
it follows that $G^\infty(\overline{\tau})=\dual{G^\infty}{\tau}=1$, that is
$\overline{\tau}\in\WC\cap\SG$, for $\pi_x$-a.e.\ $\tau$.

Let $\tau$ be as above, and let
$\ell\in\mathcal{L}_{\overline{\tau}}$, which exists by Lemma~\ref{lem:dent}.
As in the proof of Proposition~\ref{prop:rigidity}(ii),
$\int_{\SG}(G^\infty-\ell)\,\mathrm{d}\tau=1-\ell(\overline{\tau})=0$ with a
continuous nonnegative integrand, so
$\operatorname{supp}\tau\subseteq\{z\in\SG:\ell(z)=1\}$; intersecting over
$\ell\in\mathcal{L}_{\overline{\tau}}$ gives
$\operatorname{supp}\tau\subseteq F_{\overline{\tau}}\subseteq K_G$. As
$N_G=\SG\setminus K_G$ is relatively open in $\SG$, approximating its
indicator from below by continuous functions gives
$$
  \nu^\infty_x(N_G)=\int\tau(N_G)\,\mathrm{d}\pi_x(\tau)=0 ,
$$
so $\operatorname{supp}\nu^\infty_x$ lies in the closed set $K_G\cap\SG$. This
proves (i). Under \eqref{eq:star}, each such $\tau$ is a probability measure
on $F_{\overline{\tau}}$ with barycentre $\overline{\tau}\in\WC\cap\SG$, hence
$\tau=\delta_{\overline{\tau}}$; then
$\nu^\infty_x=\int\delta_{\overline{\tau}}\,\mathrm{d}\pi_x(\tau)$ is carried
by the closed set $\WC\cap\SG$, which proves (ii).
\end{proof}
\begin{remark}[Sharpness of \eqref{eq:star}]\label{rmk:sharp}
Since $F_{\theta}$ is compact, \eqref{eq:star} holds at $\theta$ if and only
if $\theta$ is an extreme point of $\operatorname{co}F_{\theta}$: a
non-extreme $\theta$ is, by Carath\'eodory's theorem, a non-trivial finite
convex combination of points of $F_{\theta}$, while an extreme point of a
compact convex set is represented only by its Dirac mass. This is in turn
equivalent to $\theta$ being an extreme point of
$C:=\overline{\mathrm{co}}(\SG)=\{G^{\infty\ast\ast}\le1\}$. Indeed, an extreme point of $C$ is extreme in
$\operatorname{co}F_{\theta}\subseteq C$. Conversely, if $\theta$ is not
extreme in $C$, it lies in the relative interior of its minimal face
$F_{\min}\neq\{\theta\}$ and is a non-trivial convex combination of extreme
points of $F_{\min}$; these are extreme points of $C$, hence lie in
$\{G^{\infty}\le1\}\cap\{G^{\infty\ast\ast}=1\}\subseteq\SG\cap K_{G}$, and they
satisfy $\ell=1$ for every $\ell\in\mathcal{L}_{\theta}$, because each exposed
face $\{\ell=1\}\cap C$ through $\theta$ contains $F_{\min}$; so they belong to
$F_{\theta}$. In the same way, choosing $\ell$ in the relative interior of the
convex set $\mathcal{L}_{\theta}$, one sees that $F_{\theta}=\{\theta\}$ if and
only if $\theta$ is an \emph{exposed} point of $C$. Thus \eqref{eq:star}
requires each $\theta\in\WC\cap\SG$ to be extreme in $C$, whereas
$F_{\theta}=\{\theta\}$ requires it to be exposed; the two differ, for
instance, where a flat side of $\partial C$ meets a strictly convex arc, and
there $F_{\theta}$ is a segment with endpoint $\theta$.

Condition \eqref{eq:star} is also \emph{necessary}: if it fails, the
conclusion of Proposition~\ref{prop:rigidity}\,(iii) fails for some
$G$-strictly convergent sequence. We give the argument for
$\mathscr{A}=\operatorname{curl}$ on $\V=\mathbb{R}^{N\times n}$, where $\WC$
is the rank-one cone. Let $\theta\in\WC\cap\SG$ and
$\eta\in\mathscr{M}^{+}_{1}(F_{\theta})$ with $\overline\eta=\theta$ and
$\eta\neq\delta_{\theta}$. Write $\theta=a\otimes\xi$ with $|\xi|=1$, let
$\Omega$ be a bounded Lipschitz domain with
$\mathcal{H}^{n-1}(\Omega\cap\xi^{\perp})>0$, and set
$u:=a\,\boldsymbol{1}_{\{x\cdot\xi>0\}}$. Then
$Du=\theta\,\mathcal{H}^{n-1}\llcorner(\Omega\cap\xi^{\perp})$ is purely
singular and, since $G^{\infty}(\theta)=1$,
$|Du|_{G}=\mathcal{H}^{n-1}\llcorner(\Omega\cap\xi^{\perp})$ and
$p_{G}=\theta$. In the $G$-normalisation consider the triple
$$
  \nu_{x}=\delta_{0},
  \qquad
  \lambda=|Du|_{G},
  \qquad
  \nu^{\infty}_{x}=\eta ,
$$
whose barycentre is $\overline\eta\,|Du|_{G}=Du$. We check the conditions of
\cite[Thm.~1]{KristensenRindler10}, in the form of Lemma~\ref{rmk:transfer}
and Remark~\ref{rmk:jensen-transfer}, numbered as in
\cite[Thm.~6.1]{KirchheimKristensen16}. Condition~(i), at Lebesgue points, is
trivial, as $\lambda\perp\Leb$ and $\nu_{x}=\delta_{0}$. For condition~(ii), at
singular points, if $h$ is quasiconvex then $h^{\infty}$
is positively $1$-homogeneous and rank-one convex, being the pointwise limit of
$t^{-1}h(t\,\cdot\,)$, hence convex at the rank-one matrix $\theta$ by
\cite[Cor.~1.2]{KirchheimKristensen16}; integrating a supporting linear
function at $\theta$ against $\eta$ gives
$h^{\infty}(\theta)\le\dual{h^{\infty}}{\eta}$. In other words, at singular
points condition~(ii) holds for \emph{every} probability measure
with barycentre in the wave cone; the same observation is used in
\cite[proof of Thm.~1.7]{GennaioliRindler26}. Hence, by \cite[Thm.~1]{KristensenRindler10}, $\nu$ is generated by $(Du_{j})_{j}$
for some $u_{j}\in\mathrm{BV}(\Omega,\mathbb{R}^{N})$. Since
$\dual{G^{\infty}}{\eta}=1$,
$$
  \mathcal{G}(Du_{j})\;\longrightarrow\;\ddual{G}{\nu}
  = G(0)\,|\Omega| + |Du|_{G}(\Omega) = \mathcal{G}(Du),
$$
so $Du_{j}\to Du$ $G$-strictly, whereas
$$
  |Du_{j}|(\Omega)\;\longrightarrow\;\dual{|\cdot|}{\eta}\,
  \mathcal{H}^{n-1}(\Omega\cap\xi^{\perp})
  \;>\;|\theta|\,\mathcal{H}^{n-1}(\Omega\cap\xi^{\perp}) = |Du|(\Omega).
$$
The inequality is strict because the Euclidean norm is strictly convex and
every ray meets $\SG$ exactly once, so $\eta\neq\delta_{\theta}$ is not carried
by a single ray. Hence both conclusions of
Proposition~\ref{prop:rigidity}\,(iii) fail, and strict, a fortiori
area-strict, convergence fails.

No relation between the atoms of $\eta$ is needed: the obstruction to
rigidity is not a $\WC$-compatible splitting of $\theta$ inside $F_{\theta}$,
and in particular $\operatorname{span}(F_{\theta}-\theta)\cap\WC=\{0\}$ does
not suffice for rigidity (Example~\ref{ex:transversal}). The argument extends
to constant-rank operators satisfying \textup{(A1)}, replacing the jump by the
planar wave $\theta\,\mathcal{H}^{n-1}\llcorner\xi^{\perp}$ with
$\theta\in\ker\mathscr{A}(\xi)$, \cite[Thm.~1]{KristensenRindler10} by the
characterisations of \cite{KristensenRaita22,ArroyoRabasa21}, and
\cite[Cor.~1.2]{KirchheimKristensen16} by
\cite[Thm.~1.1]{KirchheimKristensen16}. In these settings \eqref{eq:star} is
therefore necessary and sufficient for the rigidity of the concentration on
the singular part.
\end{remark}

\begin{example}[Failure of the upgrade on a flat face]\label{ex:staircase}
Let $n=2$, $\V=\mathbb{R}^{N\times2}$, $\mathscr{A}=\operatorname{curl}$, so
$\WC = \{\eta\otimes\xi\}$ is the rank-one cone, and take the (convex,
coercive) gauge
$$
  G^{\infty}(A) := |Ae_{1}| + |Ae_{2}|,
  \qquad\text{so } K_{G}=\V,\quad N_{G}=\emptyset .
$$
Let $\Omega := (0,1)\times(-1,1)$, fix $a\in\mathbb{R}^{N}$ with $|a|=1$, let
$w(t) := \lfloor t-\tfrac12\rfloor$ and $v_{j}(t) := \tfrac1j\,w(jt)$, and set
$$
  E := \{x_{2}<x_{1}\}\cap\Omega,
  \quad
  E_{j} := \{x_{2}<v_{j}(x_{1})\}\cap\Omega,
  \quad
  u_{j} := a\boldsymbol{1}_{E_{j}},\quad u := a\boldsymbol{1}_{E}.
$$
Since $x_{1}-\tfrac{3}{2j} < v_{j}(x_{1}) \le x_{1}-\tfrac{1}{2j}$, one has
$|E\triangle E_{j}| = 1/j$, so $u_{j}\to u$ in $\LL^{1}(\Omega,\mathbb{R}^{N})$ and
$Du_{j}\overset{\ast}{\rightharpoonup}Du$. The jumps of $v_{j}$ in $(0,1)$ sit at
$t=(k+\tfrac12)/j$, $k=0,\dots,j-1$, and $-\tfrac1j\le v_{j}\le 1-\tfrac1j$, so
the whole graph of $v_{j}$ lies in $\Omega$. Recall that
$D\boldsymbol{1}_{F} = \nu_{F}\,\Hm\llcorner\partial^{\ast}F$ with $\nu_{F}$ the
generalized \emph{inner} normal \cite{AFP}. Since $E_{j}$ lies
below each horizontal piece of its boundary and to the right of each vertical
one, the jump set of $u_{j}$ consists of horizontal segments of total length $1$
with inner normal $-e_{2}$ and of $j$ vertical segments of height $1/j$, of
total length $1$, with inner normal $e_{1}$. Hence $Du_{j}$ is purely singular
with $G$-polars $-a\otimes e_{2},\,a\otimes e_{1}\in\SG$, while $Du$ is carried
by the diagonal $\Gamma$, $\Hm(\Gamma)=\sqrt2$, with inner normal
$\tfrac{1}{\sqrt2}(1,-1)$ and $G$-polar
$p_{G} = \tfrac12\,a\otimes(1,-1)\in\SG$. Hence
$$
  |Du_{j}|_{G}(\Omega) = 1+1 = 2,
  \qquad
  |Du|_{G}(\Omega) = \sqrt2\cdot\sqrt2 = 2,
$$
and, since $u_{j}$ and $u$ are piecewise constant, the absolutely continuous
parts of $Du_{j}$ and $Du$ vanish identically, so that
$\mathcal{G}(Du_{j}) = G(0)\,|\Omega| + 2 = \mathcal{G}(Du)$ for \emph{every}
integrand $G$ admissible in \textup{(A3)} with this recession function. Thus
$Du_{j}\to Du$ $G$-strictly, whereas $|Du_{j}|(\Omega)=2$ for every $j$ and
$|Du|(\Omega)=\sqrt2$: strict, hence also area-strict, convergence fails. The
generated generalized Young measure is
$$
  \nu_{x}=\delta_{0},
  \qquad
  \lambda = \sqrt2\,\Hm\llcorner\Gamma = |D^{s}u|_{G},
  \qquad
  \nu^{\infty}_{x} = \tfrac12\delta_{a\otimes e_{1}}
                   + \tfrac12\delta_{-a\otimes e_{2}} ,
$$
so that $\theta_{x} = \overline{\nu}^{\infty}_{x} = p_{G}(x)\in\SG$ and $b_{G}=1$: all
the conclusions of Proposition~\ref{prop:rigidity}\,(i)--(ii) hold, and only
\eqref{eq:star} fails. Indeed $G^{\infty}$ is affine on the cone over
$[\,a\otimes e_{1},\,-a\otimes e_{2}\,]\subseteq F_{\theta_{x}}$, and
$\theta_{x}$ is the midpoint of this segment. Here the generating sequence is
explicit.
\end{example}
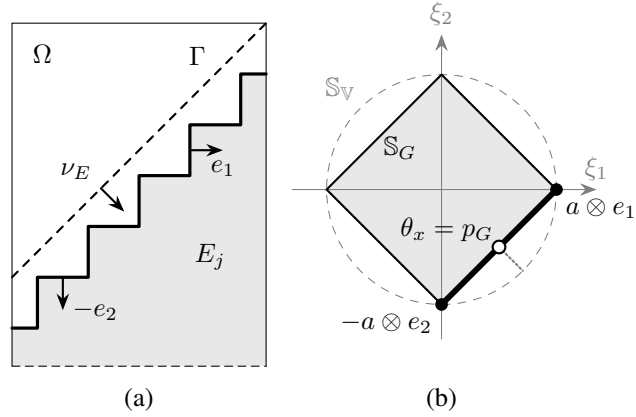
\begin{figure}[ht!]
\centering
\begin{tikzpicture}[>={Stealth[length=2.2mm]}, line cap=round, line join=round,scale=0.8]

\begin{scope}[x=4.2cm, y=4.2cm]
  \fill[gray!18]
    (0,-0.35) -- (0,-0.2) -- (0.1,-0.2) -- (0.1,0) -- (0.3,0) -- (0.3,0.2)
    -- (0.5,0.2) -- (0.5,0.4) -- (0.7,0.4) -- (0.7,0.6) -- (0.9,0.6)
    -- (0.9,0.8) -- (1,0.8) -- (1,-0.35) -- cycle;
  \draw[thin] (0,-0.35) -- (0,1) -- (1,1) -- (1,-0.35);
  \draw[thin, densely dashed] (0,-0.35) -- (1,-0.35);
  \draw[thick, dashed] (0,0) -- (1,1);
  \draw[very thick]
    (0,-0.2) -- (0.1,-0.2) -- (0.1,0) -- (0.3,0) -- (0.3,0.2) -- (0.5,0.2)
    -- (0.5,0.4) -- (0.7,0.4) -- (0.7,0.6) -- (0.9,0.6) -- (0.9,0.8) -- (1,0.8);
  \draw[->, thick] (0.2,0)     -- (0.2,-0.13) node[right] {$-e_2$};
  \draw[->, thick] (0.7,0.5)   -- (0.83,0.5)  node[below] {$e_1$};
  \draw[->, thick] (0.35,0.35) -- ++(0.092,-0.092);
  \node[above left] at (0.35,0.35) {$\nu_E$};
  \node at (0.12,0.88)                {$\Omega$};
  \node[above left] at (0.8,0.8)      {$\Gamma$};
  \node at (0.78,0.08)                {$E_j$};
  \node at (0.5,-0.48)                {(a)};
\end{scope}

\begin{scope}[shift={(7.1cm,1.45cm)}, x=1.9cm, y=1.9cm]
  \fill[gray!18] (1,0) -- (0,1) -- (-1,0) -- (0,-1) -- cycle;
  \draw[->, gray] (-1.3,0) -- (1.35,0) node[above] {$\xi_1$};
  \draw[->, gray] (0,-1.3) -- (0,1.35) node[above] {$\xi_2$};
  \draw[gray, dashed] (0,0) circle (1);
  \draw[thick] (1,0) -- (0,1) -- (-1,0) -- (0,-1) -- cycle;
  \draw[line width=2.2pt] (1,0) -- (0,-1);
  \draw[gray, densely dotted, thick] (0.5,-0.5) -- (0.7071,-0.7071);
  \fill (1,0)  circle (0.055) node[below right] {$a\otimes e_1$};
  \fill (0,-1) circle (0.055) node[below left]  {$-a\otimes e_2$};
  \draw[thick, fill=white] (0.5,-0.5) circle (0.055);
  \node[above left, inner sep=2pt] at (0.5,-0.5) {$\theta_x=p_G$};
  \node at (-0.36,0.36)         {$\SG$};
  \node[gray] at (-0.88,0.88)   {$\mathbb{S}_{\V}$};
  \node at (0,-1.826)           {(b)};
\end{scope}

\end{tikzpicture}
\caption{The staircase example for $j=5$. (a) The set $E_j$ (shaded) lies
below the graph of $v_j$; $\Gamma$ is the limit interface, and the arrows are
inner normals. (b) The section of $\SG$ (solid) and of the Euclidean sphere
$\mathbb{S}_{\V}$ (dashed) by the plane $\{a\otimes\xi:\xi\in\R^2\}$, in the
coordinates $\xi$. The concentration-angle measure
$\nu^\infty_x=\tfrac12\delta_{a\otimes e_1}+\tfrac12\delta_{-a\otimes e_2}$ is
carried by the face $F_{\theta_x}$ (thick), and its barycentre
$\theta_x=p_G$ lies on $\SG$ but inside $\mathbb{S}_{\V}$, since
$|p_G|=1/\sqrt2$ (dotted): this is the Euclidean mass lost in the limit.}
\label{fig:staircase}
\end{figure}
\begin{example}[A flat face transversal to the wave cone]\label{ex:transversal}
Let $n=N=2$, $\mathscr{A}=\operatorname{curl}$, $0<\eps<1$, and take the
norm
$$
  G^{\infty}(A) := \max\bigl\{|A_{11}-A_{22}|,\ \eps^{-1}|A_{22}|\bigr\}
                 + |A_{12}| + |A_{21}|
  \qquad (A\in\mathbb{R}^{2\times2}),
$$
which is convex, hence admissible in \textup{(A2)}, with $G:=G^{\infty}$ in
\textup{(A3)}. The rank-one matrix $\theta:=e_{1}\otimes e_{1}$ lies on $\SG$,
and near $\theta$ the first term of the maximum is the active one, so
$\mathcal{L}_{\theta}$ consists of the functionals
$\ell_{s,t}(A)=A_{11}-A_{22}+sA_{12}+tA_{21}$ with $|s|,|t|\le1$. Imposing
$\ell_{s,t}(z)=1$ for all such $s,t$ together with $G^{\infty}(z)=1$ gives
$$
  F_{\theta}=\bigl\{\,e_{1}\otimes e_{1}+\tau\,\mathrm{Id} \;:\;
               |\tau|\le\eps\,\bigr\},
$$
a segment with midpoint $\theta$, so \eqref{eq:star} fails at $\theta$.
Remark~\ref{rmk:sharp}, with $u=e_{1}\boldsymbol{1}_{\{x_{1}>0\}}$ and
$\eta=\tfrac12\delta_{\theta+\eps\mathrm{Id}}
     +\tfrac12\delta_{\theta-\eps\mathrm{Id}}$,
produces gradients converging $G$-strictly but not strictly to
$Du=e_{1}\otimes e_{1}\,\Hm\llcorner\{x_{1}=0\}$.

Two features set this apart from Example~\ref{ex:staircase}. First,
$\det(\theta\pm\eps\,\mathrm{Id})=\pm\eps(1\pm\eps)\neq0$: the
concentration-angle measure is supported entirely \emph{off} the wave cone,
although the singular polar of the limit is rank one, as it must be. Second,
the atoms differ by $2\eps\,\mathrm{Id}\notin\WC$ and
$\operatorname{span}(F_{\theta}-\theta)\cap\WC=\{0\}$: no rank-one-compatible
splitting of $\theta$ exists inside $F_{\theta}$, and rigidity fails
nonetheless. The generating sequence is provided by the characterisation
theorem and is not explicit; compare \cite[Appendix]{GennaioliRindler26},
where no explicit generating sequence is known even for
$\tfrac12\delta_{\mathrm{Id}/\sqrt2}+\tfrac12\delta_{-\mathrm{Id}/\sqrt2}$.
\end{example}

\begin{example}[Reduction to Reshetnyak's theorem]\label{ex:reshetnyak}
In the setting of Section~\ref{convergence of energy section}, let
$G(z)=G^\infty(z)=|z|$, so that $\SG=\mathbb S^{N\times n-1}$,
$G$-strict convergence is strict convergence, and
$\int_{\SG}|z|\dif\nu^\infty_x=1$. At $\Leb$-a.e.\ $x$,
\eqref{equality 1 G infty and nux}, the triangle inequality and Jensen's
inequality give
$$
|\nabla u(x)|=\int_{\M}|z|\dif\nu_x+\lambda^a
=\big|\overline{\nu}_x+\lambda^a\overline{\nu}^\infty_x\big|
\leq|\overline{\nu}_x|+\lambda^a\,|\overline{\nu}^\infty_x|
\leq\int_{\M}|z|\dif\nu_x+\lambda^a ,
$$
so that all these quantities are equal. Since the Euclidean norm is strictly
convex, the equality cases give, where $\nabla u(x)\neq0$ and with
$e(x):=\nabla u(x)/|\nabla u(x)|$:
\begin{enumerate}[label=\textup{(\alph*)}]
  \item $\overline{\nu}_x$ and $\lambda^a\overline{\nu}^\infty_x$ are
    nonnegative multiples of $e(x)$, and
    $|\nabla u(x)|=|\overline{\nu}_x|+\lambda^a(x)$;
  \item $\nu_x$ is carried by the half-line $\{t\,e(x):t\geq0\}$;
  \item if $\lambda^a(x)>0$, then $\nu^\infty_x=\delta_{e(x)}$;
\end{enumerate}
while $\nu_x=\delta_0$ and $\lambda^a(x)=0$ where $\nabla u(x)=0$. On the
singular part, every point of $\mathbb S^{N\times n-1}$ is an extreme point
of the unit ball, so \eqref{eq:star} holds, and
Proposition~\ref{prop:rigidity}\,(i) and~(iii) give $\lambda^s=|D^su|$ and
$\nu^\infty_x=\delta_{\frac{D^su}{|D^su|}(x)}$ for $|D^su|$-a.e.\ $x$.

Let now $f\colon\M\to\R$ be continuous and positively $1$-homogeneous.
Since $f$ is linear along half-lines, (a)--(c) give
$\langle f,\nu_x\rangle+\lambda^a\langle f,\nu^\infty_x\rangle
=f(e(x))\big(|\overline{\nu}_x|+\lambda^a(x)\big)=f(\nabla u(x))$ for
$\Leb$-a.e.\ $x$, and testing the generation of $\nu$ with $f$ yields
$$
\lim_{j\to\infty}\int_\Omega f(\nabla u_j)\dif x
=\int_\Omega f(\nabla u)\dif x
+\int_\Omega f\Big(\frac{D^su}{|D^su|}\Big)\dif|D^su| ,
$$
which is Reshetnyak's continuity theorem. The Young measure $\nu$ itself
need not be elementary. For $n=N=1$, $\Omega=(0,1)$ and $u(x)=x$, let
$u_j(x)=\int_0^x g_j\dif t$ with either
$$
g_j=j\,\boldsymbol{1}_{E_j},\quad
E_j=\bigcup_{k=0}^{j-1}\Big[\frac kj,\frac kj+\frac1{j^2}\Big),
\qquad\text{or}\qquad
g_j=2\,\boldsymbol{1}_{F_j},\quad
F_j=\bigcup_{k=0}^{j-1}\Big[\frac kj,\frac kj+\frac1{2j}\Big).
$$
In both cases $u_j\to u$ strictly, and the generated Young measures are
$(\delta_0,\Lebo\Lcorner\Omega,\delta_1)$, with diffuse concentration, and
$(\tfrac12\delta_0+\tfrac12\delta_2,0,\cdot\,)$, with oscillation. Neither
convergence is area-strict: $\langle Du_j\rangle(\Omega)$ tends to $2$ and to
$\tfrac{1+\sqrt5}{2}$ respectively, while $\langle Du\rangle(\Omega)=\sqrt2$.
Area-strict convergence, by contrast, forces the generated Young measure to
be elementary, see \cite{KristensenRindler10,rindlerbook}.
\end{example}
The next proposition reflects \cite[Thm.~2.5 (i)]{alibertbouchitte} in our compactification. Since the differential constraint plays no role, we state it for arbitrary generating sequences, $\mathscr{A}$-free or not.
\begin{proposition}\label{C closed set and supp(nu_x^infty) subset C_G cap S_G}
Let $C \subset \V$ be a closed set and let $(v_j) \subset \mathrm{L}^1(\Omega,\V)$ generate the Young measure $v_j \toY \nu=(\nu_x,\lambda, \nu_x^\infty)$, with compactification $\mathscr{M}_1^+(\mathbb{S}_G)$. Suppose
$$
v_j(x) \in C \quad \Leb \emph{ a.e. }x \in \Omega, \ \emph{for all }j.
$$
Then $\emph{supp}(\nu_x) \subset C$ for $\Leb$ a.e. $x \in \Omega$ and, setting
$$
\widetilde C_G=\overline{ \left\{ \frac{z}{1+G^\infty(z)} : z \in C\right\} } \subset \overline{\B},
$$
we have
$$
\emph{supp}(\nu_x^\infty) \subset \widetilde C_G \cap \mathbb S_G \quad \lambda \emph{ a.e. }x \in \Omega.
$$
\end{proposition}
\begin{proof}
\emph{Step 1.} Take $h_k \in \CC(\V)$ such that
$$
h_k(z)=1 \text{ on } C, \quad h_k(z) \searrow 0 \text{ on } \V \setminus C,
$$
for example
$$
h_k(z)=(1-k\,\text{dist}(z,C))^+.
$$
Note that $h_k^\infty\equiv 0$. Testing the convergence with $\phi \in \CC_0(\Omega)$, we get
$$
\int_\Omega \phi(x)h_k(v_j(x))\dif x \underset{j \to \infty}{\longrightarrow} \int_\Omega \phi(x) \langle h_k, \nu_x \rangle \dif x.
$$
Since $v_j(x) \in C$ for $\Leb$ a.e. $x$, the left-hand side is equal to $\int_\Omega \phi(x) \dif x$ for all $j$, thus
$$
\langle h_k, \nu_x \rangle=1
$$
for $\Leb$ a.e. $x$. Taking $k \nearrow \infty$, we find
$$
\nu_x(C)=1 \quad \Leb \text{ a.e.},
$$
meaning that $\text{supp}(\nu_x)\subset C$ for $\Leb$ a.e. $x$. \\
\emph{Step 2.} We test the convergence with $\phi \otimes T^{-1}f_k$:
\begin{multline*}
\int_\Omega \phi(x) f_k\left(\frac{v_j(x)}{1+G^\infty(v_j(x))}\right)(1+G^\infty(v_j(x)))\dif x = \int_\Omega \phi(x)T^{-1}f_k(v_j(x)) \dif x \\
\underset{j \to \infty}{\longrightarrow} \nu(\phi \otimes T^{-1}f_k)=\int_\Omega \phi(x) \langle T^{-1}f_k, \nu_x \rangle \dif x + \int_{\Omegacl}\phi(x) \langle (T^{-1}f_k)^\infty, \nu_x^\infty \rangle \dif \lambda,
\end{multline*}
where $f_k \in \CC(\overline{\B})$ is such that
$$
f_k(\hat z)=1 \text{ on }\widetilde C_G, \quad f_k(\hat z) \searrow 0 \text{ on } \overline{\B} \setminus \widetilde C_G
$$
as $k \to \infty$; this approximation from above is possible because $\widetilde C_G$ is closed. Since $v_j(x) \in C$ for $\Leb$ a.e. $x$, we have
$$
f_k\left(\frac{v_j(x)}{1+G^\infty(v_j(x))}\right)=1
$$
for every $j$ and $k$. Thus the left-hand side is also equal to
\begin{multline*}
\int_\Omega \phi(x) (1+G^\infty(v_j(x)))\dif x \underset{j \to \infty}{\longrightarrow} \int_\Omega \phi(x) \langle 1+G^\infty , \nu_x \rangle \dif x + \int_{\Omegacl} \phi(x) \langle \underbrace{G^\infty}_{=1 \text{ on } \mathbb S_G}, \nu_x^\infty \rangle \dif \lambda \\
=\int_\Omega \phi(x) \langle 1+G^\infty , \nu_x \rangle \dif x + \int_{\Omegacl} \phi(x) \dif \lambda.
\end{multline*}
Since $\text{supp}(\nu_x)\subset C$ for $\Leb$ a.e. $x$, we have
\begin{multline*}
\langle T^{-1}f_k, \nu_x \rangle =\int_{\V} f_k \left(\frac{z}{1+G^\infty(z)} \right)(1+G^\infty(z))\dif \nu_x(z) \\
= \int_C f_k \left(\frac{z}{1+G^\infty(z)} \right)(1+G^\infty(z))\dif \nu_x(z) = \int_C (1+G^\infty(z))\dif \nu_x(z) = \langle 1+G^\infty , \nu_x \rangle,
\end{multline*}
since $z \in C$ implies $\dfrac{z}{1+G^\infty(z)} \in \widetilde C_G$. Therefore, equating the two limits, we are left with
$$
\int_{\Omegacl}\phi(x) \langle (T^{-1}f_k)^\infty, \nu_x^\infty \rangle \dif \lambda = \int_{\Omegacl} \phi(x) \dif \lambda.
$$
Now, for $z \in \mathbb S_G$ we have $G^\infty(tz)=t$ by 1-homogeneity, so the continuity of $f_k$ gives
$$
(T^{-1}f_k)^\infty(z)= \lim_{t \to \infty} f_k \left(\frac{tz}{1+t} \right)\left(\frac{1}{t} + 1\right) = f_k(z).
$$
Hence
$$
\int_{\Omegacl} \phi(x) \int_{\mathbb S_G} f_k(z) \dif \nu_x^\infty(z) \dif \lambda =\int_{\Omegacl} \phi(x) \dif \lambda,
$$
so that $\langle f_k, \nu_x^\infty \rangle=1$ for $\lambda$ a.e. $x\in\Omega$. Since $f_k \searrow \mathbf{1}_{\widetilde C_G}$, taking $k \nearrow \infty$ we get
$$
\nu_x^\infty(\mathbb S_G \cap \widetilde C_G)=1 \quad \lambda \text{ a.e. } x \in \Omega.
$$
Finally, we note that $\widetilde C_G \cap \mathbb S_G = C_\infty \cap \mathbb S_G$, where $C_\infty=\{\lim_n t_n w_n : t_n \searrow 0,\ w_n \in C\}$ is the asymptotic cone of $C$, since $\frac{w}{1+G^\infty(w)}-\frac{w}{G^\infty(w)} \to 0$ as $|w| \to \infty$.
\end{proof}

\section{Transfer principle}\label{transfer principle section}
This section is devoted to a \emph{transfer} principle, which is the
linear-growth analogue of the Evans--Gariepy splitting, as described in Remark \ref{evans gariepy connection} below.
\begin{proposition}\label{transfer principle proposition}
Let $G: \M \to \R$ be a quasiconvex function of linear growth with regular recession function, so that
$$
G^\infty(x)= \lim_{t \to + \infty} \frac{G(tx)}{t}
$$
exists locally uniformly in $x$. Let $F: \M \to \R$ be continuous of linear growth, with generalized recession function (\cite{fonsecamueller-bv, Muller92})
$$
F^{\#}(x):= \limsup_{\substack{x' \to x \\  t \to +\infty}} \frac{F(tx')}{t}
$$
and we assume that $F-G$ is quasiconvex. \\
Additionally, we assume the nondegeneracy condition
$$
F^\# >0 \quad \text{on } \M \setminus \{0\}.
$$ 
Let $\Omega \subset \R^n$ be a bounded open set, with no regularity assumption imposed on $\partial \Omega$. Let $g \in \bv(\R^n, \R^N)$ with $|D^sg|(\partial \Omega)=0$ (for instance $g \in W^{1,1}(\R^n,\R^N)$), $u_j \in \bv_g(\Omega,\R^N)$ for each $j \in \mathbb{N}$, meaning that $u_j \in \bv(\Omega,\R^N)$ and $u_j=g$ on $\partial \Omega$ in the sense of rough traces (see Proposition \ref{Prop bv_0} below for its definition). Assume $u_j \overset{\ast}{\rightharpoonup} u$ in $\bv(\Omega, \R^N)$ for some $u \in \bv(\Omega, \R^N)$, where $u$ is not necessarily in $\bv_g(\Omega,\R^N)$. We denote for $F$
$$
\int_\Omega F(Du):= \int_\Omega F(\nabla u)\dd x +  \int_\Omega F^\# \left( \frac{D^su}{|D^su|} \right) |D^su|,
$$
and similarly for $G$
$$
\int_\Omega G(Du):= \int_\Omega G(\nabla u)\dd x +  \int_\Omega G^\infty \left( \frac{D^su}{|D^su|} \right) |D^su|.
$$
Then, in this set of assumptions,
$$
\lim_{j\to\infty}\int_\Omega F(Du_j)=\int_\Omega F(Du)
\qquad\Longrightarrow\qquad
\lim_{j\to\infty}\int_\Omega G(Du_j)=\int_\Omega G(Du).
$$
Moreover, the extension of $u-g$ by zero outside $\Omega$ has derivative $D(u-g)\Lcorner\Omega$; in particular, the limit $u$ does not detach from $g$ at $\partial \Omega$.
\end{proposition}
Before proving the transfer principle, we discuss the possibility of loss of boundary datum $g$ for the sequence $u_j$ tending to $u$.
To this aim, first consider a bounded Lipschitz domain $\Omega'$ compactly containing $\Omega$, $\Omega \Subset \Omega'$, and extend both $u_j, u$ to $\Omega' \setminus \Omega$ equal to $g$: 
$$
\tilde u_j= \begin{cases}
   u_j \quad \text{in }\Omega,
   \\
   g \quad \, \, \, \text{in }\Omega' \setminus \Omega,
\end{cases} \qquad \tilde u= \begin{cases}
   u \quad \quad \quad \text{in }\Omega,
   \\
   g \quad \, \, \, \text{in }\Omega' \setminus \Omega.
\end{cases}
$$
Note that, by Proposition \ref{Prop bv_0}, $u_j \in \bv_g(\Omega,\R^N)$ implies that the function
$$
\tilde u_j-g= \begin{cases}
    u_j-g \quad \text{in }\Omega,
   \\
   0 \quad \, \, \, \text{in }\Omega' \setminus \overline{\Omega},
\end{cases} 
$$
is in $\bv(\Omega',\R^N)$ and $|D(\tilde u_j-g)|(\partial \Omega)=0$ (the converse fails in general, e.g.\ on domains with slits), but not necessarily for $u$ (see Example \ref{example u limit not in bv_g}).  This property is not preserved under weak$^*$ convergence, since mass may concentrate on $\partial \Omega$. Note also that $\partial \Omega$ might not have zero Lebesgue measure. If $\partial \Omega$ is Lipschitz, this definition of the Dirichlet class coincides with the one using traces.
\begin{proposition}\label{Prop bv_0}
We say $u \in \bv_0(\Omega)$ if there exists a sequence $(\phi_j) \subset \CC_c^\infty(\Omega)$ converging strictly to $u$, which means 
$$
\phi_j \overset{\LL^1(\Omega)}{\longrightarrow}  u, \quad D\phi_j \wstar Du, \quad \text{and}\quad  |D\phi_j|(\Omega)=\int_{\Omega} |\nabla \phi_j| \dd x \to |Du|(\Omega).  
$$
Then the function
$$
v:= \begin{cases}
    u \quad \text{in }\Omega,
   \\
   0 \quad \, \, \, \text{in }\R^n\setminus \Omega,
\end{cases} 
$$
is in $\bv(\R^n)$ and
$$
Dv=Du \Lcorner \Omega.
$$
In particular $|Dv|(\partial \Omega)=0$.
\end{proposition}
The proof of the Proposition above is postponed in the Appendix, Section \ref{appendix section}. With this definition, for $u_j$ to be equal to $g$ on $\partial\Omega$ in the sense of rough traces means precisely $u_j -g \in \bv_0(\Omega,\R^N)$.

\hfill \break
As already mentioned, the weak$^*$ limit $u$ is not necessarily equal to $g$ at the boundary, as the next example shows.
\begin{example}\label{example u limit not in bv_g}
Let in $\bv(0,1)$ the sequence $u_j(x)=\boldsymbol{1}_{(1/j,1-1/j)}(x)$. Then $u_j \to 1$ in $\LL^1(0,1)$ and $Du_j= \delta_{1/j}-\delta_{1-1/j} \wstar 0$, so $u_j \wstar 1$ in $\bv(0,1)$. For all $j$, close to the boundary of $(0,1)$ $u_j$ is equal to $g\equiv 0 \in \bv(\R)$, whereas $u=\boldsymbol{1}_{(0,1)}$ has trace equal to $1$ at the endpoints of its domain.       
\end{example}
\begin{remo}
Let $\Omega$ be open and bounded with infinite perimeter, and let (in the scalar case for simplicity) $u=1$ in $\Omega$ and $g=0$ in $\Omega'$, $\Omega \Subset \Omega'$. Then the extension $u$ is not in $\bv(\Omega')$ because $|Du|(\partial \Omega)=\operatorname{Per}( \Omega)= + \infty$.  Any $u_j= \boldsymbol{1}_{\Omega_j}$ where $\Omega_j \Subset \Omega$ is an increasing sequence of smooth relatively compact sets with $\Omega_j \nearrow \Omega$, is in $\bv_g(\Omega)$ and converges only w$^*$-$\bv_{\mathrm{loc}}(\Omega)$ to $u=\boldsymbol{1}_\Omega$, and not w$^*$-$\bv(\Omega)$, since the sequence of perimeters $\operatorname{Per}( \Omega_j)$ is unbounded by lower semicontinuity of the perimeter (so it converges in duality with compactly supported functions but not with their closure of functions vanishing at the boundary). The extension of $u$ to $\Omega'$ is not in $\bv(\Omega')$, however this does not contradict the assumptions because the convergence is not in w$^*$-$\bv(\Omega)$ but only local.
\end{remo}
Applying Proposition \ref{Prop bv_0} with $g \in \bv(\Omega',\R^N)$ in place of $0$, so to $u_j-g \in \bv_0(\Omega,\R^N)$, we get
$$
v= \begin{cases}
    u_j-g \quad \text{in }\Omega,
   \\
   0 \quad \, \, \, \text{in }\Omega' \setminus \Omega,
\end{cases} 
$$
with $v \in \bv(\Omega',\R^N)$ and $Dv=D(u_j-g) \Lcorner \Omega $. Therefore the extension $\tilde u_j = v+g$ satisfies 
\begin{equation}\label{Duj decomposition}
D \tilde u_j = Dv+Dg =D(u_j-g) \Lcorner \Omega + Dg =Du_j \Lcorner \Omega + Dg \Lcorner \Omega' \setminus\Omega.
\end{equation}
Instead, for $\tilde u$ (which belongs to $\bv(\Omega',\R^N)$, as shown below) the same decomposition does not hold.  Setting $\mu:=D\tilde u\Lcorner\partial \Omega-Dg\Lcorner\partial \Omega$, a finite Radon measure supported in $\partial \Omega$, we have
\begin{equation}\label{Du decomposition}
D \tilde u = Du \Lcorner \Omega + Dg \Lcorner \Omega' \setminus \Omega + \mu.   
\end{equation}
In the case of $\Omega$ Lipschitz, and $\partial^* \Omega$ being its reduced boundary, $|\mu| \ll \Ha \Lcorner \partial^* \Omega$. \\
The weak$^*$ convergence is preserved in the bigger set $\Omega'$, indeed $\tilde u_j$ trivially converges to $\tilde u$ in $\LL^1(\Omega',\R^N)$ since it does in $\Omega$ and out of it is identically $g$, and the derivatives have uniformly bounded mass  since by \eqref{Duj decomposition} 
\begin{align*}
|D\tilde u_j|(\Omega')=|D u_j|(\Omega) + |D g|(\Omega' \setminus \Omega)
\end{align*}
which is uniformly bounded in $j \in \N$. Thus every subsequence of $D\tilde u_j$ has a further subsequence weakly$^*$ converging as measures to some $D v$, with $v= \tilde u$ by closedness with the $\LL^1$ convergence. 

Therefore we can assume $\tilde u \in \bv(\Omega', \R^N)$ and $\tilde u_j \overset{\ast}{\rightharpoonup} \tilde u$ in $\bv(\Omega', \R^N)$.
Since  $\tilde u_j=\tilde u=g$ on $\Omega'\setminus\overline{\Omega}$, which contains a neighbourhood of $\partial\Omega'$ in $\Omega'$, all these functions have the same trace on the Lipschitz boundary $\partial\Omega'$, we have the following lower semicontinuity result for \emph{signed} quasiconvex functions of linear growth with fixed boundary values (see \cite{kristensenrindler-relax}) :
$$
\int_{\Omega'} G(D\tilde u) \leq \liminf_{j \to + \infty} \int_{\Omega'}G(D\tilde u_j),
$$
and since also $F-G$ is quasiconvex
\begin{equation}\label{F-G qc lsc}
\int_{\Omega'} (F-G)(D\tilde u) \leq \liminf_{j \to + \infty} \int_{\Omega'}(F-G)(D\tilde u_j).
\end{equation}
\begin{remo}\label{remark F(Duj)}
For $u_j \in \bv_g(\Omega)$ it is clear by (\ref{Duj decomposition}) that in the bigger set $\Omega'$ the functional with integrand $F$ can be written as
\begin{multline*}
\int_{\Omega'} F(D \tilde u_j)= \int_\Omega F(Du_j) + \int_{\Omega'\setminus \Omega} F(Dg) =\\
= \int_\Omega F(\nabla u_j) \dd x + \int_\Omega F^{\#} \left(\frac{D^s u_j}{|D^s u_j|}\right)\dd |D^s u_j|+  \int_{\Omega'\setminus \Omega} F(\nabla g) \dd x + \int_{\Omega' \setminus \Omega} F^{\#} \left(\frac{D^s g}{|D^s g|}\right)\dd |D^s g|,
\end{multline*}
whereas for $u$ and its extension $\tilde u$ by (\ref{Du decomposition}) we only have that
$$
\int_{\Omega'} F(D \tilde u)= \int_\Omega F(Du) + \int_{\Omega'\setminus \Omega} F(Dg)+\int_{\partial \Omega} F(\mu).
$$
Here, since $|D^sg|(\partial \Omega)=0$, $\int_{\partial \Omega} F(\mu)$ denotes $\int_{\partial \Omega} F(D\tilde u)-\int_{\partial \Omega} F(\nabla g)\dd x$, the energy carried by $\mu$, so that the identity above holds by definition. The goal is to find a way to represent $F(\mu)$.
\end{remo}

By the equality of the limit energy for $F$ and the extension of both $u_j, u$ to $g$ in $\Omega' \setminus \Omega$, we can add to both sides of the limit energy for $F$ the quantity $\int_{\Omega' \setminus \Omega} F(Dg)$ and we get 
$$
   \lim_{j \to +\infty }\int_{\Omega'} F(D\tilde u_j)=   \lim_{j \to +\infty }\int_{\Omega} F(Du_j)+ \int_{\Omega' \setminus \Omega} F(Dg) = \int_{\Omega} F(Du) + \int_{\Omega' \setminus \Omega} F(Dg),
$$
where on the left-hand side the first equality holds by Remark \ref{remark F(Duj)},  since $u_j-g\in\bv_0(\Omega,\R^N)$ for every $j$. The right-hand side is not exactly equal to $\int_{\Omega'} F(D\tilde u)$ again by Remark \ref{remark F(Duj)}, since  $u-g$ need not belong to $\bv_0(\Omega,\R^N)$, but
$$
 \int_{\Omega} F(Du) + \int_{\Omega' \setminus \Omega} F(Dg)=  \int_{\Omega'} F(D\tilde u) - \int_{\partial \Omega} F(\mu),
$$
and if we apply lower semicontinuity to $F(D\tilde u_j)$ in $\Omega'$ ($F$ is quasiconvex because $G$ and $F-G$ are) we get 
$$
 \lim_{j \to +\infty }\int_{\Omega'} F(D\tilde u_j) = \int_{\Omega'} F(D\tilde u) - \int_{\partial \Omega} F(\mu) \leq \liminf_{j \to +\infty }\int_{\Omega'} F(D\tilde u_j) - \int_{\partial \Omega} F(\mu),
$$
i.e. 
\begin{equation}\label{F(mu) leq 0}
 \int_{\partial \Omega} F(\mu) \leq 0.
\end{equation}
\begin{remo}
If $\Omega$ is  Lipschitz, then $(\ref{F(mu) leq 0})$ says that
$$
\int_{\partial \Omega} F^{\#}((g-u) \otimes \nu_\Omega)\, \dd \Ha \leq 0.
$$
and by the assumption of $F^{\#}>0$ away from $0$, $g-u=0$ on $\partial \Omega$. \\
We want to prove that also when $\Omega$ is not necessarily  Lipschitz, then we have the representation
$$
\int_{\partial \Omega} F(\mu) = \int_{\partial \Omega} F^{\#}(\mu),
$$
and this happens  because $\mu$ is the barycenter of the part of the concentration of $(Du_j)$ that escapes to $\partial \Omega$, and it is singular with respect to $\Leb$ (Proposition \ref{Prop mu singular}); the latter is not automatic for a concentration, which may have an absolutely continuous part when $\Leb(\partial \Omega)>0$.  Since $F^{\#}>0$ away from 0, the condition \eqref{F(mu) leq 0}  then leads to $\mu=0$.

\end{remo}
\begin{proposition}\label{Prop mu singular}
The measure $\mu$ is  the barycenter of the concentration of $(Du_j)$ at $\partial \Omega$: if $D u_j \toY \nu$ in $\Omega$ along a subsequence, then $\mu=\overline{\nu}^\infty_x \lambda \Lcorner \partial \Omega$. Moreover, $\mu$ is singular with respect to $\Leb$, so that $\int_{\partial \Omega} F(\mu)=\int_{\partial \Omega} F^{\#}\left(\frac{\mu}{|\mu|}\right)\dd |\mu|$.
\end{proposition}
\begin{proof}
Take any countable collection of functions $\phi \in \CC(\overline{\Omega'})$, $f \in \text{LIP}(\M) $ which uniquely determine the Young measures $\textbf{Y}(\Omega', \M)$.
In $\Omega$ we have, along a subsequence, $D u_j \toY \nu$ for some YM, i.e.
\begin{equation}\label{Duj toY nu}
\int_{\Omega} \phi(x)f(D u_j) \to \int_{\Omega} \phi(x) \langle f, \nu_x \rangle \dif x + \int_{\overline{\Omega}} \phi(x) \langle f^\infty, \nu_x^\infty \rangle \dif \lambda(x).
\end{equation}
Also in $\Omega'$, $D \tilde u_j \toY \tilde \nu$, 
$$
\int_{\Omega'} \phi(x)f(D \tilde u_j)  \to \int_{\Omega'} \phi(x) \langle f, \tilde \nu_x \rangle \dif x + \int_{\overline{\Omega'}} \phi(x) \langle f^\infty, \tilde \nu_x^\infty \rangle \dif \tilde \lambda(x).
$$
But the left-hand side, by  \eqref{Duj decomposition}, is split in
$$
\int_{\Omega'} \phi(x)f(D \tilde u_j)  = \int_{\Omega} \phi(x)f(D u_j) + \int_{\Omega' \setminus \Omega} \phi(x)f(D g) 
$$
and we know that each term singularly converges, the first term to (\ref{Duj toY nu}) whereas the second is constantly equal to
$$
\int_{\Omega' \setminus \Omega} \phi(x)f(D g)= \int_{\Omega' \setminus \Omega} \phi(x)f(\nabla g)\dif x +  \int_{\Omega' \setminus \Omega} \phi(x)f^\infty\left(\frac{D^sg}{|D^sg|}\right)\dif |D^sg|(x),
$$
thus giving an expression for $D \tilde u_j \toY \tilde \nu$,
\begin{multline*}
\int_{\Omega'} \phi(x) \langle f, \tilde \nu_x \rangle \dif x + \int_{\overline{\Omega'}} \phi(x) \langle f^\infty, \tilde \nu_x^\infty \rangle \dif \tilde \lambda(x) \\ = \int_{\Omega} \phi(x) \langle f, \nu_x \rangle \dif x + \int_{\overline{\Omega}} \phi(x) \langle f^\infty, \nu_x^\infty \rangle \dif \lambda(x) + \int_{\Omega' \setminus \Omega} \phi(x)f(D g), \\
 = \int_{\Omega} \phi(x) \langle f, \nu_x \rangle \dif x + \int_{\overline{\Omega}} \phi(x) \langle f^\infty, \nu_x^\infty \rangle \dif \lambda(x) + \int_{\Omega' \setminus \Omega} \phi(x)f(\nabla g)\dif x \\ +  \int_{\Omega' \setminus \Omega} \phi(x)f^\infty\left(\frac{D^sg}{|D^sg|}\right)\dif |D^sg|(x),
\end{multline*}
which gives
\begin{gather*}
\tilde \nu_x=\begin{cases}
    \nu_x \quad \text{if }x \in \Omega, \\
    \delta_{\nabla g(x)}\quad \text{if }x \in \Omega'\setminus\Omega,
\end{cases}    \\
\tilde \nu^\infty_x=\begin{cases}
    \nu^\infty_x \quad \text{if }x \in \Omega, \\
    \mu_x^\infty \quad \text{if }x \in \partial \Omega,
    \\
    \delta_{\frac{D^sg}{|D^sg|}(x)}\quad \text{if }x \in \Omega'\setminus\overline{\Omega},
\end{cases} \quad
 \tilde \lambda= \begin{cases}
     \lambda \quad \text{if }x \in \Omega, \\
    \lambda + |D^sg| \quad \text{if }x \in \partial \Omega,
     \\
    |D^sg|\quad \text{if }x \in \Omega'\setminus\overline{\Omega}.
\end{cases}
\end{gather*}
with 
$$
\mu_x^\infty= \begin{cases}
   \frac{a(x)}{a(x)+1}\nu_x^\infty + \frac{1}{a(x)+1}\delta_{\frac{D^sg}{|D^sg|}(x)} \quad \text{for } |D^sg|\text{-almost every }x \in \partial \Omega, \\
   \nu_x^\infty \quad \text{for } \lambda^*\text{-almost every }x \in \partial \Omega,
\end{cases}
$$
where $a(x)$ comes from the Lebesgue--Radon--Nikod\'ym decomposition of $\lambda$ with respect to $|D^sg|$, $\lambda=a|D^sg|+\lambda^* $. The exact expression is not relevant but can be found using Proposition 6 (\emph{Shifts}) in \cite{Kri Rindler}. Using that $D\tilde u= \overline{\tilde \nu} \Lcorner \Omega'$, and
$$
D\tilde u \Lcorner \partial \Omega = Dg \Lcorner \partial \Omega + \mu= \nabla g \Leb \Lcorner \partial \Omega + D^sg \Lcorner \partial \Omega + \mu
$$
$$
\overline{\tilde \nu} \Lcorner \partial \Omega = \overline{\tilde \nu}_x \Leb \Lcorner \partial \Omega +  \overline{\tilde \nu}_x^\infty \tilde \lambda \Lcorner \partial \Omega = \nabla g \Leb \Lcorner \partial \Omega +  \overline{\tilde \nu}_x^\infty \tilde \lambda \Lcorner \partial \Omega
$$
we see by identification (note that the Lebesgue part $\nabla g \Leb \Lcorner \partial \Omega$ cancels out regardless of its Lebesgue measure) that the third component $\mu$ of $D\tilde u$ is 
$$
\mu= \overline{\tilde \nu}_x^\infty \tilde \lambda \Lcorner \partial \Omega - D^sg \Lcorner \partial \Omega
$$
 and, since on $\partial \Omega$ the concentration part of $\tilde \nu$ is the sum of that of $\nu$ and that of $Dg$, we have $\overline{\tilde \nu}_x^\infty \tilde \lambda \Lcorner \partial \Omega=\overline{\nu}^\infty_x \lambda \Lcorner \partial \Omega+D^sg \Lcorner \partial \Omega$, that is,
$$
\mu=\overline{\nu}^\infty_x \lambda \Lcorner \partial \Omega.
$$
It remains to show that $\mu \perp \Leb$. This is immediate if $\Leb(\partial \Omega)=0$, but not in general, since a concentration may have an absolutely continuous part. By the Calder\'on--Zygmund theorem (see \cite[Theorem 3.83]{AFP}), $\tilde u$ and $g$ are approximately differentiable $\Leb$-almost everywhere, with approximate differentials equal to the densities of $(D\tilde u)^a$ and $(Dg)^a$; since $\tilde u=g$ on $\Omega' \setminus \Omega \supset \partial \Omega$, they coincide $\Leb$-almost everywhere on $\partial \Omega$. Hence $(D\tilde u)^a \Lcorner \partial \Omega=(Dg)^a \Lcorner \partial \Omega$, and $\mu$ has no absolutely continuous part. In terms of the Young measure, if $\lambda=\lambda^a+\lambda^s$ is the Lebesgue decomposition of $\lambda$, this means that $\overline{\nu}^\infty_x=0$ for $\lambda^a$-almost every $x\in\partial \Omega$: the concentration may well have an absolutely continuous part on $\partial \Omega$, but there its barycenter vanishes, and $\mu=\overline{\nu}^\infty_x \lambda^s \Lcorner \partial \Omega$. Finally, as $|D^sg|(\partial \Omega)=0$, we have $D\tilde u \Lcorner \partial \Omega=\nabla g \Leb \Lcorner \partial \Omega+\mu$ with $\mu \perp \Leb$, so $\mu$ is the singular part of $D\tilde u$ on $\partial \Omega$ and
$$
\int_{\partial \Omega} F(\mu)= \int_{\partial \Omega} F(D\tilde u)-\int_{\partial \Omega} F(\nabla g) \dd x=\int_{\partial \Omega} F^{\#}\left(\frac{\mu}{|\mu|}\right)\dd |\mu|.
$$
\end{proof}
By Proposition \ref{Prop mu singular}, \eqref{F(mu) leq 0} reads
$$
\int_{\partial \Omega} F^{\#}\left(\frac{\mu}{|\mu|}\right)\dd |\mu| \leq 0,
$$
and since $F^{\#}>0$ on $\M \setminus \{0\}$, we conclude $\mu=0$.

\hfill \break
\hfill \break
Getting back to the lower semicontinuity inequality (\ref{F-G qc lsc}), and recalling that $\int (F-G)(D\,\cdot)=\int F(D\,\cdot)-\int G(D\,\cdot)$ because $(F-G)^{\#}=F^{\#}-G^\infty$, the recession function of $G$ being a limit, it holds
\begin{multline*}
\int_{\Omega'} F(D\tilde u)- \int_{\Omega'} G(D\tilde u) \leq \liminf_{j \to + \infty} \int_{\Omega'}(F-G)(D\tilde u_j) = \lim_{j \to + \infty} \int_{\Omega'}F(D\tilde u_j) - \limsup_{j \to + \infty} \int_{\Omega'}G(D\tilde u_j) \\ = \int_{\Omega'} F(D\tilde u) - \int_{\partial \Omega} F(\mu) - \limsup_{j \to + \infty} \int_{\Omega'}G(D\tilde u_j) 
\end{multline*}
thus
\begin{equation}\label{limsup G}
\limsup_{j \to + \infty} \int_{\Omega'}G(D\tilde u_j)  \leq \int_{\Omega'} G(D\tilde u) - \int_{\partial \Omega} F(\mu). 
\end{equation}
It would be enough to have the condition $- \int_{\partial \Omega} F(\mu) \leq 0$, in order to get
$$
\int_{\Omega'} G(D\tilde u) \leq \liminf_{j \to + \infty} \int_{\Omega'}G(D\tilde u_j) \leq \limsup_{j \to + \infty} \int_{\Omega'}G(D\tilde u_j) \leq \int_{\Omega'} G(D\tilde u).
$$ 
i.e.
\begin{equation}\label{energy equality for G}
\lim_{j \to \infty} \int_{\Omega'}G(D\tilde u_j)= \int_{\Omega'}G(D\tilde u).
\end{equation}
Since we proved the stronger condition $\mu=0$ (which is also the last assertion of Proposition \ref{transfer principle proposition}, because $D\big((u-g)\boldsymbol{1}_\Omega\big)=D(u-g)\Lcorner\Omega+\mu$), which in fact gives $\int_{\partial \Omega} F(\mu)=0$, then on one side
$$
\lim_{j \to \infty} \int_{\Omega'}G(D\tilde u_j)= \lim_{j \to \infty} \int_{\Omega}G(D u_j) + \int_{\Omega' \setminus \Omega} G(Dg),
$$
and on the other side 
$$
\int_{\Omega'}G(D\tilde u)= \int_{\Omega}G(D u) + \int_{\Omega' \setminus \Omega} G(Dg)
$$
therefore \eqref{energy equality for G}.

\begin{remo}\label{remark regular recession F}
If, in addition, $F$ has a regular recession function $F^\infty$, the assumption $|D^sg|(\partial \Omega)=0$ can be dropped. Let $D u_j \toY \nu$ in $\Omega$ along a subsequence. Then
$$
\lim_{j\to\infty}\int_\Omega F(Du_j)=\int_\Omega \langle F,\nu_x\rangle \dd x
+\int_{\Omega}\langle F^\infty,\nu_x^\infty\rangle \dd\lambda
+\int_{\partial \Omega}\langle F^\infty,\nu_x^\infty\rangle \dd\lambda .
$$
By the Jensen inequalities for gradient Young measures (see \cite{Kri Rindler}), the first two terms are bounded below by $\int_\Omega F(Du)$. The zero extensions of $u_j-g$ generate in $\Omega'$ a Young measure with the same concentration on $\partial \Omega$, since $Dg\Lcorner\Omega$ does not concentrate at $\partial \Omega$, and with barycenter $\mu$ there, while $\partial \Omega$ now lies inside $\Omega'$. Hence the singular Jensen inequality bounds the last term below by $\int_{\partial \Omega}F^\infty\big(\frac{\mu}{|\mu|}\big)\dd|\mu|$, the absolutely continuous part of $\lambda$ on $\partial \Omega$ contributing a nonnegative amount. The hypothesis then gives directly $\int_{\partial \Omega}F^\infty\big(\frac{\mu}{|\mu|}\big)\dd|\mu|\le 0$, hence $\mu=0$, and the rest of the proof applies to any $g\in\bv(\R^n,\R^N)$, since then $D\tilde u\Lcorner\partial \Omega=Dg\Lcorner\partial \Omega$.
\end{remo}

\begin{remo}[Connection to Evans--Gariepy article]\label{evans gariepy connection} The assumptions on $F,G$ are reminiscent of those made in \cite{evans gariepy some remarks}; we briefly summarize the result the authors find.
Suppose $F$ has the form
$$
F(P) = G(P) + \lambda |P|^q \qquad (P \in \M)
$$
where $\lambda > 0$ and $1 < q < \infty$, and
$$
G \text{ is quasiconvex}, \qquad 0 \leq G(P) \leq \Lambda (1 + |P|^q)
$$
for some constant $\Lambda$. If
$$
\begin{cases}
u_k \rightharpoonup u \quad \text{weakly in } \WW^{1,q}(\Omega, \mathbb{R}^N), \\
I[u_k] \to I[u],
\end{cases}
$$
where $I[u]= \int_{\Omega} F(Du) \dif x$, then
$$
u_k \to u \quad \text{strongly in } \WW^{1,q}(\Omega, \mathbb{R}^N).
$$
To see this, note simply that by the Morrey--Acerbi--Fusco lower semicontinuity theorem
\begin{align*}
I[u]-\lambda \int_{\Omega}|Du|^q &= \int_\Omega G(Du) \dif x \\
&\leq \liminf_{k \to \infty} \int_\Omega G(Du_k) \dif x \\
&= \liminf_{k \to \infty} \left( I[u_k] - \lambda \int_\Omega |Du_k|^q \dif x \right)
\end{align*}
and thus since $I[u_k] \to I[u]$,
$$
\limsup_{k \to \infty} \int_\Omega |Du_k|^q \dif x \leq \int_\Omega |Du|^q \dif x.
$$
This estimate together with $\int_\Omega |Du|^q \dif x \leq \liminf_{k \to \infty} \int_\Omega |Du_k|^q \dif x$ (given by the weak convergence) and the uniform convexity of $\LL^q$ spaces with $1 < q < \infty$, yields the strong convergence. \\
A more general hypothesis is the following. We call $F$ \emph{uniformly strictly quasiconvex} if for some positive $\gamma>0$ it holds
$$
\int_Q F(A+D\phi)\dif x \geq F(A) + \gamma \int_Q |D\phi|^q \dif x, 
$$
for all $A \in \R^{n \times N}$ and $\phi \in \WW_0^{1,q}(Q,\R^N)$. The case $\gamma=0$ corresponds to being quasiconvex.
\end{remo}

\section{Appendix}\label{appendix section}
\subsection{Elementary Young measures} We prove here that the usual representation of Young measures through the triple $\nu=(\nu_x,\lambda,\nu_x^\infty)$, by an application of the Disintegration theorem, works the same with our new compactification by $G^\infty$. 
\begin{proof}[Proof of Lemma \ref{YM measure theoretic definition}]
Assume \emph{(1),...,(4)} and the formula below. Let $\Phi(x,z)= \phi(x) \boldsymbol{1}_\V(z)$, so that $\Phi^\infty=0$. By the formula of $\nu(\Phi)$ above and $\nu_x(\V)=1$ for $\Leb$-a.e. $x$, it holds
$$
\nu(\Phi)= \int_{ \Omega} \phi(x) \dif \nu_x(z) \dif x + 0 = \int_{\Omega} \phi(x) \dif x,
$$
and on the other hand by the characterization (\ref{expression of ell}) (which makes use of $(T^*)^{-1}$) it is also equal to 
$$
\nu(\Phi)= \int_{\overline{\Omega}\times \V} \Phi \dif \mu^0 +0 = \int_{\overline{\Omega} \times \B} \big(1-G^\infty(\hat{z})\big) \phi(x) \dif \mu(x,\hat z) 
$$
and therefore (we can integrate on $\overline{\B}$ since on $\mathbb{S}_G$ it holds $\big(1-G^\infty(\hat{z})\big)=0$)
$$
\int_{\overline{\Omega \times \B}} \big(1-G^\infty(\hat{z})\big) \phi(x) \dif \mu(x,\hat z) = \int_{\Omega} \phi(x) \dif x.
$$
\\
Now assume that $(T^*)^{-1}\nu$ belongs to the set (\ref{T*-1 of YM}). Then by the Disintegration Theorem, there exists a measure $\kappa \in \mathscr{M}^+(\overline{\Omega})$ and a weakly$^*$ $\kappa$-measurable family of probability measures $(\eta_x)_{x \in \Omega} \subset \mathscr{M}_1^+(\overline{\B})$ such that 
$$
\mu= \kappa (\dif x)\otimes \eta_x.
$$
Let also $\kappa=\kappa^a \Leb \Lcorner \Omega + \kappa^s$ the Lebesgue--Radon--Nikod\'ym decomposition of $\kappa$, with $\kappa^a \in \LL^1(\Omega)$. Then rewriting (\ref{T*-1 of YM}) in terms of the decomposition of $\mu$
\begin{multline*}
\int_{\Omega} \phi(x) \left(\kappa^a(x) \cdot \int_{\overline{\B}} \big(1-G^\infty(\hat{z})\big) \dif \eta_x(\hat z) -1\right)\dif x \\ + \int_{\overline{\Omega}} \phi(x) \left(\int_{\overline{\B}} \big(1-G^\infty(\hat{z})\big) \dif \eta_x(\hat z)\right)\dif \kappa^s(x)=0.
\end{multline*}
By the uniqueness of the Radon--Nikod\'ym decomposition, and playing over all $\phi$, we get 
$$
\int_{\overline{\B}} \big(1-G^\infty(\hat{z})\big) \dif \eta_x(\hat z)=0
$$
for $\kappa^s$-almost every $x \in \overline{\Omega}$, whereby $\eta_x$ is concentrated on $\mathbb{S}_G$ for $\kappa^s$-almost every $x \in \overline{\Omega}$ (since $1-G^\infty(\hat{z})>0$ on $\B$). Moreover looking at the first term
$$
\kappa^a(x)= \dfrac{1}{\int_{\overline{\B}} \big(1-G^\infty(\hat{z})\big)\dif \eta_x(\hat z)} \quad \Leb\text{-a.e. }x \in \Omega.
$$
We now find the triple $((\nu_x)_{x \in \Omega}, \lambda, (\nu_x^\infty)_{x \in \overline{\Omega}})$. For $\Phi(x,z)=\phi(x)f(z)$, $\phi \in \CC(\overline{\Omega}), f \in \HH$, and $Tf(\hat z)= \big(1-G^\infty(\hat{z})\big) f\left(\frac{\hat{z}}{1- G^\infty(\hat{z}) }\right) $, by $\mu=\kappa (\dif x) \otimes \eta_x$ and $\kappa=\kappa^a \Leb + \kappa^s$, it holds 
\begin{multline*}
\nu(\Phi)= \int_{\overline{\Omega \times \B}} T\Phi \, \dif \mu = \int_{\overline{\Omega}}\phi(x) \int_{\overline{\B}} Tf(\hat z) \dif \eta_x(\hat z) \dif \kappa(x) 
= \int_{\Omega}\phi(x) \left( \kappa^a(x) \int_{\B} Tf(\hat z) \dif \eta_x(\hat z) \right) \dif x \\ + \int_{\overline{\Omega}}\phi(x) \underbrace{\int_{\B} Tf(\hat z)\dif \eta_x(\hat z)}_{=0 \text{ for }\kappa^s \text{-a.e. }x} \dif \kappa^s(x)   + \int_{\overline{\Omega}}\phi(x) \left(  \int_{\mathbb{S}_G } f^\infty(\hat z) \dif \eta_x(\hat z) \right) \dif \kappa(x)= \\ =  \int_{\Omega}\phi(x) \left( \kappa^a(x) \int_{\B} Tf(\hat z) \dif \eta_x(\hat z) \right) \dif x+ \int_{\overline{\Omega}}\phi(x) \left(  \int_{\mathbb{S}_G } f^\infty(\hat z) \dif \eta_x(\hat z) \right) \dif \kappa(x) = \\=  \int_{\Omega}\phi(x) \left( \kappa^a(x) \int_{\B} Tf(\hat z) \dif \eta_x(\hat z) \right) \dif x+ \int_{\overline{\Omega}}\phi(x) \left(  \frac{1}{\eta_x(\mathbb{S}_G)}\int_{\mathbb{S}_G } f^\infty(\hat z) \dif \eta_x(\hat z) \right) \eta_x(\mathbb{S}_G) \dif \kappa(x)
\end{multline*}
It makes sense to define the measures $\nu_x \in \mathscr{M}^+(\V)$, $x \in \Omega$ as
$$
\langle \nu_x,f \rangle := \kappa^a(x)\int_{\B} Tf(\hat z) \dif \eta_x(\hat z),
$$
the measures $\nu_x^\infty \in \mathscr{M}^+_1(\mathbb{S}_G)$, $x \in \overline{\Omega}$ as
$$
\langle \nu_x^\infty,f^\infty \rangle := \frac{1}{\eta_x(\mathbb{S}_G)}\int_{\mathbb{S}_G } f^\infty(\hat z) \dif \eta_x(\hat z)
$$
and 
$$
\lambda(\dif x) := \eta_x(\mathbb{S}_G) \kappa=  \eta_x(\mathbb{S}_G) [\kappa^a \Leb \Lcorner \Omega + \kappa^s ] \in \mathscr{M}^+(\overline{\Omega}).
$$
With these definitions, the last line of equalities of $\nu(\Phi)$ can be written as
$$
\nu(\Phi)= \int_\Omega \phi(x) \int_\V f(z) \dif \nu_x(z) \dif x + \int_{\overline{\Omega}} \phi(x) \int_{\mathbb{S}_G} f^\infty(z) \dif \nu_x^{\infty}(z) \dif \lambda(x).
$$
It is clear that, the way they are defined, $\nu_x^\infty$ are probability measures being defined as averages on $\mathbb{S}_G= \{ G^\infty=1\}$, the set on which they are supported for $\lambda$-a.e. $x\in \overline{\Omega}$. These are called the \textbf{concentration-angle} measures, whereas $\lambda$ is called the \textbf{concentration} measure. \\
We prove now that also $\nu_x$, called \textbf{oscillation} measures, are probability measures, so $\nu_x \in \mathscr{M}_1^+(\V)$. Indeed by taking $f\equiv 1$ on $\V$, 
$$
\langle \nu_x,1 \rangle := \kappa^a(x)\int_{\B}\big(1-G^\infty(\hat{z})\big) \dif \eta_x(\hat z)=1 
$$
for $\Leb$-a.e. $x \in \Omega$ since as already found 
$$
\kappa^a(x)= \dfrac{1}{\int_{\overline{\B}} \big(1-G^\infty(\hat{z})\big)\dif \eta_x(\hat z)}= \dfrac{1}{\int_{\B} \big(1-G^\infty(\hat{z})\big)\dif \eta_x(\hat z)} \quad \Leb\text{-a.e. }x \in \Omega.
$$
Finally we prove the moment condition \emph{(4)}. Using $\Phi(x,z)= 1+ G^\infty(z) \in \E_G(\Omega,\V)$, we have $\Phi^\infty(x,z)=G^\infty(z)=1$ on $\Omegacl \times \mathbb S_G$ and so
$$
\int_\Omega \int_\V (1+G^\infty(z))\dif \nu_x(z) \dif x + \lambda(\overline{\Omega}) = \langle \mu, T\Phi \rangle = \mu(\overline{\Omega \times \B}) < \infty,
$$
since $T\Phi(x,\hat z) = \big(1-G^\infty(\hat{z})\big)\left( 1+ \frac{G^\infty(\hat{z})}{1- G^\infty(\hat{z}) } \right) \equiv 1$.
\end{proof}

\begin{lemma}
A Young measure $\nu$ generated by a sequence of elementary Young measures, i.e.  $(v_j)$ is a sequence in $ \mathscr{M}(\Omega,\V)$ and $\boldsymbol{\eps}_{v_j} \wstar \nu$ in $\E_G(\Omega,\V)^*$, cannot have oscillation at the boundary $\partial \Omega$.
\end{lemma}
\begin{corollary}
All Young measures defined in Lemma \ref{YM measure theoretic definition} have this property. 
\end{corollary}
\begin{proof}[Proof of Corollary] One needs to prove that the YM generated by elementary sequences coincide with $(T^*)^{-1}$ of the right-hand side of (\ref{T*-1 of YM}). This result is essentially found in \cite{kristensen raita}.       
\end{proof}
\begin{proof}[Proof of Lemma] We recall that the elementary Young measures $\boldsymbol{\eps}_{v_j}$ act on $\phi \in \CC(\Omegacl)$, $f \in \HH $ as
$$
\langle \boldsymbol{\eps}_{v_j}, \phi \otimes f \rangle = \int_{\Omega}\phi(x)f(v_j^a(x))\dd x + \int_{\Omega}\phi(x) f^\infty \left(\frac{\dd v_j^s}{\dd |v_j^s|}(x)\right) \dd |v_j^s|(x),
$$
and by assumption on the generating sequence
$$
\langle \boldsymbol{\eps}_{v_j}, \phi \otimes f \rangle \to \langle \nu, \phi \otimes f \rangle = \int_{\Omegacl \times \overline{\B}} \phi \otimes f \dif \nu= \int_{ \Omegacl} \phi(x) \langle f, \nu_x \rangle \dif x + \int_{\overline \Omega} \phi(x) \langle f^\infty , \nu_x^\infty \rangle \dif \lambda (x).
$$
In particular taking $f\equiv 1 \in \HH$, so that $f^\infty=0$, it holds
$$
\langle \boldsymbol{\eps}_{v_j}, \phi \otimes 1 \rangle = \int_{\Omega}\phi(x)\dd x  
$$
but also
$$
\langle \boldsymbol{\eps}_{v_j}, \phi \otimes 1 \rangle \to \int_{ \Omegacl} \phi(x) \langle 1, \nu_x \rangle \dif x= \int_{\Omegacl} \phi(x) \nu_x(\V) \dif x.
$$
Therefore 
$$
\int_{\Omegacl} \phi(x) \nu_x(\V) \dif x =\int_\Omega \phi(x) \dif x, \quad \forall \phi \in \CC(\Omegacl).
$$
The measures $\nu_x(\V)\Leb \Lcorner \Omegacl$ and $\Leb\Lcorner\Omega$ are equal for all $\phi \in \CC(\Omegacl)$, so they coincide, and therefore it must hold $\nu_x(\V)=0$ for almost every $x \in \partial \Omega$. Alternatively one takes a sequence $(\phi_k)$ with $\phi_k\equiv 1$ on $\partial \Omega$ and $\phi_k \searrow 0$ in $\Omega$, so that taking the limit at both sides it yields $\int_{\partial \Omega} \nu_x(\V) \dif x= 0$ and then the same conclusion holds.

\end{proof}
\begin{remo}
We found two expressions for the Young measure $\nu$. \\
The first one is defined for $\Phi \in \E_G(\Omega,\V)$ as
    
    $$
    \nu(\Phi)= \int_{\overline{\Omega}\times \V} \Phi \, \dif \mu^0 + \int_{\overline{\Omega}\times \mathbb{S}_G} \Phi^{\infty} \, \dif \mu^\infty,
    $$
with $\mu=(T^*)^{-1}\nu$ and
\begin{align*}
    &\mu^0= (\mathrm{id}\times\iota_G^{-1})_{\#}\left[ \big(1-G^\infty(\hat{z})\big) \mu  \Lcorner \overline{\Omega} \times \B \right] \in \mathscr{M}(\overline{\Omega} \times \V), \\
    &\mu^\infty =\mu  \Lcorner \overline{\Omega} \times \mathbb{S}_G \in \mathscr{M}(\overline{\Omega}\times \mathbb{S}_G),
\end{align*}
The second one for $\Phi \in \E_G(\Omega,\V)$ 
$$
\nu(\Phi)= \int_{ \Omega}\int_\V \Phi(x,z)\dif \nu_x(z) \dif x + \int_{\overline \Omega}\int_{\mathbb{S}_G} \Phi^\infty(x,z)\dif \nu^\infty_x(z) \dif \lambda (x).
$$
This means by identification
$$
\mu^0=\Leb \Lcorner \Omega \otimes \nu_x 
$$
and
$$
\mu^\infty= \lambda \otimes \nu_x^\infty.
$$
The connection with the homogeneous setting - cfr. Section \ref{hahn banach section} where we have $(\nu^0,\nu^\infty) \in \mathscr{M}_1^+(\V) \times \mathscr{M}^+(\mathbb{S}_G)$ -  can be done by taking any reference $\Omega$ of Lebesgue measure $1$, for example $\XX=(-\frac{1}{2},\frac{1}{2})^n$, and no $x$-dependence, so $\nu_x=\nu^0$, $\nu_x^\infty=\nu^\infty$, and therefore
$$
(\mu^0,\mu^\infty) =(\Leb(\Omega)\nu^0, \lambda(\overline{\Omega})\nu^\infty) = (\nu^0, \lambda(\overline{\Omega})\nu^\infty)\in \mathscr{M}_1^+(\V) \times \mathscr{M}^+(\mathbb{S}_G)
$$
since in general $\lambda(\overline{\Omega})\neq 1$.
\end{remo}

\subsection{Kantorovich norm}\label{kantorovich norm subsection}
Let $(X,\text{d}_X)$ be a compact metric space. The dual of the Banach space $\CC(X)$ with the supremum norm is the space of signed bounded Radon measure $\mathscr{M}(X)$, normed by the total variation. The subspace $\text{LIP}(X)$ of all Lipschitz functions $\Phi:X \to \R$ is a Banach space under the norm
$$
\|\Phi\|_{\text{LIP}}= \sup_{x\in X} |\Phi(x)|+\text{lip}(\Phi), \quad \text{lip}(\Phi)= \sup_{x,y \in X, x \neq y} \frac{|\Phi(x)-\Phi(y)|}{\text{d}_X(x,y)}.
$$
The Kantorovich norm of a signed bounded Radon measure $\mu$ on $X$ is given by
$$
\|\mu\|_{\text{K}}=\sup\{ \langle \mu,\Phi \rangle : \Phi \in \text{LIP}(X), \|\Phi\|_{\text{LIP}}\leq 1\}.
$$
As a subspace of $\CC(X)^*$, the space of positive Radon measures $\mathscr{M}^+(X)$ also inherits the weak$^\ast$ topology:
$$
\left\{ O \cap \mathscr{M}^+(X): O \in \sigma(\CC(X)^*,\CC(X)) \right\}.
$$
\begin{lemma}
On $\mathscr{M}^+(X)$ the relative weak$^\ast$ topology is exactly the topology determined by the Kantorovich metric $\emph{d}_{\emph{K}}$.  
\end{lemma}
Therefore the distance between two positive measures in $\mathscr{M}^+(X)$ amounts to
$$
\|\mu-\nu\|_{\text{K}} = \sup_{\Phi \in \text{LIP}(X), \|\Phi\|_{\text{LIP}}\leq 1 } \left|\int_X \Phi \dif \mu - \int_X \Phi \dif \nu \right|.
$$
These results are used for the compact metric space $X=\overline{\BG}$, with the Euclidean distance. The next remark clarifies a relation for the integrands of this space.
\begin{remo}\label{lip remark}
By definition, the isomorphism $T$ in \eqref{definition of T homogeneous setting} sends each $\Phi \in \HH$ to $T\Phi \in  \CC(\overline{\B})$, and $\Phi=(1+G^\infty)\,T\Phi\circ\iota_G$. With Lipschitz constants on $\V$ and on $\overline{\B}\subset\V$ computed with respect to the Euclidean distance, the two functions satisfy
$$
\text{lip}(\Phi)\leq \text{lip}(G^\infty)\max|T\Phi|+\Big(1+\frac{\text{lip}(G^\infty)}{c}\Big)\text{lip}(T\Phi)\leq L\,\|T\Phi\|_{\text{LIP}},
$$
where $c$ is the constant in \textup{(A2)} and $L$ is defined in \eqref{lip constant L}; for $G^\infty=|\cdot|$ this reads $\text{lip}(\Phi)\leq \max|T\Phi|+2\,\text{lip}(T\Phi)\leq 2\|T\Phi\|_{\text{LIP}}$. Indeed, let $x,y\in\V$; by symmetry we may assume that $G^\infty(y)\leq G^\infty(x)$. Then
\begin{multline*}
|\Phi(x)-\Phi(y)|=\Big|\big(G^\infty(x)-G^\infty(y)\big)\,T\Phi\big(\iota_G(x)\big)+\big(1+G^\infty(y)\big)\Big(T\Phi\big(\iota_G(x)\big)-T\Phi\big(\iota_G(y)\big)\Big)\Big| \\
\leq \text{lip}(G^\infty)\max|T\Phi|\,|x-y| + \big(1+G^\infty(y)\big)\,\text{lip}(T\Phi)\,\big|\iota_G(x)-\iota_G(y)\big|,
\end{multline*}
and we are left to prove that
$$
\big(1+G^\infty(y)\big)\big|\iota_G(x)-\iota_G(y)\big|\leq \Big(1+\frac{\text{lip}(G^\infty)}{c}\Big)|x-y|.
$$
Since $x\big(1+G^\infty(y)\big)-y\big(1+G^\infty(x)\big)=(x-y)\big(1+G^\infty(y)\big)+y\big(G^\infty(y)-G^\infty(x)\big)$, we have
\begin{multline*}
\big(1+G^\infty(y)\big)\big|\iota_G(x)-\iota_G(y)\big|=\frac{\big|x\big(1+G^\infty(y)\big)-y\big(1+G^\infty(x)\big)\big|}{1+G^\infty(x)}\\
\leq \frac{1+G^\infty(y)+\text{lip}(G^\infty)|y|}{1+G^\infty(x)}\,|x-y|,
\end{multline*}
and the claim follows from $G^\infty(y)\leq G^\infty(x)$ and $|y|\leq c^{-1}G^\infty(y)\leq c^{-1}\big(1+G^\infty(x)\big)$.
\end{remo}
\subsection{Extension of \texorpdfstring{$\bv_0$}{bv0} functions}
\begin{proof}[Proof of Proposition \ref{Prop bv_0}]
Find countably many relatively compact open sets $\Omega_i, i \in \N$, so that $\Omega_i \Subset \Omega_{i+1}$, $\Omega= \bigcup_i \Omega_i$, with smooth boundary (always possible by mollification, thus they have zero Lebesgue measure) and so that 
$$
|Du|\left(\bigcup_i \partial \Omega_i\right)=0,
$$
which is possible by $|Du|(\Omega)< \infty$. Take any test function $\eta \in \CC_c^{\infty}(\R^n,\R^n)$, then by $\bv$ definition
$$
\langle Dv, \eta \rangle= -\int_{\R^n} v  \di \eta \, \dd x,
$$
but $v=u$ in $\Omega$ and $0$ outside so
$$
\langle Dv, \eta \rangle= -\int_{\R^n} v  \di \eta \, \dd x = -\int_{\Omega} u\di \eta \, \dd x = -\lim_j \int_{\Omega} \phi_j \di \eta \, \dd x,
$$
here we used the $\LL^1(\Omega)$ convergence only. We still need to use the strict convergence $\emph{and}$ the compact support of $\phi_j$ inside $\Omega$ (otherwise, they could be just smooth and the strict convergence approximation holds likewise). First of all, for each $j \in \N$ eventually there exists some $i=i(j) \in \N$ for which the smooth set $\Omega_{i}$ previously constructed contains $\text{supp}(\phi_j)$ for all $i \geq i(j)$. Set $\Omega_{i(j)}$ and there we can apply the smooth Gauss-Green theorem: 
\begin{align*}
-\lim_j \int_{\Omega} \phi_j \di \eta \, \dd x &= -\lim_j \int_{\Omega_{i(j)}} \phi_j \di \eta \, \dd x \\ &= \lim_j \left( \int_{\Omega_{i(j)}} \nabla \phi_j \cdot \eta \, \dd x - \int_{\partial \Omega_{i(j)}} \phi_j \, \eta \, \dd \Ha \right) \\
&= \lim_j \int_{\Omega_{i(j)}} \nabla \phi_j \cdot \eta \, \dd x = \lim_j \int_{\Omega} \nabla \phi_j \cdot \eta \, \dd x .
\end{align*}
Now, in general measure theory holds for $\mu_j,\mu \in \mathscr{M}(\Omega,\R^n)$
\begin{equation}\label{general measure theory upgrade convergence}
\begin{cases}
    \mu_j \wstar \mu \\ 
    |\mu_j|(\Omega) \to |\mu|(\Omega)
\end{cases} \implies \mu_j \to \mu \quad w^*\text{-}\CC_b(\Omega,\R^n)
\end{equation}
meaning in duality with continuous bounded functions in $\Omega$, see below for a proof. Applying this to $D\phi_j=\nabla \phi_j  \Leb, Du$ and $\eta \in \CC_c^{\infty}(\R^n,\R^n) $, with $\eta|_\Omega \in \CC_b(\Omega,\R^n)$ 
$$
\lim_j \int_{\Omega} \nabla \phi_j \cdot \eta \, \dd x  = \int_{\Omega} \eta \cdot \dd Du  = \int \eta \cdot \dd (Du \Lcorner \Omega).
$$
So
$$
\langle Dv, \eta \rangle = \int \eta \cdot \dd (Du \Lcorner \Omega).
$$
This says also that $v \in \bv(\R^n)$, with $Dv= Du \Lcorner \Omega$. \\
\emph{Proof of} \eqref{general measure theory upgrade convergence}: First fix $i\in \N$ so large that 
$$
|Du|(\Omega \setminus \Omega_i)\leq \frac{1}{i}.
$$
Find $\psi_i \in \CC_c(\Omega)$, $\psi_i \equiv 1 $ in $\overline{\Omega}_i$, $0 \leq \psi_i \leq 1$. Then by triangle inequality
\begin{multline*}
\left|\int_{\Omega} \nabla \phi_j \cdot \eta \, \dd x - \int_{\Omega} \eta \cdot \dd Du  \right| \leq \left|\int_{\Omega} \nabla \phi_j \cdot \eta\psi_i \, \dd x - \int_{\Omega} \eta \psi_i \cdot \dd Du  \right| \\+ \left|\int_{\Omega} \nabla \phi_j \cdot \eta \, \dd x - \int_{\Omega} \nabla \phi_j \cdot \eta\psi_i \, \dd x   \right|+ \left| \int_{\Omega} \eta \psi_i \cdot \dd Du - \int_{\Omega} \eta \cdot \dd Du \right|= I_j + I'_j + I''.
\end{multline*}
Since $\eta \psi_i \in \CC_c(\Omega)$ for each $i \in \N$, by $\D\phi_j \wstar Du$ in $\Omega$ 
$$
\lim_{j \to \infty} I_j = 0.
$$
For the second term $I_j'$, by the choice of $\psi_i$
\begin{multline*}
\limsup_{j\to \infty}I'_j=    \limsup_{j\to \infty} \left| \int_\Omega \nabla \phi_j \cdot \eta(1-\psi_i) \, \dd x \right| \leq \norm{\eta}_\infty   \limsup_{j\to \infty} |D\phi_j|(\Omega\setminus\Omega_i)\\ = \norm{\eta}_\infty  \limsup_{j\to \infty} \big(|D\phi_j|(\Omega) - |D\phi_j|(\Omega_i)\big)=  \norm{\eta}_\infty \big(|Du|(\Omega) - \liminf_{j \to \infty}|D\phi_j|(\Omega_i) \big) \\
\leq \norm{\eta}_\infty \big(|Du|(\Omega) - |Du|(\Omega_i) \big) = \norm{\eta}_\infty |Du|(\Omega \setminus \Omega_i),
\end{multline*}
by the lower semicontinuity of the total variation on the open set $\Omega_i$, given by $\D\phi_j \wstar Du$. \\ The third term $I''$ is constant in $j$ and
$$
I''= \left| \int_{\Omega} \eta(1-\psi_i)\cdot Du \right| \leq \norm{\eta}_\infty |Du|(\Omega \setminus\Omega_i). 
$$
Finally by the previous three estimates
$$
\limsup_{j \to \infty} \left|\int_{\Omega} \nabla \phi_j \cdot \eta \, \dd x - \int_{\Omega} \eta \cdot \dd Du  \right| \leq 2\norm{\eta}_\infty |Du|(\Omega \setminus\Omega_i) \leq 2\norm{\eta}_\infty \, \frac{1}{i},
$$
and let $i \to \infty$.
\end{proof}
\section*{Acknowledgements}
The author wishes to thank his supervisor, Jan Kristensen, who suggested the problem considered here together with the main ideas for its solution, and who provided guidance throughout this work. This paper forms part of the author’s doctoral thesis.
\section*{Use of AI tools}
A large language model was used in two limited ways in the preparation of
this paper. First, it suggested the definition of the dyadic martingale
applied to Young measures, namely the conditional expectations of
Section~\ref{inhomogenization section}; it had no part in the estimates of
that section or in their proofs, which are the author's own. The use of a martingale argument for Young measures was a long-standing idea of the author's supervisor. Second, it was
used to help with the exposition of Sections~\ref{wave cone section}
and~\ref{G strict rigidity principle}, whose core mathematical ideas are all
already present in Section~\ref{convergence of energy section}. Apart from
that definition, all the ideas and the mathematics of the paper predate the
use of AI tools; the ideas were suggested by the author's supervisor, Jan
Kristensen, as part of the author's doctoral project. All of the main results, arguments and techniques presented here in were developed solely by the author. The author takes full
responsibility for the content of the paper. 
\section*{Statements and Declarations}

\subsection*{Funding}
Open access funding provided by Scuola Internazionale Superiore di Studi
Avanzati -- SISSA within the CRUI-CARE Agreement.

\subsection*{Competing interests}
The author has no relevant financial or non-financial interests to disclose.

\subsection*{Data availability}
Data sharing not applicable to this article as no datasets were generated or
analysed during the current study.

\end{document}